\documentclass{article}
\usepackage{amssymb}
\usepackage{amsmath}
\usepackage{arxiv}
\usepackage[utf8]{inputenc}
\usepackage[T1]{fontenc}
\usepackage{url}
\usepackage{booktabs}
\usepackage{amsfonts}
\usepackage{nicefrac}
\usepackage{microtype}
\usepackage{color}
\usepackage{lipsum}
\usepackage{amsthm}
\usepackage{indentfirst}
\usepackage{graphicx}
\usepackage{epstopdf}
\usepackage{fourier}
\usepackage{bm}
\usepackage{mathtools}
\usepackage{latexsym,enumerate}
\usepackage{multicol}
\usepackage[skip=2pt]{caption}
\usepackage[font=small,skip=0pt]{subcaption}
\usepackage{array}
\usepackage{ascii}
\usepackage{algorithm,algorithmic}
\usepackage{mathtools}
\usepackage{arydshln}

\usepackage[pdfstartview=FitH, CJKbookmarks=true,
bookmarksnumbered=true, bookmarksopen=true,
colorlinks,
linkcolor=green,
anchorcolor=blue, citecolor=blue
]{hyperref}

\newcommand{\comment}[1]{}

\newcommand{\BEA}{\begin{eqnarray}}
\newcommand{\EEA}{\end{eqnarray}}

\newcommand{\BR}{\mathbb{R}}

\newtheorem{thm}{Theorem}[section]
\newtheorem{prop}[thm]{Proposition}
\newtheorem{example}[thm]{Example}
\newtheorem{lem}[thm]{Lemma}
\newtheorem{cor}[thm]{Corollary}

\newtheorem{rem}[thm]{Remark}

\newcommand{\PreserveBackslash}[1]{\let\temp=\\#1\let\\=\temp}
\newcolumntype{C}[1]{>{\PreserveBackslash\centering}p{#1}}
\newcolumntype{R}[1]{>{\PreserveBackslash\raggedleft}p{#1}}
\newcolumntype{L}[1]{>{\PreserveBackslash\raggedright}p{#1}}

\newcommand{\stkout}[1]{\ifmmode\text{\sout{\ensuremath{#1}}}\else\sout{#1}\fi}

\begin{document}

\title{Generalized Finite Difference Method on unknown manifolds}
\author{
Shixiao Willing Jiang \\
Institute of Mathematical Sciences, ShanghaiTech University, Shanghai
201210, China\\
\texttt{jiangshx@shanghaitech.edu.cn}
\And Rongji Li \\
School of Information Science and Technology, ShanghaiTech University, Shanghai
201210, China\\
\texttt{lirj2022@shanghaitech.edu.cn}
\And Qile Yan \\
Department of Mathematics, The Pennsylvania State University, University
Park, PA 16802, USA\\
\texttt{qzy42@psu.edu}
\And
John Harlim \\
Department of Mathematics, Department of Meteorology and Atmospheric
Science, \\
Institute for Computational and Data Sciences \\
The Pennsylvania State University, University Park, PA 16802, USA\\
\texttt{jharlim@psu.edu}
}
\date{\today}
\maketitle

\begin{abstract}
In this paper, we extend the Generalized Finite Difference Method (GFDM) on unknown compact submanifolds of the Euclidean domain, identified by randomly sampled data that (almost surely) lie on the interior of the manifolds. Theoretically, we formalize GFDM by exploiting a representation of smooth functions on the manifolds with Taylor's expansions of polynomials defined on the tangent bundles. We illustrate the approach by approximating the Laplace-Beltrami operator, where a stable approximation is achieved by a combination of Generalized Moving Least-Squares algorithm and novel linear programming that relaxes the diagonal-dominant constraint for the estimator to allow for a feasible solution even when higher-order polynomials are employed. We establish the theoretical convergence of GFDM in solving Poisson PDEs and numerically demonstrate the accuracy on simple smooth manifolds of low and moderate high co-dimensions as well as unknown 2D surfaces. For the Dirichlet Poisson problem where no data points on the boundaries are available, we employ GFDM with the volume-constraint approach that imposes the boundary conditions on data points close to the boundary. When the location of the boundary is unknown, we introduce a novel technique to detect points close to the boundary without needing to estimate the distance of the sampled data points to the boundary. We demonstrate the effectiveness of the volume-constraint employed by imposing the boundary conditions on the data points detected by this new technique compared to imposing the boundary conditions on all points within a certain distance from the boundary, where the latter is sensitive to the choice of truncation distance and require the knowledge of the boundary location.
\end{abstract}


\lhead{}

\newpage

\section{Introduction}

An important scientific problem that commands a wide range of applications in the natural sciences and practical engineering is to solving Partial Differential Equations (PDEs) on manifolds. In image processing, applications include segmentation of images \cite{tian2009segmentation}, image inpainting \cite{shi2017weighted} and restoration of damaged patterns \cite{bertalmio2001navier,macdonald2010implicit}. In computer graphics, applications include flow field visualization \cite{bertalmio2001variational}, surface reconstruction \cite{zhao2001fast} and brain imaging \cite{memoli2004implicit}. In physics and biology, applications include granular flow \cite{rauter2018finite}, phase formation \cite{dogel2005two} and phase ordering \cite{schoenborn1999kinetics} on surfaces, Liquid crystals on deformable surfaces \cite{nitschke2020liquid} and phase separation of biomembranes \cite{elliott2010modeling}.

In this paper, we will focus on mesh-free PDE solvers that is convenient for the setting when the manifold is only identified by randomly sampled data; that is, when triangular mesh is not available as such parameterization can be difficult to obtain when the available point cloud is randomly sampled and when the manifolds are high-dimensional and/or embedded in a high-dimensional ambient space. The setup considered in this paper assumes that for Dirichlet problems, data points on the boundary are (almost surely) not available since the boundary is a volume-measure zero set.

Several mesh-free approaches have been developed to address this issue. Among the collocation method, they include the global Radial Basis Function methods \cite{piret2012orthogonal,fuselier2013high}, the RBF-generated finite difference (FD) methods \cite{shankar2015radial,lehto2017radial}, graph-based approaches \cite{li2016convergent,li2017point,gh2019,jiang2020ghost,yan2022ghost}, Generalized Moving Least-Squares (GMLS) \cite{liang2013solving,gross2020meshfree}, and Generalized Finite Difference Method (GFDM) \cite{suchde2019meshfree}. It is worth to note that GFDM also uses GMLS for approximating the function and its derivatives, whereas the GMLS method in \cite{liang2013solving,gross2020meshfree} employs the least-squares to estimate the metric tensor and the function (and its derivatives).  In Table~\ref{tab_comp}, we list several characteristics of these methods that are relevant to practical applications. The first consideration is whether these methods can approximate general differential operators. On this aspect, all of these methods can do so except the graph-based approach which is limited to estimating the Laplacian type operator. Second, in terms of parameter robustness, RBF and graph-based methods depend crucially on the choice of the shape parameter in the radial basis function that is being used. Third, all of these methods construct a sparse approximation except for the global RBF. The fourth consideration is the stability of the Laplacian estimator, which is important for solving Poisson problems. It is well-known that RBF is an unstable interpolator (see Chapter~12 of \cite{Wendland2005Scat}). On the other hand, graph-Laplacian is a stable estimator as it induces a strictly diagonally dominant matrix \cite{gh2019,jiang2020ghost,yan2022ghost}. As for GMLS and GFDM, it does not automatically produce a stable estimator but the discrete approximate matrix can be stabilized with additional constraints (see Section~3.4 of \cite{liang2013solving} and Appendix A of \cite{suchde2019meshfree}). The fifth aspect is whether the method can address unknown manifolds with boundaries as in our setup where no boundary points are sampled. On this aspect, it is unclear how the RBFs technique can address this issue. In this paper, we will demonstrate that GFDM can address this configuration. Finally, easy implementation in practical problems is an important consideration. Among these methods, all of them can be easily implemented except for GMLS \cite{liang2013solving,gross2020meshfree} since the latter requires one to hard-code the representation of differential operators in local coordinates, which formulation can become cumbersome for higher-order derivatives of vector fields or forms as it depends on the estimated Riemannian metric tensor, its derivatives, and inverse.

\begin{table}[htp]
\caption{Comparison of several mesh-free solvers on unknown manifolds identified by randomly sampled data.}
\begin{center}
\begin{tabular}{|c|c|c|c|c|c|c|}
\hline
 & global RBF & RBF-FD & Graph-based & GMLS \cite{liang2013solving,gross2020meshfree} & GFDM \\ \hline
general operator & \checkmark & \checkmark &  & \checkmark & \checkmark \\
parameter robustness & & &  & \checkmark & \checkmark \\
sparsity & & \checkmark & \checkmark & \checkmark & \checkmark \\
stability & & & \checkmark & \checkmark & \checkmark \\
boundary & & &\checkmark & \checkmark & \checkmark \\
easy-implementation & \checkmark & \checkmark & \checkmark & & \checkmark
    \\ \hline
\end{tabular}
\end{center}
\label{tab_comp}
\end{table}%

Based on these considerations, we study GFDM which was first considered for solving PDEs on manifolds in  \cite{suchde2019meshfree}. In this paper, we establish the consistency and convergence of GFDM in solving Poisson problems for closed manifolds and compact manifolds with homogeneous Dirichlet boundary condition. The theoretical results provide error bounds that depend on the number of point cloud data, dimension of the manifold, and regularity of the solution/estimator. The consistency analysis relies on a representation of smooth functions on the manifolds with Taylor's expansions of polynomials defined on the tangent bundles and the established error analysis of GMLS \cite{mirzaei2012generalized}. From the analysis, one can also interpret GFDM as a generalization of the approach in \cite{flyer2016role} from Euclidean domain to submanifolds of Euclidean domain. To achieve the convergence, a stable GFDM estimator is constructed by solving a linear programming with additional constraints. The additional constraints in the proposed linear programming are devised to allow for a feasible solution even when higher-order polynomials are used, generalizing the quadratic programming approach in \cite{liang2013solving} which provides feasible solutions only for polynomials of degree less than four.

For the Dirichlet problem, we consider the volume-constraint approach which constructs a linear algebraic equation by imposing the PDEs strongly on points sufficiently away from the boundary and boundary conditions on the points that are sufficiently close to the boundary. While theoretical analysis is deduced by identifying these points through their distances from the boundary, we introduce a novel method to identify the points close to the boundary without needing to know the location of the boundary nor to approximate their distances from the boundary. The new close-to-boundary point detection method identifies points near the boundary using the sign of the weight of the negative-definite Laplacian estimator from {\color{black}a local least-squares regression approach} on each based point. Analogous to the one-sided finite-difference approximation of the second-order derivatives on the Euclidean domain, the weight on the based point will have a positive sign for points close to the boundary as we shall see. We will show the effectiveness of this novel close-to-boundary point detection method in solving the Poisson problem compared to the standard approach that imposes the boundary conditions to all points whose distance to the boundary is smaller than an empirically chosen distance.

The paper is organized as follows. In Section~\ref{sec2}, we provide a short review of the intrinsic polynomial expansion of functions and its derivatives on manifolds. In Section~\ref{Sec:FD_scheme}, we discuss the GFDM as an estimator of the Laplace-Beltrami operator, where linear programming is introduced to construct a stable estimator and as a byproduct, a close-to-boundary point detection algorithm is introduced. In Section~\ref{pde}, we provide a theoretical study for an application to solving Poisson problems. In Section~\ref{numerics}, we provide supporting numerical examples. We conclude the paper with a summary and a list of open problems in Section~\ref{sec6}. We include some proofs in AppendiX~\ref{app:A}, additional numerical results to show the consistency and convergence of GFDM when points on the boundary are available in Appendix~\ref{app:B}, and other approaches such as RBF-FD and graph-based VBDM in Appendix~\ref{app:C}.

\section{Preliminaries}\label{sec2}

\subsection{An intrinsic polynomial expansion of functions on manifolds}\label{Sec:Taylor}
Let $M$ be a $d-$dimensional manifold embedded in the ambient Euclidean space, $\mathbb{R}^n$. Let $f$ be a smooth function on $M$ and fix a point $\textbf{x}_0\in M$.  The goal of this section is to represent the geodesic normal coordinate expansion of $f$ at $\textbf{x}\in M$ near to the point $\textbf{x}_0$ in terms of "intrinsic" polynomials (or local polynomials defined on the tangent space at the base point $\textbf{x}_0$). First, we recall that the exponential or Riemannian normal coordinates at $\textbf{x}_0$ could be defined by identifying the tangent space $T_{\textbf{x}_0}M$ with $\mathbb{R}^d$ and using that the exponential map $\text{exp}_{\textbf{x}_0}:T_{\textbf{x}_0}M\rightarrow M$ is a diffeomorphism on a neighborhood of the origin (also see Section 5.5 in \cite{lee2018introduction}). Particularly, for any point $\textbf{x} \in M$  in the neighborhood of $\textbf{x}_0 \in M$, its geodesic distance is $\|\vec{s}\|$, where $\vec{s}=(s_1,...,s_d)$ denotes the corresponding normal coordinate that satisfies,
$$
\text{exp}_{\textbf{x}_0}(\vec{s})=\textbf{x},\quad \text{with}\ \ \text{exp}_{\textbf{x}_0}(\vec{0})=\textbf{x}_0.
$$
For convenience of the discussion below, we define a local mapping $X:\mathbb{R}^d\rightarrow M\subset \mathbb{R}^n$ as $X(\vec{s})=\text{exp}_{\textbf{x}_0}(\vec{s})$. We denote the $i$th tangent vector $\textbf{t}_i = \frac{\partial}{\partial s_i}X(\vec{0}) \in \mathbb{R}^n,$ which is also the $i$th column of the $n\times d$ Jacobian matrix of $X$, denoted by $\textbf{J}$ with components $\textbf{J}_{ji}=\frac{\partial }{\partial s_i}X^j(\vec{0})$ for $i=1,\ldots,d$ and $j=1,\ldots,n$. Since $\vec{s}$ is a normal coordinate of $\textbf{x}$ with reference $\textbf{x}_0$,  then $\big\langle\frac{\partial}{\partial s_i}X(\vec{0}),\frac{\partial}{\partial s_j}X(\vec{0})\big\rangle= \textbf{t}_i^\top \textbf{t}_j = \delta_{ij}$, where $\delta_{ij}$ is the Kronecker delta.

The Taylor expansion of $X$ centered at $\vec{0}$ is
\BEA
X(\vec{s})-X(\vec{0})=\sum_{i=1}^{d} s_{i} \frac{\partial}{\partial s_{i}} X(\vec{0})+\frac{1}{2} \sum_{i, j=1}^{d} s_{i} s_{j} \frac{\partial^2}{\partial s_{i} \partial s_{j}} X(\vec{0})+O\left(\|\vec{s}\|^{3}\right).
\EEA
 Multiplying by $\textbf{J}^\top$ at both sides of above equality,
\BEA
\textbf{z}:=\textbf{J}^\top(\textbf{x}-\textbf{x}_0)=\textbf{J}^\top(X(\vec{s})-X(\vec{0}))=\vec{s}+\frac{1}{2}\sum_{i,j=1}^ds_is_j\vec{\Gamma}_{ij}+O(\Vert \vec{s}\Vert^3).
\EEA
where we used the notation $\vec{\Gamma}_{ij}=(\Gamma_{ij}^1,...,\Gamma_{ij}^d)^\top$ with Christoffel symbols $\Gamma_{ij}^k = \big\langle\frac{\partial}{\partial s_k}X(\vec{0}),\frac{\partial^2}{\partial s_i\partial s_j}X(\vec{0})\big\rangle$. The resulting $d$-dimensional vector $\textbf{z}=(z^1,..., z^d)^\top$ has components $z^i=\textbf{t}_i^\top (\textbf{x}-\textbf{x}_0)$. Using above equality we can write the normal coordinate $\vec{s}$ in terms of the $d$-dimensional Cartesian coordinate $\textbf{z}$:
\BEA
\vec{s}&=&\textbf{z}-\frac{1}{2}\sum_{i,j=1}^d(z^i+O(\Vert\vec{s}\Vert^2))(z^j+O(\Vert\vec{s}\Vert^2))\vec{\Gamma}_{ij}+O(\Vert\vec{s}\Vert^3)\notag\\
&=& \textbf{z}-\frac{1}{2}\sum_{i,j=1}^d z^i z^j\vec{\Gamma}_{ij}+O(\Vert\vec{s}\Vert^3),\label{exp_s}
\EEA
where we use the relation $\Vert\vec{s}\Vert=O(\Vert\textbf{z}\Vert)$.

For a smooth function $f$ on $M$, we can define a function $\tilde{f}:\mathbb{R}^d\rightarrow \mathbb{R}$ by setting $\tilde{f}(\vec{s}):=f(\text{exp}_{\textbf{x}_0}(\vec{s}))=f(\textbf{x})$ with $\tilde{f}(\vec{0})=f(\textbf{x}_0)$. A Taylor expansion of $\tilde{f}$ with respect to $\vec{s}$ could be written as
\BEA
f(\textbf{x})=\tilde{f}(\vec{s})=\tilde{f}(\vec{0})+\sum_{i=1}^{d} s_{i} \frac{\partial}{\partial s_{i}} \tilde{f}(\vec{0})+\frac{1}{2} \sum_{i, j=1}^{d} s_{i} s_{j} \frac{\partial^{2}}{\partial s_{i} \partial s_{j}} \tilde{f}(\vec{0})+O(\Vert \vec{s}\Vert^3).
\label{eqn:taylor_exp}
\EEA
Substituting \eqref{exp_s} into \eqref{eqn:taylor_exp}, we obtain
$$
\begin{aligned}
f(\textbf{x})&=\tilde{f}(\vec{0})+\sum_{k=1}^d( z^k-\frac{1}{2}\sum_{i,j=1}^d z^i z^j\Gamma^k_{ij})\frac{\partial}{\partial s_k}\tilde{f}(\vec{0})+\frac{1}{2}\sum_{i,j=1}^dz^iz^j\frac{\partial^2}{\partial s_i\partial s_j}\tilde{f}(\vec{0})+O(\Vert\textbf{z}\Vert^3)\\
&=f(\textbf{x}_0)+\sum_{k=1}^dz^k\frac{\partial}{\partial s_k}\tilde{f}(\vec{0})+\frac{1}{2}\sum_{i,j=1}^dz^iz^j\left(\frac{\partial^2}{\partial s_i\partial s_j}\tilde{f}(\vec{0})-\sum_{k=1}^d\Gamma_{ij}^k\frac{\partial}{\partial s_k}\tilde{f}(\vec{0}) \right)+O(\Vert\textbf{z}\Vert^3).
\end{aligned}
$$
Since $\frac{\partial}{\partial s_k}\tilde{f}(\vec{0}),\frac{\partial^2}{\partial s_i\partial s_j}\tilde{f}(\vec{0}),\Gamma_{ij}^k$ are constants that only depend on $\textbf{x}_0$, the above equality suggests the following equation holds on a neighborhood of $\textbf{x}_0$:
\BEA
f(\textbf{x})=f(\textbf{x}_0)+\sum_{1\leq |\alpha|\leq 2}b_\alpha \textbf{z}^\alpha+O(\Vert\textbf{z}\Vert^3).
\label{eqn:int_poly}
\EEA
Using the same argument, one can also show that for any degree $l\geq 2$,
\BEA
f(\textbf{x})=f(\textbf{x}_0)+\sum_{1\leq|\alpha|\leq l}b_\alpha \textbf{z}^\alpha+O(\Vert \textbf{z}\Vert^{l+1}).
\label{eqn:int_poly2}
\EEA
where the coefficients $\{b_\alpha\}_{1\leq|\alpha|\leq l}$ depend only on $\textbf{x}_0$, $\alpha=(\alpha_1,..,\alpha_d)$ denotes the multi-index notation with $|\alpha|=\sum_{i=1}^d\alpha_i$, and $\textbf{z}^\alpha = \prod_{i=1}^d (z^i)^{\alpha_i}$.

\begin{rem}\label{rem_basis}
The key point in \eqref{eqn:int_poly2} is that any smooth functions on manifolds can be Taylor expanded with respect to the intrinsic polynomial basis functions $\{\mathbf{z}^\alpha\}_{|\alpha|\leq l}$.
The coefficients $\{b_\alpha\}_{|\alpha|\leq l}$ will be determined via a least-squares regression approach as shown in Section \ref{Sec:int_LS}. Incidentally, we will not pursue the dependence of the coefficients on the geometry for high-order terms due to its complexity. While all Christoffel symbols vanish at the base point in the normal coordinate, we still keep them in the above derivation to indicate that the dependence on the geometry can be very complicated for high-order terms.
\end{rem}

\subsection{Review of approximation of differential operators on manifolds}\label{Sec:diff_ambient}
In this section, we first review the representations of the differential operators  on manifolds as tangential gradients  \cite{fuselier2013high,harlim2022rbf}, which could be formulated as the projection of the appropriate derivatives in the ambient space. Subsequently, we approximate the differential operators based on a local polynomial approximation based on the expansion discussed in the previous subsection.

For any point $\textbf{x}=(x^1,...,x^n)\in M \subset \mathbb{R}^n$, we denote the tangent space of $M$ at $\textbf{x}$ as $T_\textbf{x}M$ and  a set of orthonormal tangent vectors that span this tangent space as $\{\textbf{t}_i\}_{i=1}^d$. Then the projection matrix $\textbf{P}$ which projects vectors in $\mathbb{R}^n$ to $T_\textbf{x}M$ could be written as $\textbf{P}=\sum_{i=1}^d\textbf{t}_i\textbf{t}_i^\top$ (see e.g. \cite{harlim2022rbf} for a detailed derivation). Subsequently, the surface gradient acting on a smooth function $f:M\to \mathbb{R}$ evaluated at $\textbf{x}\in M$ in the Cartesian coordinates can be given as,
$$
\textup{grad}_g f(\mathbf{x}):=\textbf{P}\overline{\textup{grad}}_{\mathbb{R}^n}f(\mathbf{x})=\Big(\sum_{i=1}^d\textbf{t}_i\textbf{t}_i^\top\Big)\overline{\textup{grad}}_{\mathbb{R}^n}f(\mathbf{x}),
$$
where $\overline{\textup{grad}}_{\mathbb{R}^n}=[\partial_{x^1}, \cdots, \partial_{x^n}]^\top$ is the usual gradient operator defined in $\mathbb{R}^n$ and the subscript $g$ is to denote the differential operator defined with respect to the Riemannian metric $g$ induced by $M$ from $\mathbb{R}^n$. Let $\mathbf{e}_\ell,\ell=1,...,n$ be the standard orthonormal vectors corresponding to the coordinate $x^\ell$  in Euclidean space $\mathbb{R}^n$, we can rewrite above expression in component form as
\BEA
\textup{grad}_g f (\mathbf{x}):=\left[\begin{array}{c}\left(\mathbf{e}_{1} \cdot \mathbf{P}\right) \overline{\textup{grad}}_{\mathbb{R}^n} f(\mathbf{x}) \\ \vdots\\ \left(\mathbf{e}_{n} \cdot \mathbf{P}\right) \overline{\textup{grad}}_{\mathbb{R}^n}f(\mathbf{x}) \end{array}\right]
=\left[\begin{array}{c}\mathbf{P}_{1} \cdot \overline{\textup{grad}}_{\mathbb{R}^n}f(\mathbf{x}) \\ \vdots \\ \mathbf{P}_{n} \cdot \overline{\textup{grad}}_{\mathbb{R}^n} f(\mathbf{x})\end{array}\right]:=\left[\begin{array}{c}\mathcal{G}_{1}f(\mathbf{x}) \\ \vdots\\ \mathcal{G}_{n} f(\mathbf{x})\end{array}\right]
\label{eqn:grad}
\EEA
where $\textbf{P}_\ell$ is the $\ell-$th row of the projection matrix $\textbf{P}$. With this identity, one can write
the Laplace-Beltrami operator of $f$ at $\textbf{x}\in M$ as
\BEA
\Delta_{M}f (\mathbf{x}) :=\textup{div}_{g} \textup{grad}_{g} f (\mathbf{x})=\left(\mathbf{P} \overline{\textup{grad}}_{\mathbb{R}^n}\right) \cdot\left(\mathbf{P} \overline{\textup{grad}}_{\mathbb{R}^n}\right)f (\mathbf{x})=\sum_{\ell=1}^n\mathcal{G}_{\ell} \mathcal{G}_{\ell}f (\mathbf{x}),
\label{eqn:laplace}
\EEA
where $\Delta_{M}$ is defined to be negative-definite.

To approximate the differential operators in \eqref{eqn:grad} and \eqref{eqn:laplace}, we need an approximate model for $f$. There are various ways to solve this problem, including finding an interpolation solution (or effectively collocation method) \cite{shankar2015radial,fuselier2013high} or regression solution \cite{liang2013solving,suchde2019meshfree,gross2020meshfree}. In this paper, we consider a regression solution. Let us illustrate the idea in approximating the
Laplace-Beltrami operator on manifolds. Let $f:M\rightarrow\mathbb{R}$ be an arbitrary smooth function. Given a set of (distinct) nodes $X=\{\textbf{x}_i\}_{i=1}^N\subset M$ and function values $\mathbf{f}:=(f(\textbf{x}_1),\ldots, f(\textbf{x}_N))^\top$. Let $\mathcal{I}:\mathbf{f}\in\mathbb{R}^N\rightarrow \mathcal{I}\mathbf{f}\in C^m(\mathbb{R}^n)$ be an arbitrary approximation operator corresponding to the estimator $\mathcal{I}\textbf{f}$ for the function $f$. For convenience of the discussion below, we define the following restriction operator,
$$
\mathcal{R}_Xf=\textbf{f}=(f(\textbf{x}_1),...,f(\textbf{x}_N))^\top, \quad \textbf{x}_i \in X.
$$
For an interpolation solution, one seeks an estimator for the Laplace-Beltrami operator acting on any $f\in C^2(M)$ that satisfies,
\comment{In practice, some common choices of radial function include the Gaussian $\phi_s(r)=\text{exp}(-(sr)^2)$ \cite{fasshauer2012stable} or Matern class kernel \cite{fuselier2013high} where $s$ is called the shape parameter.  For conditionally positive definite radial function, one may include polynomial terms in above interpolation, together with some constraints on the coefficients (see Section 2.1 of \cite{flyer2016role} for example). An example of conditionally positive definite radial function is polyharmonic splines (PHS): $\phi(r)=r^{2m-1}$ or $r^{2m}\log r, m\in\mathbb{N}$. Notice that no shape parameter is needed in PHS radial function.}
\BEA
\mathcal{R}_X\Delta_Mf=\mathcal{R}_X\big[\sum_{\ell=1}^n\mathcal{G}_\ell \mathcal{G}_\ell f\big] \approx \big[ \sum_{\ell=1}^n (\mathcal{R}_X\mathcal{G}_\ell \mathcal{I}) (\mathcal{R}_X \mathcal{G}_\ell \mathcal{I}\mathbf{f})\big].\label{LBestimator}
\EEA
Practically, this also requires an appropriate approximate model (e.g., radial basis interpolants \cite{fuselier2013high,shankar2015radial} with a number of coefficient parameters that equals the number of training data, $N$). Importantly, the distribution of the data in $X$ plays a crucial role in ensuring the wellposed-ness of the interpolating solution (see Chapter~12 of \cite{Wendland2005Scat}). In this paper, we will consider imposing the condition in \eqref{LBestimator} in a least-squares sense, where we will employ the local polynomial model based on the Taylor expansion discussed in the previous section, such that the number of local regression coefficients is much smaller than the number of training data, $N$. In the next section, we shall see that the proposed approach yields  a finite-difference type scheme based on the polynomial least squares approximation. We call the proposed approach a finite-difference type method analogous to the classical finite-difference approximation on equispaced grid points that arise from taking derivatives of the polynomial interpolating function.

\section{GFDM approximation of Laplace-Beltrami on manifolds}\label{Sec:FD_scheme}

For an arbitrary point ${\textbf{x}_0}\in X\subset M$, we denote its $K-$nearest neighbors in $X$ by $S_{{\textbf{x}_0}}=\{{\textbf{x}_{0,k}}\}_{k=1}^K\subset X$. In most literature, such a set $S_{{\textbf{x}_0}}$ is called  a ``stencil''. By definition, we have ${\textbf{x}_{0,1}}={\textbf{x}_0}$ to be the base point. Denote $\textbf{f}_{{\textbf{x}_0}}=(f({\textbf{x}_{0,1}}),...,f(\textbf{x}_{0,K}))^\top$. Our goal is to approximate the Laplace-Beltrami operator $\Delta_M$ acting on a function $f$ at the base point ${\textbf{x}_0}$ by a linear combination of the function values $\{f({\textbf{x}_{0,k}})\}_{k=1}^K$, i.e., find the weights $\{w_k\}_{k=1}^K$ such that
\BEA
\Delta_M f({\textbf{x}_0}) \approx \sum_{k=1}^Kw_kf({\textbf{x}_{0,k}}).
\label{eqn:FD}
\EEA
Arranging the weights at each point into each row of a sparse $N$ by $N$ matrix $L_X$, we can approximate the operator over all points by $L_X\mathbf{f}$.

The remainder of this section is organized as follows. In Subsection~\ref{Sec:int_LS}, we will demonstrate how to calculate the weights $\{w_k\}_{k=1}^K$ in (\ref{eqn:FD}) for the Laplace-Beltrami operator by applying the approach introduced in Section~\ref{Sec:diff_ambient}, together with a polynomial least squares approximation. In this paper, we will consider a type of generalized moving least-squares (GMLS) fitting to the intrinsic polynomials, which yields the generalized finite difference method (GFDM) that was introduced in \cite{suchde2019meshfree}. At the end of this subsection, we will mention the existing theoretical error bound for GMLS \cite{mirzaei2012generalized} and numerically verify this convergence rate on a simple test example and simultaneously demonstrate the advantages of this representation over the extrinsic polynomial approach proposed in \cite{flyer2016role}. In Subsection~\ref{stabilization_section}, we will see how unstable is the operator approximation and we will impose additional constraints to stabilize the intrinsic polynomial least squares discretization. In Section~\ref{pde}, we will apply the approximation to solve Poisson problems on compact manifolds without and with boundaries.

 \subsection{Intrinsic polynomial least squares approximation}\label{Sec:int_LS}
Inspired by the expansion (\ref{eqn:int_poly2}) in Section \ref{Sec:Taylor}, we consider using the "intrinsic" polynomials to approximate functions over a neighborhood of ${\textbf{x}_0}$. First, let $\mathbb{P}_{\textbf{x}_0}^{l,d}$ be the space of intrinsic polynomials with degree up to $l$ in $d$-dimensions at the point ${\textbf{x}_0}$, i.e., $\mathbb{P}_{\textbf{x}_0}^{l,d}=\text{span}(\{p_{{\textbf{x}_0},\alpha}\}_{|\alpha|\leq l})$, where  $\alpha=(\alpha_1,..,\alpha_d)$ is the multi-index notation and $p_{{\textbf{x}_0},\alpha}$ is the basis polynomial functions defined as
$$
 p_{{\textbf{x}_0},\alpha}(\textbf{x})= \textbf{z}^{\alpha} =  \prod_{i=1}^d (z^i)^{\alpha_i}= \prod_{i=1}^d \left[ \textbf{t}_{{\textbf{x}_0},i}\cdot(\textbf{x}-{\textbf{x}_0})\right]^{\alpha_i},\quad |\alpha|\leq l.
$$
Here, $ \textbf{t}_{{\textbf{x}_0},i}$ is the $i$th tangent vector at ${\textbf{x}_0}$ as defined in Section \ref{Sec:diff_ambient}. By definition, the dimension of the space  $\mathbb{P}_{\textbf{x}_0}^{l,d}$ is $m=\left(\begin{matrix}l+d\\ d\end{matrix}\right)$. For $K>m$, we can define an operator $\mathcal{I}_{\mathbb{P}}:\textbf{f}_{{\textbf{x}_0}}\in\mathbb{R}^K\rightarrow \mathcal{I}_{\mathbb{P}}\textbf{f}_{{\textbf{x}_0}}\in \mathbb{P}_{\textbf{x}_0}^{l,d}$ such that $\mathcal{I}_{\mathbb{P}}\textbf{f}_{{\textbf{x}_0}}$ is the optimal solution of the following least-squares problem:
 \BEA
\underset{q\in \mathbb{P}_{\textbf{x}_0}^{l,d}}{\operatorname{min}}\sum_{k=1}^K\left(f({\textbf{x}_{0,k}})-q ({\textbf{x}_{0,k}})\right)^2.
 \label{eqn:int_LS}
 \EEA
One can also consider weighted least-squares with weighted norms as in {\cite{Wendland2005Scat,liang2013solving,gross2020meshfree}}, where an additional calibration of the weights yields more accurate results when the data is quasi-uniform. Here, we only consider an equal weight since we do not see advantages in our setting with randomly sampled data.
The solution to the least-squares problem (\ref{eqn:int_LS}) can be represented as $\mathcal{I}_{\mathbb{P}}\textbf{f}_{{\textbf{x}_0}}=\sum_{|\alpha|\leq l}b_\alpha p_{{\textbf{x}_0},\alpha}$, where the concatenated coefficients $\textbf{b}=(b_{\alpha(1)},...,b_{\alpha(m)})^\top$ satisfy the normal equation,
 \BEA
 (\boldsymbol{\Phi}^\top\boldsymbol{\Phi})\textbf{b}=\boldsymbol{\Phi}^\top\textbf{f}_{{\textbf{x}_0}}.
 \label{eqn:inter}
 \EEA
Notice that  here is $m$ number of coefficients $\alpha$ such that $|\alpha|\leq l$ and thereafter we used the notation $\left\{\alpha(j)\right\}_{j=1,\dots, m}$ to denote all possible $m$ multi-indices. The above normal equation can be uniquely identified if $\boldsymbol{\Phi}$, with components
 \BEA
 \boldsymbol{\Phi}_{kj}=p_{{\textbf{x}_0},\alpha(j)}({\textbf{x}_{0,k}}),\quad 1\leq k\leq K,\ 1\leq j\leq m,
 \label{eqn:matrix_phi}
 \EEA
is a full rank matrix.

Using the notations defined in Section \ref{Sec:diff_ambient}, we can approximate the differential operator,
$$
\mathcal{G}_\ell f({\textbf{x}_{0,k}})\approx (\mathcal{G}_\ell \mathcal{I}_{\mathbb{P}}\textbf{f}_{{\textbf{x}_0}})({\textbf{x}_{0,k}})=\sum_{|\alpha|\leq l}b_{\alpha}\mathcal{G}_\ell p_{{\textbf{x}_0},\alpha}({\textbf{x}_{0,k}}),\quad \forall k=1,...,K, \ \ell=1,...,n.
$$
For each $\ell$, the above relation can also be written in matrix form as,
\BEA
  \left[\begin{array}{c}\mathcal{G}_\ell f({\mathbf{x}_{0,1}}) \\ \vdots \\ \mathcal{G}_\ell f({\mathbf{x}_{0,K}})\end{array}\right]
 \approx
\underbrace{\left[\begin{array}{ccc}\mathcal{G}_\ell p_{\mathbf{x}_0,\alpha(1)}(\mathbf{x}_{0,1}) & \cdots & \mathcal{G}_\ell p_{\mathbf{x}_0,\alpha(m)}(\mathbf{x}_{0,1}) \\ \vdots & \ddots & \vdots \\ \mathcal{G}_\ell p_{\mathbf{x}_0,\alpha(1)}(\mathbf{x}_{0,K}) & \cdots  & \mathcal{G}_\ell p_{\mathbf{x}_0,\alpha(m)}(\mathbf{x}_{0,K}) \end{array}\right] }_{\mathbf{B}_\ell}
\underbrace{\left[\begin{array}{c}b_{\alpha(1)} \\ \vdots \\ b_{\alpha(m)}\end{array}\right]}_{\mathbf{b}}
=\mathbf{B}_\ell(\boldsymbol{\Phi}^\top \boldsymbol{\Phi})^{-1}\boldsymbol{\Phi}^\top\mathbf{f}_{{\textbf{x}_0}},\label{eqn:G_ell} 
\EEA
where $\mathbf{B}_\ell$ is a $K$ by $m$ matrix with $[\mathbf{B}_\ell]_{ij}=\mathcal{G}_\ell p_{{\textbf{x}_0},\alpha(j)}({\textbf{x}_{0,i}})$ and we have used $\boldsymbol{\Phi}$ as defined in (\ref{eqn:matrix_phi}) in the last equality. Hence, the differential matrix for the operator $\mathcal{G}_\ell$ over the stencil is approximated by the $K$ by $K$ matrix,
\BEA
\mathbf{G}_\ell:=\mathbf{B}_\ell(\boldsymbol{\Phi}^\top \boldsymbol{\Phi})^{-1}\boldsymbol{\Phi}^\top.
\label{eqn:diff-op}
\EEA
Then the Laplace-Beltrami operator can be approximated at  the base point ${\mathbf{x}_0}$ as,
\BEA
\Delta_Mf ({\mathbf{x}_0})=\sum_{\ell=1}^n\mathcal{G}_\ell\mathcal{G}_\ell f ({\mathbf{x}_0})\approx \sum_{\ell=1}^n\mathcal{G}_\ell \mathcal{I}( \mathbf{G}_\ell\textbf{f}_{{\textbf{x}_0}})({\mathbf{x}_0})\approx  \big(\sum_{\ell=1}^n \mathbf{G}_\ell \mathbf{G}_\ell \textbf{f}_{{\textbf{x}_0}}\big)_1, \label{eqn:GlG1}
\EEA
where subscript$-1$ is to denote the first element of the resulting $K$-dimensional vector. Denoting the elements of the first row of the $K$ by $K$ matrix $ \sum_{\ell=1}^n \mathbf{G}_\ell \mathbf{G}_\ell$ by $\{w_k\}_{k=1}^K$, we have identified the weights in \eqref{eqn:FD}, that is,
\BEA
\Delta_Mf ({\mathbf{x}_0}) \approx  \sum_{k=1}^Kw_kf({\textbf{x}_{0,k}}).
\label{eqn:approximation}
\EEA
{\color{black}Notice that these weights in \eqref{eqn:approximation} are prescribed in analogous to that in the classical finite-difference scheme on equal-
spaced points in the Euclidean domain. It is worth to note that this formulation is a version of GFDM introduced by \cite{suchde2019meshfree} where we use equal weights in the generalized moving least-squares (GMLS) fitting \cite{mirzaei2012generalized}, whereas the classical finite-difference scheme uses the local polynomial interpolation. We will refer to the estimate in \eqref{eqn:approximation} as the least-squares estimate and the weights $w_k$ as the least-squares weights, to distinguish with the GMLS approach in \cite{liang2013solving,gross2020meshfree} which is employed to estimate both metric tensors and functions. } The complete procedure is listed in Algorithm \ref{algo:local-LS}.


\begin{algorithm}[ht]
	\caption{Intrinsic Polynomial Least Squares Approximation for the Laplace-Beltrami Operator}
	\begin{algorithmic}[1]
	\STATE {\bf Input:} A set of (distinct) nodes $X=\{\textbf{x}_i\}_{i=1}^N\subset M$, a positive constant $K(\ll N)$ nearest neighbors,  the degree of intrinsic polynomials $l$,  projection matrix $\mathbf{P}$ with each row denoted by $\mathbf{P}_\ell$, and orthonormal tangent vectors $\{\textbf{t}_{\mathbf{x}_i,j}\}_{j=1}^d$ at each point $\mathbf{x}_i$.
	\STATE Set $m=\left(\begin{matrix}l+d\\ d\end{matrix}\right)$. Set $L_X$ to be a sparse $N$ by $N$ matrix with $NK$ nonzeros.
	\FOR{ $i\in\{1,...,N\}$ }
		\STATE Find the $K$ nearest neighbors of the point $\textbf{x}_i$ in the stencil $S_{\textbf{x}_i}=\{{\textbf{x}_{i,k}}\}_{k=1}^K$.
		\STATE Construct the $K$ by $m$ matrix $\boldsymbol{\Phi}$ with
		$$
        \boldsymbol{\Phi}_{kj}=p_{{\textbf{x}_i},\alpha(j)}({\textbf{x}_{i,k}})=\prod_{r=1}^d \left[ \textbf{t}_{{\textbf{x}_i},r}\cdot(\textbf{x}_{i,k}-{\textbf{x}_{i}})\right]^{\alpha(j)_r}.
		$$
		\STATE For $1\leq \ell\leq n$, construct the $K$ by $m$ matrix $\textbf{B}_\ell$ with
		$$
        [\mathbf{B}_\ell]_{kj}=\mathcal{G}_\ell p_{{\textbf{x}_i},\alpha(j)}({\textbf{x}_{i,k}})=\big(\mathbf{P}_{\ell} \cdot \overline{\textup{grad}}_{\mathbb{R}^n}\big) \left(p_{{\textbf{x}_i},\alpha(j)}({\textbf{x}_{i,k}})\right).
		$$
		\STATE Set $\mathbf{G}_\ell:=\mathbf{B}_\ell(\boldsymbol{\Phi}^\top \boldsymbol{\Phi})^{-1}\boldsymbol{\Phi}^\top$ and $\tilde{L}:=\sum_{\ell=1}^n\mathbf{G}_\ell\mathbf{G}_\ell$.
		\STATE Extract the first row of the matrix $\tilde{L}$ {\color{black}to obtain the least-squares weights $w_k$} and arrange them into corresponding columns in the $i$th row of $L_X$.
	\ENDFOR
	\STATE {\bf Output:} The approximate operator matrix $L_X$.
	\end{algorithmic}\label{algo:local-LS}
\end{algorithm}

\begin{rem}\label{rem_error}
Theoretical error estimate for such the generalized moving least-squares approximation has been derived in \cite{mirzaei2012generalized}. Their error bound is described in terms of the fill distance. In our context, we define the fill distance with respect to geodesic distance, $d_g:M\times M \to \mathbb{R}$ as,
\BEA
h_{X,M} := \sup_{x\in M} \min_{j\in\{1,\ldots, N\}} d_g(\mathbf{x}, \mathbf{x}_j).\notag
\EEA
which is bounded below by the separation distance,
\BEA
q_{X,M} = \frac{1}{2} \min_{i\neq j} d_g(\mathbf{x}_i,\mathbf{x}_j).\notag
\EEA
Let $X\subseteq M$ 
be such that $h_{X,M} \leq h_0$ for some constant $h_0>0$. Define $M^* = \bigcup_{\mathbf{x}\in M} B(\mathbf{x},C_2h_0)$, a union of geodesic ball over the length $C_2h_0$, for some constant $C_2>0.$ Then for any $f \in C^{l+1}(M^*)$ and $\alpha$ that satisfies $|\alpha |\leq l$,
\BEA
\big|D^\alpha f (\mathbf{x}) - \widehat{D^\alpha f}(\mathbf{x}) \big|  \leq C h_{X,M}^{l+1-|\alpha|}\big|f\big|_{C^{l+1}(M^*)}, \label{MLS_error}
\EEA
for all $\mathbf{x}\in M$ and some $C>0$.
Here the semi-norm
\[
|f|_{C^{l+1}(M^*)} := \max_{|\beta| = l+1} \|D^\beta f\|_{L^\infty (M^*)},
\]
is defined over $M^*$. In the error bound above, we used the notation $\widehat{D^\alpha f}$ as the GMLS approximation to $D^\alpha f$ using local polynomials up to degree $l$, where $D^\alpha$ denotes a general multi-dimensional derivative with multiindex $\alpha$. For the approximation of the Laplace-Beltrami operator, we have
\BEA
\left| \Delta_M f(\mathbf{x}_i) - (L_X \mathbf{f})_i \right|
= \big| \Delta_M f(\mathbf{x}_i) - \sum _{k=1}^K w_k f(\mathbf{x}_{i,k}) \big|
= O(h_{X,M}^{l-1}).\label{errorrate}
\EEA
For uniformly sampled data $X\subset M$, one can show that with probability higher than $1-\frac{1}{N}$, $h_{X,M} = O(N^{-\frac{1}{d}})$ (see Lemma B.2 in \cite{harlim2022rbf}), so the error rate in terms of $N$ is of order-$N^{-\frac{l-1}{d}}$.
We will verify this error rate in the following numerical example. One can also show that if the $d$-dimensional Riemannian manifold $M$ has positive Ricci curvature at all points in $X$, then with probability higher than $1-\frac{1}{N}$, we have $c(d) N^{-\frac{2}{d}}\leq q_{X,M}$  for some constant $c(d)>0$ (see Lemma A.2 in \cite{yan2023spectral} for the detailed proof). With the upper bound on the fill distance, then with probability higher than $1-\frac{2}{N}$, we have
\BEA
c(d) N^{-\frac{2}{d}}\leq q_{X,M} \leq h_{X,M} \leq C(d)N^{-\frac{1}{d}}\label{filldist_bdd}
 \EEA
for some constant $C(d)>0$. This bound will be useful for the discussion for manifold with boundary. 
\end{rem}

\begin{example}\label{ex1}
In Fig.~\ref{fig1_cons}(a), we illustrate the consistency of the discrete operator $L_X$ obtained from Algorithm \ref{algo:local-LS} on 1D ellipse embedded in $\mathbb{R}^2$ via the
following map,
\begin{equation}
\mathbf{x}:=\left( x^{1},x^{2}\right) =(\cos \theta ,2\sin \theta )\in M\subset
\mathbb{R}^{2}.  \label{eqn:ellpar}
\end{equation}%
The Riemannian metric is given by,
\begin{equation}
g\left( \theta \right) =\sin ^{2}\theta +4\cos ^{2}\theta ,\text{ \
\ for\ }0\leq \theta < 2\pi .  \label{Eqn:ellipse_metr}
\end{equation}%
The negative-definite Laplace-Beltrami operator acting on any smooth function $u$ is given in local
coordinates as,
\BEA
\Delta _{M}u:=\frac{1}{\sqrt{|g|}}\frac{\partial }{\partial \theta }\left(
\sqrt{\left\vert g\right\vert }g^{-1}\frac{\partial u}{\partial \theta }%
\right) .
\label{eqn:Lap_ellipse}
\EEA
\end{example}

In this numerical experiment, we test the consistency for the approximation of the Laplace-Beltrami operator acting on the function $u(\mathbf{x}(\theta))=\sin(\theta)$. The approximation is based on
point clouds $\left\{ \mathbf{x}_{i}\right\} _{i=1}^{N}$ that are randomly and uniformly distributed in the intrinsic coordinate $\theta \in  [0,2\pi )$.
We fix $K=21$ nearest neighbors for computational efficiency. Here,10 independent random samples $X$ are run for each $N=[800,1600,3200,6400,12800,25600]$ and we will examine the accuracy of the approximation with polynomials of degree $l=2,3,4,5$. The forward error (FE),
\BEA
\mathbf{FE}=\max\limits_{1\leq i\leq N}|\Delta_Mu(\mathbf{x}_i)-(L_X\mathbf{u})_i|,
\EEA
is defined to verify the error rate in \eqref{errorrate}, where $\mathbf{u} = \mathcal{R}_X u = \left(u(\mathbf{x}_1),\ldots,u(\mathbf{x}_N)\right)^\top$. \comment{\color{red}We should point out that since the data is random, the matrix $\boldsymbol{\Phi}^\top\boldsymbol{\Phi}$ is ill-conditioned. To deal with such singularity, we used ridge regression in the least squares approximation (\ref{eqn:ext_LS}). This is for extrinsic?}

In Fig.~\ref{fig1_cons}(a), we plot the errors with error bars obtained from the standard deviations from 10 independent randomly sampled training data $X$. We see that the $\mathbf{FE}$ of {\color{black}the least-squares} estimator with intrinsic polynomials converges with rates that agree with the theoretical error bound in the previous remark, namely of order-$N^{-(l-1)}$ for this one-dimensional example until they reach the round-off limit before they start to increase with rate $h_{X,M}^{-2}=O(N^{2})$. 
For instance, for polynomial of degree $l=5$, the minimal \textbf{FE} can be reached around $10^{-15}/h_{X,M}^{2}\approx 10^{-15}N^{2}\approx 10^{-7}$ when $N\approx 10^{4}$, which is the accuracy barrier  in double precision. The similar phenomenon, but for the Laplacian approximation in Euclidean space, is also observed in reference \cite{flyer2016role}. 

In Fig.~\ref{fig1_cons}(b), we show the corresponding result obtained using extrinsic polynomial approximation \cite{flyer2016role}, where the only difference from the intrinsic polynomial approach is that the least-squares problem in \eqref{eqn:int_LS} is solved over the space of polynomials in the ambient space $\mathbb{R}^n$. Notice that for the same degree $l$, the extrinsic polynomial approximation uses more basis functions compared to the intrinsic one. For example, when degree $l=2$, six centered monomial basis functions, $\{1,x^{1}-x^{1}_0,x^{2}-x^{2}_0, (x^{1}-x^{1}_0) ^{2},(x^{1}-x^{1}_0)(x^{2}-x^{2}_0),( x^{2}-x^{2}_0) ^{2}\}$, are used for the extrinsic polynomials while only three centered monomial basis functions, $\{1,\mathbf{t}\cdot(\mathbf{x}-{\textbf{x}_0}),\left( \mathbf{t}\cdot (\mathbf{x}-{\textbf{x}_0})\right) ^{2}\}$, are used for the intrinsic polynomials. Here, $\mathbf{x}=(x^1,x^2)$ and the base point is $\mathbf{x}_0 = (x_0^1,x_0^2)$. We should point out that since the six extrinsic polynomials are linearly dependent on the 1D manifold and the data are randomly distributed, the interpolation matrix $\boldsymbol{\Phi}^\top\boldsymbol{\Phi}$ constructed from the extrinsic polynomial basis functions is ill-conditioned. To deal with such singularity in the design matrix of the extrinsic polynomials, we used ridge regression in the least squares approximation. On the other hand, for the intrinsic polynomials, such a singularity does not occur based on our numerical experiments. One can see from Fig.~\ref{fig1_cons}(b) that when $N$ is small, the forward error (\textbf{FE}) for extrinsic polynomials are relatively small compared to those for intrinsic polynomials in Fig.~\ref{fig1_cons}(a). However, no obvious consistency can be observed due to the linear dependence of extrinsic polynomials on manifolds (Fig.~\ref{fig1_cons}(b)). When $N$ reaches around $4000$, the \textbf{FE} for extrinsic polynomial reach the \textquotedblleft round-off limit\textquotedblright\ and start to increase on the order of $O(N^{2}).$

\begin{figure*}[htbp]
{\scriptsize \centering
\begin{tabular}{cc}
{\normalsize (a) intrinsic polynomials} & {\normalsize (b) extrinsic
polynomials} \\
\includegraphics[width=3
in, height=2.2 in]{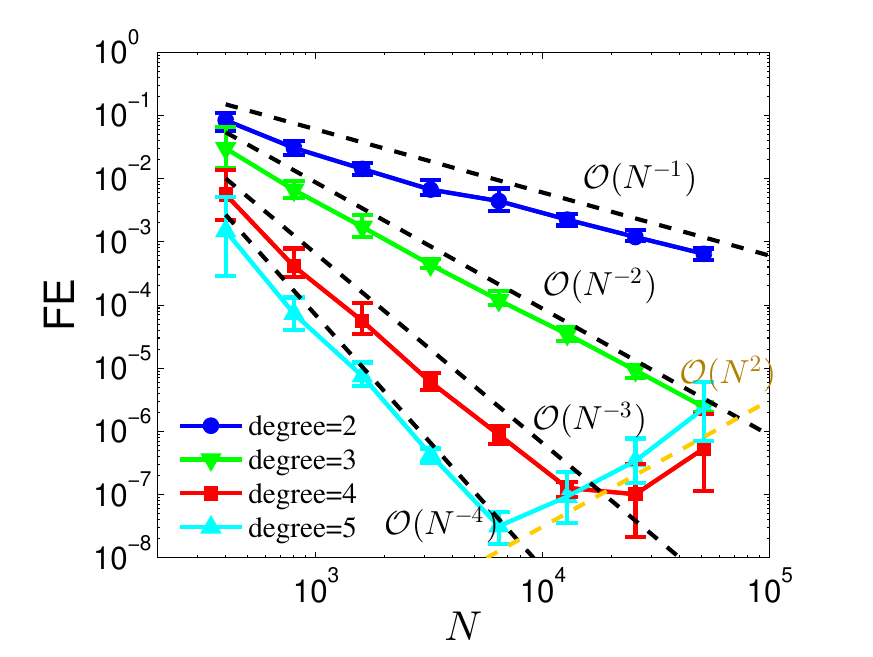} &
\includegraphics[width=3
in, height=2.2 in]{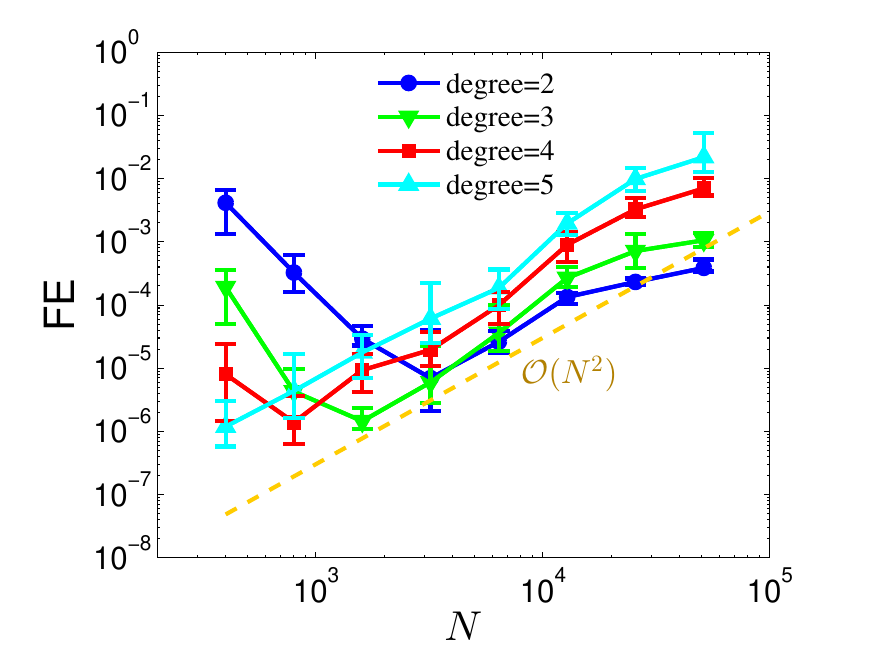} %
\end{tabular}
}
\caption{\textbf{1D ellipse in $\mathbb{R}^2$}. $K=21$ nearest neighbors are used. Comparison of
forward errors of the operator estimation using (a) intrinsic polynomials and
(b) extrinsic polynomials.  }
\label{fig1_cons}
\end{figure*}

\subsection{Stabilization based on optimization}\label{stabilization_section}

\comment{\color{red}To Qile and Shixiao, I don't think I understand the constraints optimization method in Hongkai Zhao's paper well. I am just confused about the consistency constraints: In our notation, their consistency constraints (or Eq. 3.14) looks to me like,
\[
\sum_{k=1}^K\hat{w}_k p_{\underline{\textbf{x}},\alpha}(\underline{\textbf{x}_k}) =  \Delta_M p_{\underline{\textbf{x}},\alpha} (\underline{\textbf{x}}),
\]
but I am not sure. In any case, this section below needs to be rewritten. Maybe we discuss the quadratic optimization approach in Hongkai Zhao's paper first. The approach below is to fit to the GMLS estimator with additional constraints.
}

{\bf{A quadratic optimization approach}.}\label{Sec:quadg_opt} While we have the consistency for the operator estimation using the intrinsic polynomial least squares, unfortunately, the resulting discrete operator $L_X$ is not stable when sample points are random on manifolds. As an example, let us examine the stability of the numerical Example~\ref{ex1}. In Fig.~\ref{fig1_stab2}(c), one can see that the infinity norm of the $L_{X,I}^{-1}$ grows dramatically as a function of $N$. Here we investigate the norm of $L_{X,I}^{-1}:=(I-L_{X})^{-1}$ instead of $L_{X}^{-1}$ since $L_{X}$ involves one zero eigenvalue on manifolds without boundaries. This instability issue, which will be problematic in the application of numerical PDEs, has been pointed out in a related work \cite{liang2013solving}. Motivated by the fact that the GMLS is equivalent to a constraint quadratic programming (see e.g. \cite{mirzaei2012generalized}), Liang and Zhao \cite{liang2013solving} additionally impose a diagonally-dominant constraint to the resulting approximation to ensure the stability of the resulting approximation. In our context, the proposed approach is to solve
\BEA
\min\limits_{\hat{w}_1,...,\hat{w}_k} \sum_{k=1}^K\hat{w}_k^2
\label{eqn:quad_optm}
\EEA
for $\{\hat{w}_i\}_{i=1,\ldots, K}$ with constraints
\BEA
\begin{cases}
\sum_{k=1}^K\hat{w}_kp_{\textbf{x}_0,\alpha}(\textbf{x}_{0,k})= \sum_{k=1}^Kw_kp_{\textbf{x}_0,\alpha}(\textbf{x}_{0,k})
\approx \Delta_M p_{\textbf{x}_0,\alpha} (\textbf{x}_0),\ \ |\alpha|\leq l, \\
\hat{w}_1<0, \hat{w}_k\geq 0,k=2,...,K,
\end{cases}\label{eqn:quad_const}
\EEA
for each base point $\mathbf{x}_0 \in X$.
Here, the first constraint in (\ref{eqn:quad_const}) guarantees that the new weights $\hat{w}_k$ are chosen such that the estimate is consistent to {\color{black}the least-squares} estimate at least for the intrinsic polynomials of degree up to $l$. Since the first basis polynomial $p_{\textbf{x}_0,\mathbf{0}}$ is constant $1$, then $\sum_{k=1}^K\hat{w}_k=\sum_{k=1}^Kw_k=\Delta_g 1=0$. With this equality and the second constraint in (\ref{eqn:quad_const}), the resulting weight is diagonal dominant, i.e., $|\hat{w}_1|\geq\sum_{k=2}^K|\hat{w}_k|$. Incidentally, we should point out that the slight difference between the quadratic optimization in (\ref{eqn:quad_optm})-(\ref{eqn:quad_const}) and the approach in  \cite{liang2013solving} is that they used a different approximation of the differential operator. Specifically, they used GMLS for the local quadratic approximation of manifolds and then calculate $\Delta_M p_{\textbf{x}_0,\alpha} (\textbf{x}_0)$ which involves the derivative of manifold metrics \cite{liang2013solving}. Here, in our context (\ref{eqn:quad_optm})-(\ref{eqn:quad_const}), we only need the pointwise information of the tangent space in order to compute the weights $\{w_k\}_{k=1}^K$ for the Laplacian approximation using the formula in \eqref{eqn:GlG1}.

In the numerical Example~\ref{ex1}, we tested this approach using the MATLAB built-in function
{\texttt quadprog.m}
to solve the quadratic optimization (\ref{eqn:quad_optm})-(\ref{eqn:quad_const}). For polynomials with $l=2$ and $l=3$, one can compare {\color{black}the least-squares weights} $w_k$ with those obtained by the quadratic programming as shown in Fig.~\ref{fig1_stab1}(a)-(b). One can see that {\color{black}the least-squares} weights are almost uniformly arranged within a certain range (magenta). From Fig.~\ref{fig1_stab1}(a)-(b), one can see that after quadratic programming, the weights increase dramatically for the base point ${\textbf{x}_0}$ and also its two or three neighbors. Besides, most weights become zeros when degree $l=2$ (blue curves in Fig.~\ref{fig1_stab1}(a)) and nonzero weights slightly increase when $l=3$
(blue curves in Fig.~\ref{fig1_stab1}(b)). Notice that the discrete Laplacian matrix is diagonally dominant (DD) after quadratic programming. However, the quadratic programming fails to find the solutions for higher-order polynomials due to the empty feasible set. This observation is analogous to that for the finite-difference approximation on 1D uniform grids, that is, for higher-order approximation of the Laplacian, the weights cannot be diagonally dominant but have to show oscillatory pattern with both positive and negative values for two-sided weights (see e.g. Table 4.1 in \cite{gustafsson2007high}).


\begin{figure*}[htbp]
{\scriptsize \centering
\begin{tabular}{cc}
{\normalsize (a) $l=2$} & {\normalsize (b) $l=3$}  \\
\includegraphics[width=2.8
in, height=2 in]{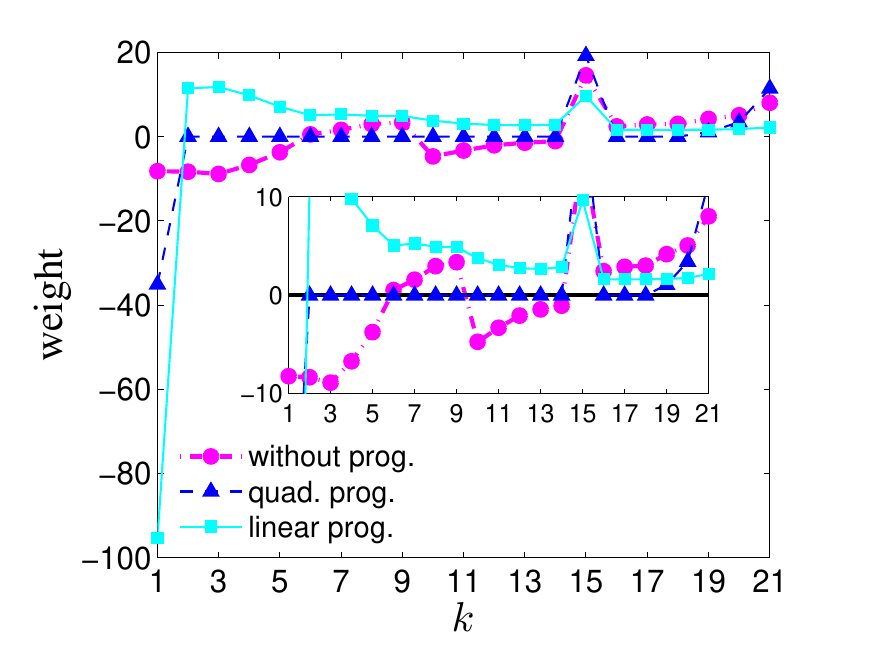} &
\includegraphics[width=2.8
in, height=2 in]{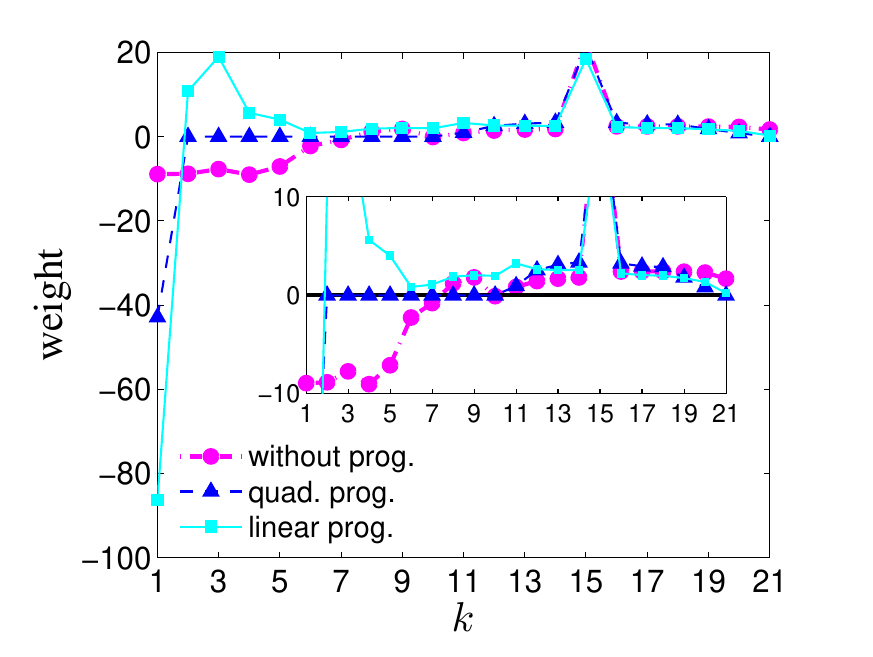} \\
{\normalsize (c) $l=4$ \& $l=5$} & {\normalsize (d) object function $C$} \\
\includegraphics[width=2.8
in, height=2 in]{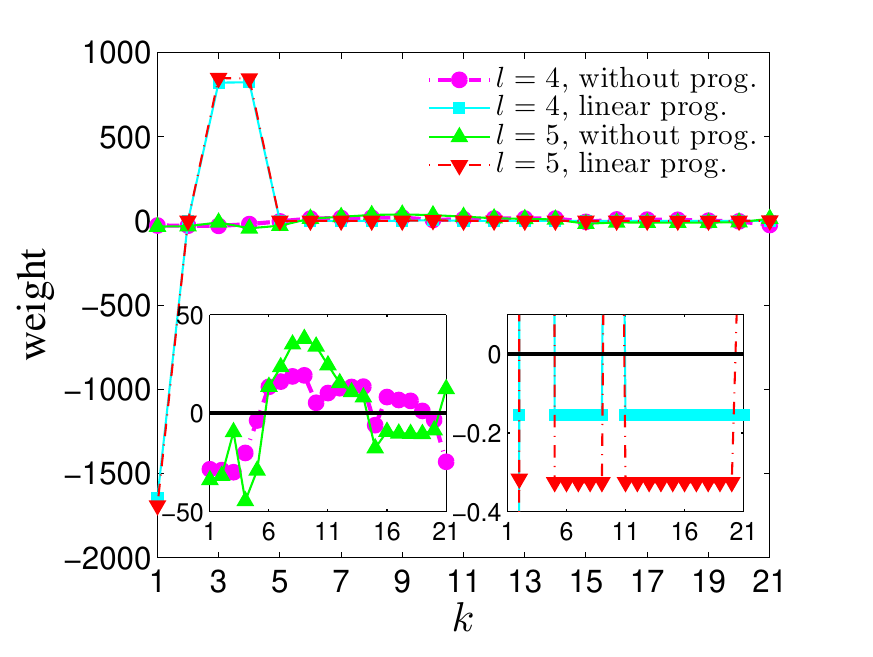} &
\includegraphics[width=2.8
in, height=2 in]{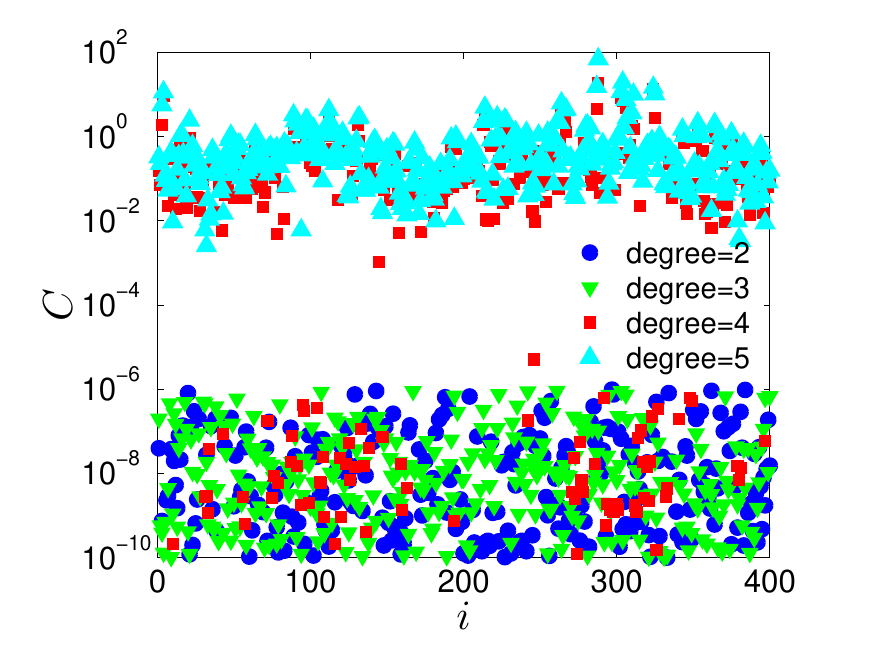} %
\end{tabular}
}
\caption{\textbf{1D ellipse in $\mathbb{R}^2$}. $N=400$. $K=21$ nearest
neighbors are used. Comparison of
weights among without optimization, quadratic programming, and linear programming for (a) $l=2$, (b) $l=3$, and
(c) $l=4,5$ at one point. The horizontal axis corresponds to the $k$-nearest neighbors ordered from the closest point to the farthest point. (d) The objective function values $C$ at each point $\{\textbf{x}_i\}_{i=1}^N$ for all degrees.}
\label{fig1_stab1}
\end{figure*}

\begin{figure*}[htbp]
{\scriptsize \centering
\begin{tabular}{ccc}
{\normalsize (a) FE comparison} & {\normalsize (b) stability for linear prog.} &
{\normalsize (c) instability for w/o prog.} \\
\includegraphics[width=2.
in, height=1.5 in]{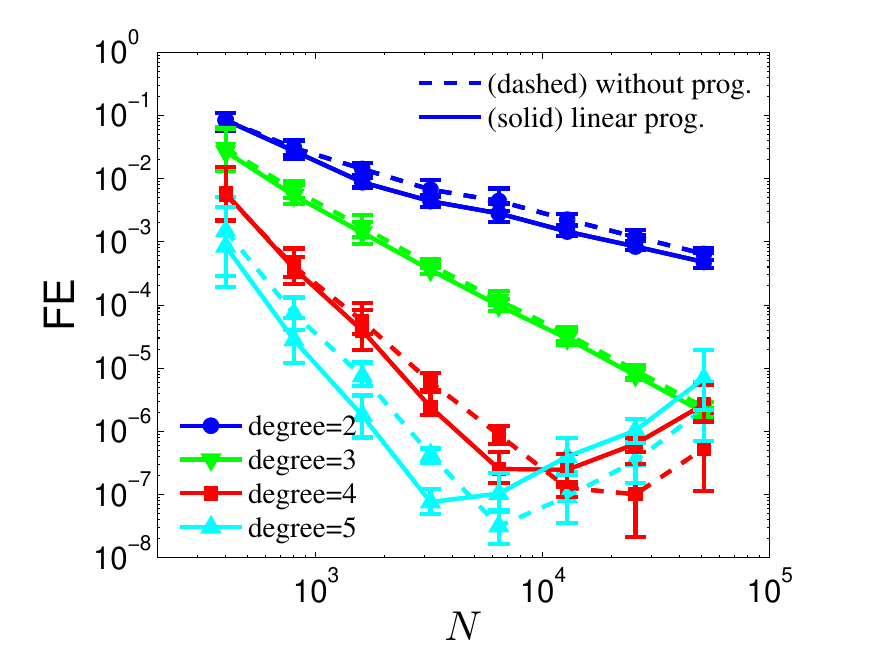} &
\includegraphics[width=2.
in, height=1.5 in]{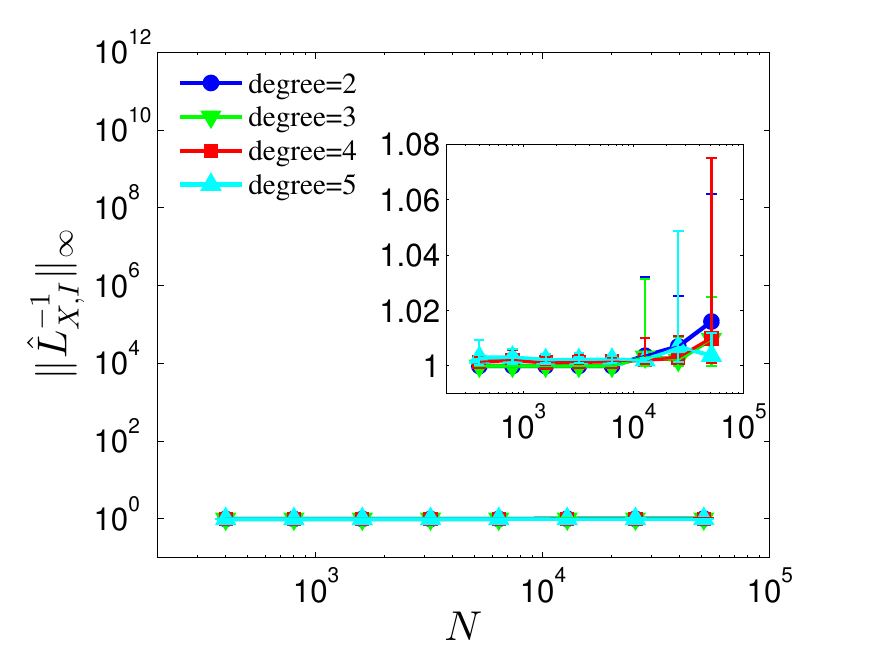} &
\includegraphics[width=2.
in, height=1.52 in]{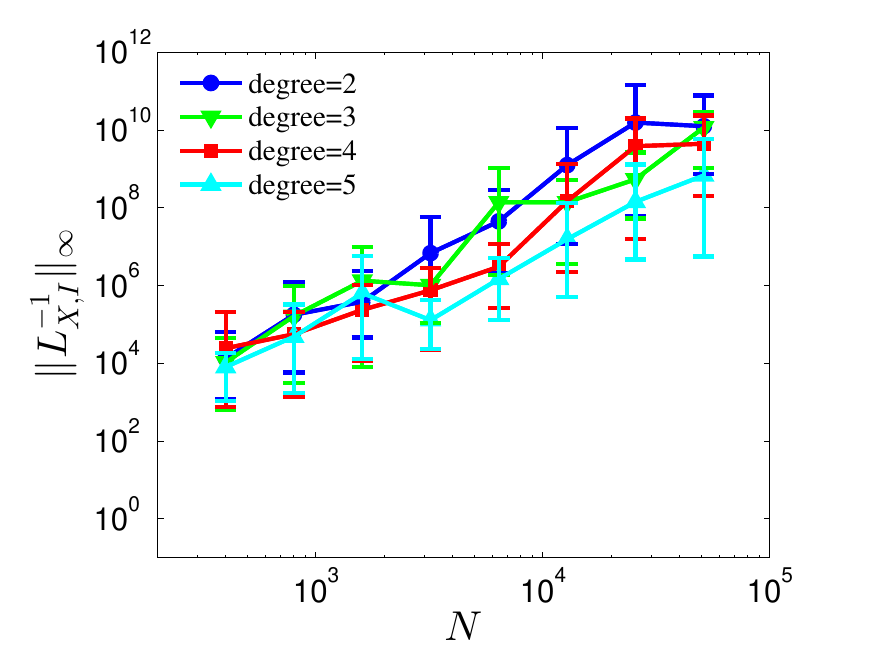}%
\end{tabular}
}
\caption{\textbf{1D ellipse in $\mathbb{R}^2$}. $K=21$ nearest
neighbors are used. (a) Comparison of
FEs between without optimization and linear programming. Comparison of the stability between (b) $\Vert \hat{L}_{X,I}^{-1}\Vert_{\infty}$ for linear programming and (c) $\Vert {L}_{X,I}^{-1}\Vert_{\infty}$ for without optimization.}
\label{fig1_stab2}
\end{figure*}

{\bf{A linear optimization approach}.}\label{Sec:linear_opt} To overcome the issue with the quadratic programming, we propose to relax the second diagonally-dominant constraint in (\ref{eqn:quad_const}). Specifically, instead of requiring $\hat{w}_k\geq 0$ for $k=2,...K$ in (\ref{eqn:quad_const}), we allow these coefficients to have a negative lower bound, i.e., there is a non-negative number $C\geq 0$ such that $w_k\geq -C,\, \forall\, k=2,...,K$. Since we hope the non-negative $C$ to be close to $0$, we set the objective function to be $C$. Consequently, we consider the following linear optimization problem:
\BEA
\min\limits_{C,\hat{w}_1,...,\hat{w}_k}C
\label{eqn:linear_optm}
\EEA
with constraints
\BEA
\begin{cases}
\sum_{k=1}^K\hat{w}_kp_{\textbf{x}_0,\alpha}(\textbf{x}_{0,k})= \sum_{k=1}^Kw_kp_{\textbf{x}_0,\alpha}(\textbf{x}_{0,k}),\ \ |\alpha|\leq l,\\
\hat{w}_1<0,\\
\hat{w}_k+C\geq 0,k=2,...,K,\\
0\leq C\leq \big\vert\min\limits_{k=2,...,K}w_k\big\vert,
\end{cases}\label{eqn:linear_const}
\EEA
where the last constraint is added to guarantee an existence of the solution under a mild assumption as stated below.

\begin{prop}\label{exist_lp}
If $w_1 <0$, then the linear optimization problem in \eqref{eqn:linear_optm}-\eqref{eqn:linear_const} has a solution.
\end{prop}

\begin{proof}
First, let us show that the feasible set is nonempty by examining that $\hat{w}_k = w_k$ and $C = |\min_{k=2,\ldots,K} w_k| \geq 0$ satisfy all the constraints. By the assumption, the first two constraints in \eqref{eqn:linear_const} are obviously satisfied. Let's check the third constraint. If $w_i \geq 0$ for any $i = 2,\ldots, K$, then of course $w_i+C \geq 0$. If $w_i< 0$ for some $i$, then $w_i + C = w_i + |\min_k w_k| =w_i + \max_k (-w_k) \geq 0$. Since the feasible set is nonempty and convex, also notice that $C$ is bounded, thus there exists a minimum value for $C$.
\end{proof}

In our implementation, we used the MATLAB built-in function
{\texttt{linprog.m}}
to solve the linear optimization \eqref{eqn:linear_optm}-\eqref{eqn:linear_const}.
We fixed $K=21$ in this Example~\ref{ex1}.
From Fig.~\ref{fig1_stab1}(a)-(c), one can see that after linear programming, the weights also increase dramatically for the center point ${\textbf{x}_0}$ as in the case of using the quadratic programming. When the degree $l=2$ and $l=3$ (cyan curves in Fig.~\ref{fig1_stab1}(a)(b)), all weights are nonnegative except for the center point.
Fig.~\ref{fig1_stab1}(d) shows that the values of $C$ are almost all zero (within a tolerance of $10^{-6}$) for all the data points (see the blue and green dots in Fig.~\ref{fig1_stab1}(d)) when $l=2$ and $l=3$. And thus, the discrete Laplacian matrix is diagonally dominant.
When degree $l$ increases to $4$ and $5$, many negative weights can be observed (see the cyan and red curves in Fig.~\ref{fig1_stab1}(c)). The values of $C$ are not zeros at various data points (see cyan and red dots in Fig.~\ref{fig1_stab1}(d)). This suggests that the discrete Laplacian matrix is not diagonally dominant.


In the remainder of this paper, we will denote the resulting matrix where each row corresponds to the weights $\hat{w}_k$ obtained from solving the linear programming as $\hat{L}_X$.
Note that in linear programming, consistency constraints are imposed only for our intrinsic polynomial basis functions $p_{\mathbf{x}_0,\alpha} \in \mathbb{P}_{\mathbf{x}_0}^{l,d}$. In Fig.~\ref{fig1_stab2}(a), one can
observe that the \textbf{FE} using linear programming become slightly smaller than
those without optimization (resulting from {\color{black}the least squares} alone) for Laplacian matrices acting on our test
function $u(\mathbf{x}(\theta))=\sin (\theta)$. From Fig.~\ref{fig1_stab2}(b), one can
see that the infinity norm of the inverse of the discrete Laplacian matrices obtained from linear programming, $\Vert \hat{L}_{X,I}^{-1}\Vert _{\infty }$, is nearly constant, which reflects a stable approximation. This empirical results show that the linear programming can stabilize the Laplacian matrix for $l=2 \sim 5$ whereas quadratic programming can only work in practice for $l=2,3$.

\begin{rem}
Diagonal dominance (or equivalently $C=0$) is only a sufficient condition for the matrix $\hat{L}_{X,I}$ to be stable. In the following Theorem~\ref{thm_closed}, we will prove the convergence under the assumption of $C=0$. When $C \neq 0$, we have no proof yet for the matrix stability. Now we  only numerically found that when $l=4,5$, the Laplacian matrices become stable after the linear optimization which gives $C \neq 0$. The basic idea behind the linear optimization is to make the negative non-diagonal weights as close to zero as possible, and meanwhile make absolute values of diagonal weights as large as possible, under the consistency constraints.
\end{rem}

\begin{rem}\label{rem_w1}
We should point out that the assumption of $w_1<0$ in Proposition~\ref{exist_lp} is only a sufficient (not a necessary) condition for the existence of a feasible set. Empirically, we found that in some cases the solution can be numerically estimated even when $w_1 \geq 0$. This is motivated by our numerical finding that $w_1<0$ can be attained for an appropriate choice of $K>0$ (nearest neighbor parameter) on interior points. In Fig.~\ref{dist_w1}(a)-(b), we label the data points uniformly distributed on a 2-dimensional semi-torus embedded in $\mathbb{R}^3$ (see details of the semi-torus in Section~\ref{sec5.3}) corresponding to $w_1<0$ in green color and $w_1 \geq 0$ in blue color (the figure is depicted in the intrinsic coordinates and the numerical example is for $N=3200$). For points near the boundary, $w_1\geq0$ for various choices of $K$. In panel (c), we also show the maximum distance to the boundary of points with $w_1\geq0$ (dashes), which suggests that these points are close to the boundary for $K\approx 20-50$. If we increase $K$, while $w_1\geq0$ persist on these points, we find that there are more points near the boundary having $w_1\geq0$. Moreover, the forward errors (FEs) increase as {\color{black}the least squares approximation} becomes less local when $K$ increases beyond 50.
\end{rem}

\subsection{Identification of points near boundary}\label{identify_boundary}

A further inspection suggests that $w_1>0$ holds when $K$-nearest-neighbor points are not equally distributed around the base point (or are only located on one side of the base point), which is highly likely to hold for points whose distance to the boundary is smaller than the separation distance or when $K$ is too small or too large. This observation is analogous to the positive weight $w_1>0$ induced by a one-sided finite-difference approximation of the second-order derivative (see e.g. Table 4.2 in \cite{gustafsson2007high}). While the positive weight of the classical finite-difference formulation can be explicitly verified by Taylor's expansion formula, it remains complicated to achieve the same result in our context since the specification of $w_1$ involves randomly sampled data, metric tensor of the manifold, $K$-nearest neighbors algorithm, and the local least-squares method.
Nevertheless, this empirical observation suggests that one can use the sign of $w_1$ with an appropriate choice of $K$ that gives accurate approximation (small FE) to detect points that are close to the boundary. It is important to point out that this way of detecting points close to the boundary does not need the information of where the boundary is, which is useful in practice since we sometimes have no sample points on the boundary as it is a measure zero set. In the following, we will employ this point detection strategy in solving boundary-value problems of PDEs when no boundary points are given on the boundary and even no location of the boundary is given (or equivalently only a point cloud is given).

\begin{figure*}[htbp]
{\scriptsize \centering
\begin{tabular}{ccc}
{\normalsize (a) $l=2,K=15$} & { \normalsize (b) $l=2, K=51$ } &{ \normalsize (c) distance \& FE}
\\
\includegraphics[width=0.31\linewidth]{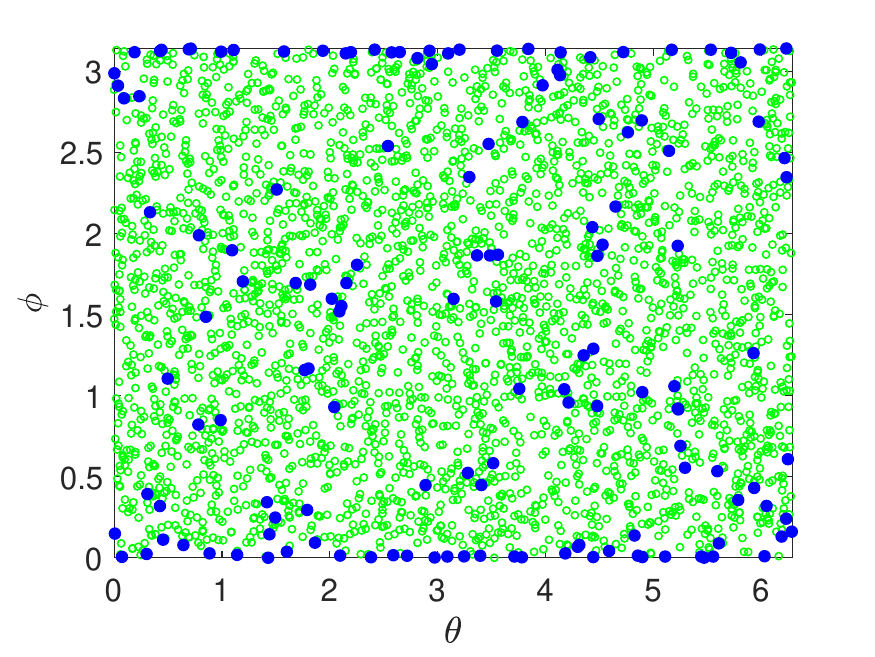} &
\includegraphics[width=0.31\linewidth]{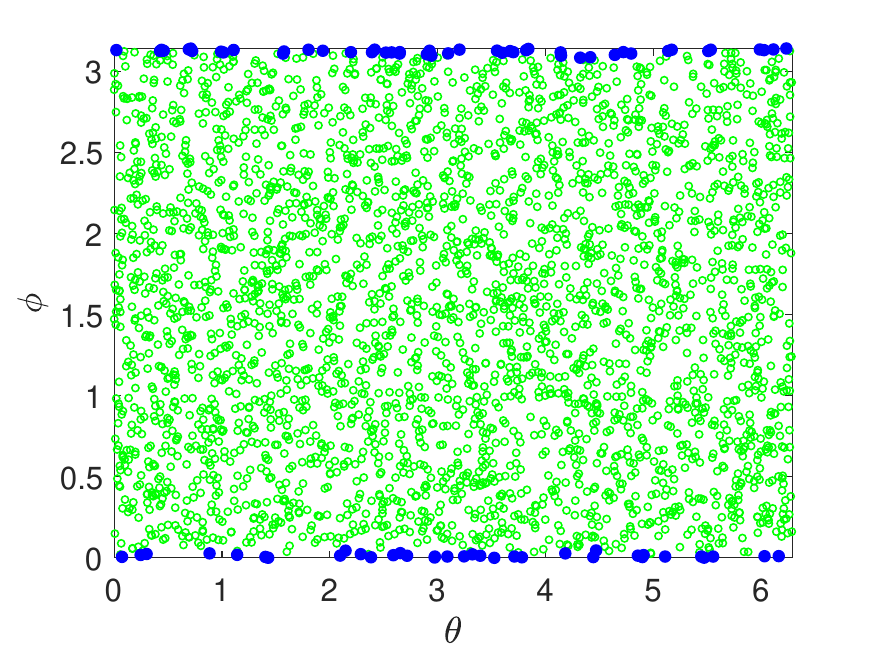}&
\includegraphics[width=0.31\linewidth]{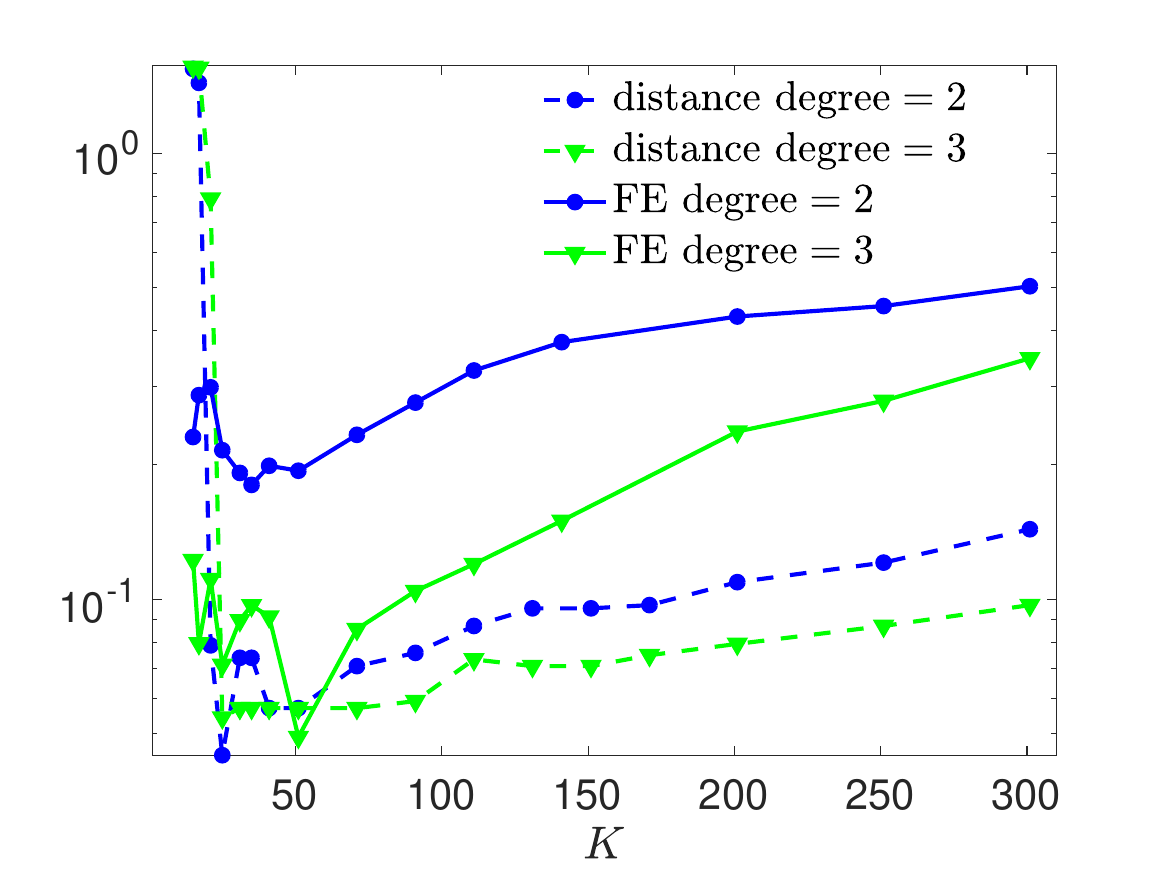}%
\end{tabular}
}
\caption{\textbf{2D semi-torus in $\mathbb{R}^3$}. $N=3200$. In panels (a)-(b), we identify points with $w_1>0$ and $w_1<0$ in blue and green, respectively. In panel (c), we show the maximum distance of points with $w_1>0$ to the boundary, $\max_{\mathbf{x}_i: w_1(\mathbf{x}_i)>0} d_g(\mathbf{x}_i,\partial M)$, in dashes, and the corresponding forward error (FE) in solid lines.}
\label{dist_w1}
\end{figure*}

\section{Application to solving Poisson problems}\label{pde}

In this section, we will discuss an application of the proposed discretization to solve two Poisson problems, one on closed manifolds and another one on compact manifolds with homogeneous Dirichlet boundary condition.

\subsection{Closed manifolds}
For a closed manifold, we consider the following PDE problem,
\BEA
(a -\Delta _{M})u = f, \quad x\in M, \label{closed_pde}
\EEA
where $a>0$ and $f$ are defined such that the problems are well-posed. Numerically, we will approximate the solution to the PDE problem in the pointwise sense on $X$, solving an $N\times N$ linear algebra problem
\BEA
\hat{L}_{X,A} \mathbf{U}  := (A- \hat{L}_X ) \mathbf{U} = \mathbf{f} = \mathcal{R}_X f,\label{closed_linproblem}
\EEA
for $\mathbf{U} \in \mathbb{R}^N$, where $A$ denotes a diagonal matrix with diagonal components $a(\mathbf{x}_i), i=1,\ldots, N$.

In this case, one can state the following convergence theorem.

\begin{thm}\label{thm_closed}
Let $u\in C^{l+1}(M^*)$ (with $l \geq 2$) be a classical solution of the PDE problem in \eqref{closed_pde} on a smooth $d-$dimensional closed manifold $M\subset \BR^n$. Let $X\subset M$ be i.i.d. uniform samples with $h_{X,M}\leq h_0$.  Let $\hat{L}_X$ in \eqref{closed_linproblem} be the approximate Laplacian matrix constructed from Algorithm~\ref{algo:local-LS} using appropriate $K$-nearest neighbors, degree-$l$ polynomials, and analytic projection matrix $\mathbf{P}$, followed by the Algorithm in \eqref{eqn:linear_optm}-\eqref{eqn:linear_const} for the linear optimization problem. Assume that the linear optimization problem in \eqref{eqn:linear_optm}-\eqref{eqn:linear_const} has a solution $C=0$, then with probability higher than {\color{black}$1-\frac{2}{N}$},
\[
\| \mathbf{U} - \mathbf{u} \|_\infty = O\left(N^{-\frac{l-1}{d}}\right),
\]
where $\mathbf{U}\in \BR^N$ solves \eqref{closed_linproblem} and $\mathbf{u} = \mathcal{R}_X u=(u(\mathbf{x}_1), \ldots,u(\mathbf{x}_N))^\top$. Here, the constant in the big-oh error bound is independent of $N$.
\end{thm}

\begin{proof}
See Appendix~\ref{app:A}.
\end{proof}

\begin{rem}\label{rem_assumption}
Uniform samples are assumed for $X$ just for the convenience of proof, while the result can be generalized to continuous sampling density with lower and upper bounds. In fact, in our numerical implementation for Example~\ref{ex1}, the random data points are not uniformly distributed on the ellipse. In particular, those points are generated from a uniform distribution in the intrinsic coordinate and then are mapped through the embedding onto the ellipse.
\end{rem}
\begin{rem}
We assume $C^\infty$ manifold for the Taylor expansion to be valid for convenience. In practice, the regularity of manifolds needs to be at least equal to the regularity of functions defined on $M^*$ in order to achieve the desired convergence rates.
\end{rem}

We now present the numerical results of solving elliptic PDEs in \eqref{closed_pde} on the 1D ellipse
in Example~\ref{ex1}{\color{black}, where the manifold is assumed to be known (i.e., we used the analytic projection matrix $\mathbf{P}$).} Setting $a=1>0$ in \eqref{closed_pde} and $u(\mathbf{x}(\theta))=\sin(\theta)$, the manufactured $f$ can be calculated in local coordinates using (\ref{eqn:Lap_ellipse}). In Fig.~\ref{fig1_convg}, we plot the error of the solutions, $%
\Vert \mathbf{U}- \mathbf{u} \Vert _{\infty }$, which we refer to as
the Inverse Error (\textbf{IE}), as a function of $N$. One can see that the solutions
are convergent  for the Laplacian matrix with linear optimization (Fig.~\ref{fig1_convg}%
(a)) whereas the \textbf{IE} obtained directly from {\color{black}the least squares approximation (without optimization)} do not decay (Fig.~\ref{fig1_convg}(b)). The reason for this result is that the Laplacian matrices become stable
after optimization whereas those matrices obtained directly from {\color{black}the least squares approximation} are not stable
as seen in Fig.~\ref{fig1_stab2}(c). The round-off limit of $O(N^{2})$ can be
seen clearly in Fig.~\ref{fig1_convg}(a). Notice, however, that the convergence rates
for polynomials of even degrees are not consistent (faster than) with the consistency rates; IEs decay
on $\mathcal{O}(N^{-2})$ and $\mathcal{O}(N^{-4})$\ for degree$-2$ and $-4$, respectively, which is one order faster than the corresponding FEs. While this phenomenon, which is known as super-convergence, has been observed on quasi-uniform data and explained with a symmetry argument in \cite{liang2013solving}, it is unclear to us why it still persists even when the symmetry is not valid when the training data are randomly sampled. This leaves room for future investigations.


\begin{figure*}[htbp]
{\scriptsize \centering
\begin{tabular}{cc}
{\normalsize (a) with linear programming} & {\normalsize (b) without optimization}
\\
\includegraphics[width=3
in, height=2.2 in]{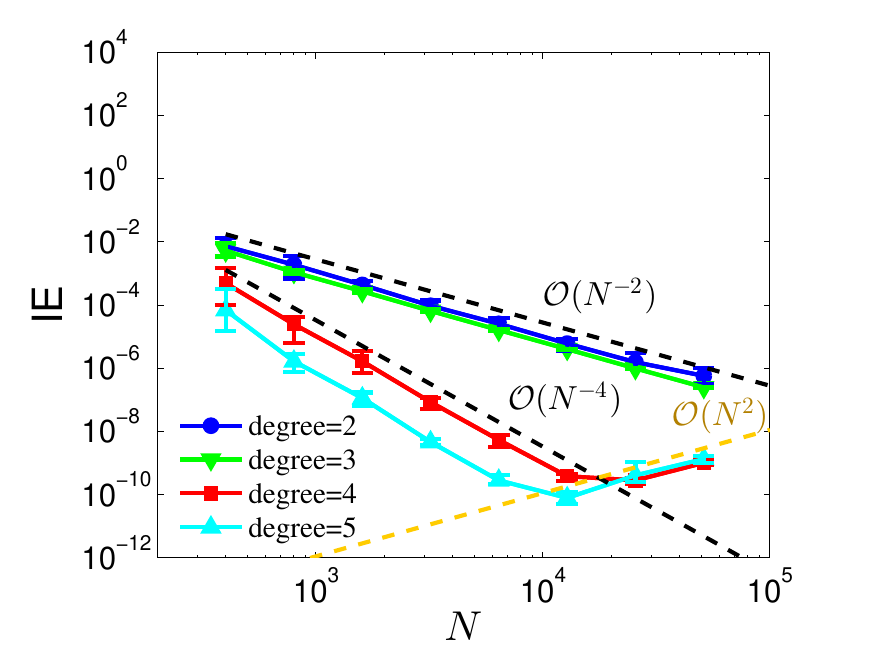} &
\includegraphics[width=3
in, height=2.2 in]{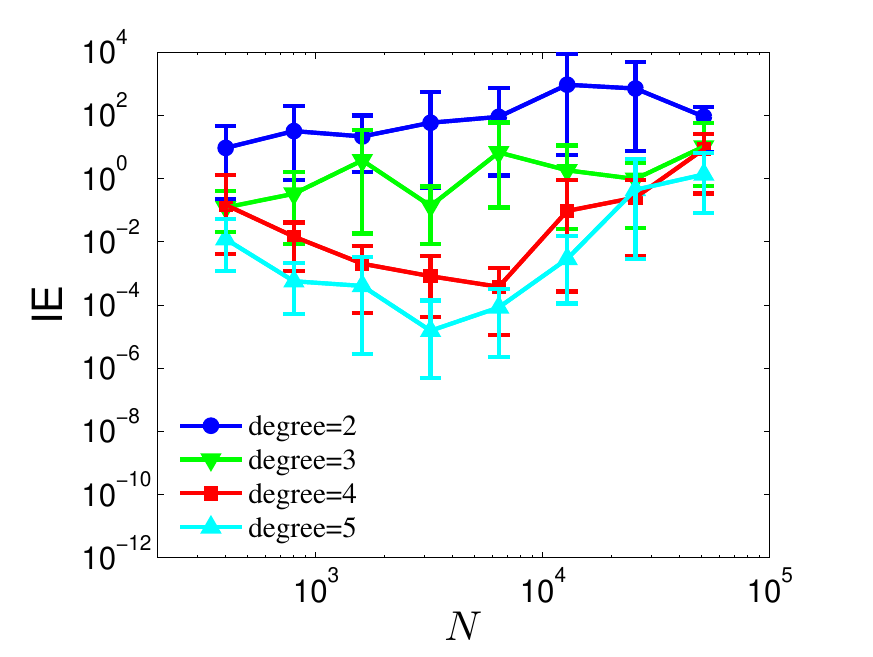}%
\end{tabular}
}
\caption{\textbf{1D ellipse in $\mathbb{R}^2$}. $K=21$ nearest
neighbors are used. Comparison of
inverse errors (IE) of solutions between (a) with linear programming and (b) without optimization. }
\label{fig1_convg}
\end{figure*}

In Appendix~\ref{app:B}, we will further show another numerical example of a 2D semi-torus in order to examine the convergence rates, where sufficiently many points are randomly sampled exactly on the boundary of the semi-torus. In that example, we can also observe the super-convergence for randomly sampled data using our approach. In the next section, we will discuss the approach when points that lie exactly on the boundary are not available since the boundary is a volume-measure zero set.

\subsection{Dirichlet boundary condition}

We also consider the Poisson problem on a compact manifold with the homogeneous Dirichlet boundary condition,
\BEA
\begin{aligned}\label{diri_pde}
\Delta _{M}u &= f, \quad \mathbf{x}\in \textup{int}(M),  \\
u &= 0, \quad  \mathbf{x}\in \partial M,
\end{aligned}
\EEA
where $f$ is defined such that the problem is well-posed. Here, we consider the regime that we have no
sample points on the boundary as the boundary is a measure zero set (see the sketch in Fig.~\ref{figtruncate}(a)).
That is, let $X=\{\mathbf{x}_i\}_{i=1}^N$ be randomly sampled data points on $\textup{int}(M)$.
Numerically, we will follow the volume-constraint approach by imposing the boundary conditions to points near the boundary \cite{du2012analysis,mengesha2014bond,shi2017enforce}.

\subsubsection{Completely unknown manifolds}\label{sec:comunk}

In this subsection, we assume that we are only given a point cloud on the interior of the domain and we do not know the location of the boundary.
For this unknown manifold setup, we will estimate the tangent spaces
using a second-order local SVD method introduced in our previous work \cite{harlim2022rbf}. One can, of course, use the classical first-order local SVD method \cite{donoho2003hessian,zhang2004principal,tyagi2013tangent}. Subsequently, we used the approximate tangent spaces to estimate the projection matrix $\mathbf{P}$ in Algorithm \ref{algo:local-LS}. While one can estimate the distance of each sample point to the boundary (e.g., see the method proposed in \cite{berry2017density}), we will introduce an approach that does not require knowledge or estimation of the distance of each sample point to the boundary. In particular, we will identify the points that are near the boundary as those correspond to $w_1>0$ as we pointed out in Remark~\ref{rem_w1}. Denoting,
\BEA
Y = \big\{\mathbf{x}\in X | w_1(\mathbf{x}) <0\big\} \subset \textup{int}(M),\label{setY}
\EEA
the volume-constraint approach to the Dirichlet problem above yields to solving the following linear algebra problem,
\BEA
 \hat{L}_Y  \mathbf{U}_1 &=& \mathcal{R}_{Y} f,\label{diri_linproblem}\\
 \mathbf{U}_2 &=& \mathbf{0},\notag
\EEA
for $\mathbf{U}_1$ that approximates $\mathcal{R}_Y u$  and $\mathbf{U}_2$ that approximates $\mathcal{R}_{X\backslash Y} u$. Here, $\hat{L}_Y \in \mathbb{R}^{|Y| \times |Y|}$ is a square matrix which takes the rows and columns of $\hat{L}_X$ corresponding to points in $Y$, where $|Y|$ denotes the number of points in $Y$. In Fig.~\ref{figtruncate}(a), green points in $Y$ are interior points on which differential equations are satisfied. Blue points in $X\backslash Y$ are treated as artificial boundary points on which the volume constraint condition $\mathbf{U}_2=\mathbf{0}$ is imposed.

\subsubsection{Knowing the location of the boundary}

In the case when the location of the boundary is known, we can deduce the following theoretical result. For the convenience of the analysis below, we define
\[
M_\epsilon = \big\{\mathbf{x}\in M : d_g(\mathbf{x},\partial M)> \epsilon\big\} \subset \textup{int}(M),
\]
to be the subset of points in $M$  whose geodesic distance from the boundary greater than $\epsilon>0$. Here, we assume that we know the location of the boundary $\partial M$ so that $d_g (\mathbf{x},\partial M)$ is computable for any $\mathbf{x} \in M$. Subsequently, we define $X_\epsilon = X\cap M_\epsilon$ to be the set of data points that are in $M_\epsilon$ (see Fig.~\ref{figtruncate}(a)). For convenience, we define $N_1 = |X_\epsilon |$ and $N_2 = |X\backslash X_\epsilon|$ such that $N=N_1+N_2$ and correspondingly define $\mathbf{u}_1 = \mathcal{R}_{X_\epsilon} u \in \mathbb{R}^{N_1}$ and $\mathbf{u}_2 = \mathcal{R}_{X\backslash X_\epsilon} u \in \mathbb{R}^{N_2}$.

 Notice that if $\epsilon>0$ is chosen such that,
\BEA
\epsilon^* := \max_{\mathbf{x}_i \in X\backslash Y} d_g(\mathbf{x}_i, \partial M), \label{eps_alg}
\EEA
and no points in $X\backslash X_{\epsilon^*}$ satisfy  $w_1 < 0$, i.e., $X\backslash X_{\epsilon^*} \cap Y =\emptyset$, then $X_{\epsilon^*}=Y$. Treating all points in $X \backslash X_\epsilon$ as artificial boundary points, the volume-constraint approach is to solve
\BEA
 \hat{L}_{X_\epsilon}  \mathbf{U}_1 &=& \mathcal{R}_{X_\epsilon} f,\label{diri_linproblem2}\\
 \mathbf{U}_2 &=& \mathbf{0},\notag
\EEA
for $\mathbf{U}_1\in \mathbb{R}^{N_1}$ and $\mathbf{U}_2 \in \mathbb{R}^{N_2}$ approximating $\mathbf{u}_1$  and $\mathbf{u}_2$, respectively. Here, $\hat{L}_{X_\epsilon} \in \mathbb{R}^{N_1\times N_1}$ is the square Laplacian matrix.

With these definitions, we can prove the following lemma.
\begin{lem}\label{consistency}
Let  $u\in C^{l+1}(M^*)$ be a classical solution of the PDE problem in \eqref{diri_pde} on a smooth compact $d-$dimensional manifold $M\subset \BR^n$ with smooth boundaries. Let $X\subset M$ be i.i.d. uniform samples with $h_{X,M}\leq h_0$. Let the Laplacian matrix correspond to the analytic projection matrix $\mathbf{P}$ and the solution of the linear optimization problem in \eqref{eqn:linear_optm}-\eqref{eqn:linear_const} restricted on $X_\epsilon$ as $[\hat{L}_{X_\epsilon}, \hat{L}_{X \backslash X_\epsilon}] \in \mathbb{R}^{N_1 \times (N_1+N_2)}$. Assume that $C=0$ for all points in $X_\epsilon$, then with probability higher than $1-\frac{2}{N}$,
\[
\left\Vert [\hat{L}_{X_\epsilon}, \hat{L}_{X\backslash X_\epsilon}] \left(
\left[
\begin{array}{c}
\mathbf{U}_{1} \\
\mathbf{U}_{2}%
\end{array}%
\right] -
\left[
\begin{array}{c}
\mathbf{u}_{1} \\
\mathbf{u}_{2}%
\end{array}%
\right]
\right) \right\Vert_\infty = O\left(N^{-\frac{l-1}{d}}\right),
\]
where $(\mathbf{U}_1,\mathbf{U}_2)^\top$ solves \eqref{diri_linproblem2} and $\mathbf{u}=\mathcal{R}_X u=(\mathbf{u}_1,\mathbf{u}_2)^\top$ is the solution.  Moreover, choosing $\epsilon \sim h_{X,M}$, then with probability higher than $1-\frac{1}{N}$,
\[
\|\mathbf{U}_2 - \mathbf{u}_2\|_\infty = \|\mathbf{u}_2\|_\infty = O\left(N^{-\frac{1}{d}}\right).
\]
\end{lem}

\begin{proof}
The proof of the first part follows the same argument as the proof of Theorem~\ref{thm_closed}. As for the second part, since $u = 0$ at the boundary, then by the fact that $\overline{M}\subset M^*$ is compact,  the regularity suggests that $u\in C^{0,1}(M)$ such that,
\[
|u(\mathbf{x})| \leq L \epsilon, \quad \forall \mathbf{x}\in M\backslash M_\epsilon,
\]
where $L$ denotes the Lipschitz constant. Setting $\epsilon \sim h_{X,M}$ and using the bound in \eqref{filldist_bdd} and the fact that $X\backslash X_\epsilon \subset M\backslash M_\epsilon$, the proof is complete.
\end{proof}

\begin{rem}
Theoretically, we assumed the analytic $\mathbf{P}$ in above Lemma~\ref{consistency} for convenience of analysis. If the approximate $\mathbf{\hat{P}}$ is used, we need to account an additional error term for the error estimation from $\mathbf{\hat{P}}$ (see e.g. Theorem~3.2 in \cite{harlim2022rbf}
for the second-order local SVD method that will be used in the numerical experiments in Section~\ref{numerics}).
\end{rem}

\begin{rem}\label{rem4.3}
The motivation for the result above is that by volume constraint, we are estimating $\mathbf{u}_2$ with $\mathbf{U}_2=0$ and the error will dominate if {\color{black}the least squares approximation} is employed with polynomials of order $l>2$. The estimation involves truncating all the points whose distance from the boundary is proportional to the fill distance $h_{X,M}$ which is larger than the separation distance $q_{X,M}\geq c(d)N^{-\frac{2}{d}}$ as noted in \eqref{filldist_bdd}. This choice of $\epsilon$, unfortunately, does not theoretically guarantee that all points in $X\backslash X_\epsilon$ satisfies $w_1 \geq 0$. As we stated before in Remark~\ref{rem_w1}, we only numerically observe that $w_1>0$ for any choice of $K$ (nearest neighbor parameter) for points near the boundary.
\end{rem}

To prove the convergence, we first show:
\begin{lem}\label{dis_max_prin}(Discrete maximum principle)
Let the Laplacian matrix corresponding to the solution $C=0$ of the linear optimization problem in \eqref{eqn:linear_optm}-\eqref{eqn:linear_const} restricted on $X_\epsilon$ be denoted as  $[\hat{L}_{X_\epsilon}, \hat{L}_{X \backslash X_\epsilon}] \in \mathbb{R}^{N_1 \times N}$, then a discrete maximum principle is satisfied in the following sense. Let $\mathbf{v}\in\mathbb{R}^N$ be such that $[\hat{L}_{X_\epsilon}, \hat{L}_{X \backslash X_\epsilon}]\mathbf{v} \geq 0$, then
\[
\max\limits_{j:\mathbf{x}_j\in \overline{X_\epsilon}} \mathbf{v}_j = \max\limits_{\ell: \mathbf{x}_\ell\in \mathbb{X}_{K,\epsilon}}  \mathbf{v}_\ell,
\]
where $\overline{X_\epsilon} = X_\epsilon \cup \mathbb{X}_{K,\epsilon} \subset X$ and
\[
\mathbb{X}_{K,\epsilon} = \Big\{\mathbf{x}_{i,k} \in X\backslash X_\epsilon \big{|} 1<k\leq K, \textup{ for } \mathbf{x}_i = \mathbf{x}_{i,1} \in X_\epsilon \Big\} ,
\]
is the set of $K$-nearest neighbors of any points $\mathbf{x}_i \in X_\epsilon$ that are in $X\backslash X_\epsilon$.
\end{lem}

\begin{proof}
Since,
$
\left([\hat{L}_{X_\epsilon}, \hat{L}_{X \backslash X_\epsilon}]\mathbf{v}\right)_i = \sum_{k=1}^K \hat{w}_k \mathbf{v}_{i,k},
$
and $C=0$, the constraints in \eqref{eqn:linear_const} suggest that $[\hat{L}_{X_\epsilon}, \hat{L}_{X \backslash X_\epsilon}]$ is diagonally dominant.
Here, $\mathbf{v}_{i,k}$ takes the value on point $\mathbf{x}_{i,k}$. By the assumption $[\hat{L}_{X_\epsilon}, \hat{L}_{X \backslash X_\epsilon}]\mathbf{v} \geq 0$, we have,
\BEA
- \hat{w}_1 \mathbf{v}_i =- \left([\hat{L}_{X_\epsilon}, \hat{L}_{X \backslash X_\epsilon}]\mathbf{v}\right)_i + \sum_{k=2}^K \hat{w}_k \mathbf{v}_{i,k} \leq \sum_{k=2}^K \hat{w}_k \mathbf{v}_{i,k}. \label{inequal}
\EEA
Since $0\leq -\frac{\hat{w}_k}{\hat{w}_1}$ and $-\sum_{k=2}^K \frac{\hat{w}_k}{\hat{w}_1} = 1$, this inequality suggests that $\mathbf{v}_i$ is bounded above by a weighted average of $\mathbf{v}$ evaluated on its $K$-nearest neighbors. This means that if  $\mathbf{v}_i$ is a strict maximum at $\mathbf{x}_i$, the following inequality $\mathbf{v}_i > \mathbf{v}_{i,k}$ for some $k=2,\ldots, K$ contradicts \eqref{inequal}. In fact, if the maximum occurs on the interior points in $X_\epsilon$, then $\mathbf{v}_i = \mathbf{v}_{i,k}$, for all $k\geq 2$, so every nearest neighbors points are also maxima. Continuing this way on all points, we conclude that if the maximum occurs on the interior point, then $\mathbf{v}$ is constant over points in $\overline{X_\epsilon}$ and its maximum is also attained at its nearest neighbor that is in $\mathbb{X}_{K,\epsilon} \subset X\backslash X_\epsilon$ and the proof is complete.
\end{proof}

\begin{cor}\label{cor:min}
If $[\hat{L}_{X_\epsilon}, \hat{L}_{X \backslash X_\epsilon}]\mathbf{v} \leq 0$, then
$$
\min\limits_{j:\mathbf{x}_j\in \overline{X_\epsilon}} \mathbf{v}_j = \min\limits_{\ell: \mathbf{x}_\ell\in \mathbb{X}_{K,\epsilon}}  \mathbf{v}_\ell.
$$
\end{cor}

\begin{thm}[Convergence]\label{convrate_diri}
Let $u\in C^{l+1}(M^*)$ be a classical solution of the PDE problem in \eqref{diri_pde} where $\Delta_M$ satisfies the maximum principle.
Let the same hypothesis as in Lemma \ref{consistency} hold true.
For $\epsilon\sim h_{X,M}$, with probability higher than $1-\frac{5}{N}$,
\[
\| \mathbf{U} - \mathbf{u} \|_\infty = O\big(N^{-\frac{l-1}{d}}\big) + O\big(N^{-\frac{1}{d}}\big),
\]
where $\mathbf{U} = (\mathbf{U}_{1},\mathbf{U}_{2})^\top\in \BR^N$ solves $\hat{L}_{X_\epsilon}\mathbf{U}_1 = \mathcal{R}_{X_\epsilon}f$ with $\mathbf{U}_{2} = 0$.
\end{thm}

\begin{proof}
Let $v\in C^{\infty }(M)$ that solves $\Delta _{M}v=-1$ for all $\mathbf{x}\in \mathrm{int}(M)$ and $v|_{\mathbf{x}\in \partial M}=0$. Here, the existence of the unique solution $v$ follows from the well-posedness
assumption of the Dirichlet problem. By maximum principle, $\min_{\overline{M}}v=\min_{\partial M}v=0$, that is, $v\geq 0$ for all $\mathbf{x}\in \overline{M}$. Denote $\mathbf{v}=\mathcal{R}_{X}v\in \mathbb{R}^{N}$. Using
the estimation in \eqref{errorrate}, for any $i$ with $\mathbf{x}_{i}\in X_{\epsilon }$, we have $\left\vert ([\hat{L}_{X_{\epsilon }},\hat{L}_{X\backslash X_{\epsilon }}]\mathbf{v})_{i}-\Delta _{M}v(\mathbf{x}_{i})\right\vert \leq c_{v}N^{-\frac{l-1}{d}}\leq \frac{1}{2},$ for sufficiently large enough $N$, with probability higher than $1-\frac{2}{N}$. That is,
\begin{equation}
([\hat{L}_{X_{\epsilon }},\hat{L}_{X\backslash X_{\epsilon }}]\mathbf{v}%
)_{i}\leq \Delta _{M}v(\mathbf{x}_{i})+\frac{1}{2}=-\frac{1}{2}.
\label{eqn:bv}
\end{equation}%
Using the estimation in \eqref{errorrate} again, for any $i$ with $\mathbf{x}%
_{i}\in X_{\epsilon }$,
\begin{equation}
\left\vert\left\{[\hat{L}_{X_{\epsilon }},\hat{L}_{X\backslash X_{\epsilon }}](\mathbf{U}-%
\mathbf{u})\right\}_{i}\right\vert
=
\left\vert f(\mathbf{x}_{i})-([\hat{L}_{X_{\epsilon }},\hat{L}_{X\backslash X_{\epsilon }}]\mathbf{u})_{i} \right\vert \leq c_{u}N^{-\frac{l-1}{d}},
\label{eqn:bu}
\end{equation}%
with probability higher than $1-\frac{2}{N}$. Let $F=2c_{u}N^{-\frac{l-1}{d}}
$ and $\Phi =\Vert \mathbf{u}_{2}\Vert _{\infty }$ be constant, and
construct a vector,%
\begin{equation*}
\mathbf{W}=F\mathbf{v}+\Phi \mathbf{I}\pm (\mathbf{U}-\mathbf{u})\in \mathbb{%
R}^{N},
\end{equation*}%
where $\mathbf{I}$ is the vector with all entries are $1$. Using %
\eqref{eqn:bv} and \eqref{eqn:bu} and noticing that $[\hat{L}_{X_{\epsilon
}},\hat{L}_{X\backslash X_{\epsilon }}]\mathbf{I=0}$, we have
\begin{eqnarray*}
\lbrack \hat{L}_{X_{\epsilon }},\hat{L}_{X\backslash X_{\epsilon }}]\mathbf{W%
} &=&F[\hat{L}_{X_{\epsilon }},\hat{L}_{X\backslash X_{\epsilon }}]\mathbf{v}%
\pm \lbrack \hat{L}_{X_{\epsilon }},\hat{L}_{X\backslash X_{\epsilon }}](%
\mathbf{U}-\mathbf{u}) \\
&\leq &-\frac{1}{2}F+c_{u}N^{-\frac{l-1}{d}}=0.
\end{eqnarray*}%
Notice that $\mathbf{v\geq 0}$ and $\mathbf{U}_{2}=\mathbf{0}$, then we have%
\begin{equation*}
\mathbf{W|}_{X\backslash X_{\epsilon }}=[F\mathbf{v}+\Phi \mathbf{I}\pm (%
\mathbf{U}-\mathbf{u})]\mathbf{|}_{X\backslash X_{\epsilon }}\geq \Phi \mathbf{I}\pm
\mathbf{u}_{2}\geq 0.
\end{equation*}%
This also suggests $\mathbf{W|}_{\mathbb{X}_{K,\epsilon }}\geq 0$ since $%
\mathbb{X}_{K,\epsilon }\subset X\backslash X_{\epsilon }$. Then, the discrete
maximum principle in Corollary~\ref{cor:min} gives us
\begin{equation*}
\min_{j:\mathbf{x}_{j}\in \overline{X_{\epsilon }}}[F\mathbf{v}+\Phi \mathbf{%
I}\pm (\mathbf{U}-\mathbf{u})]_{j}=\min_{j:\mathbf{x}_{j}\in \overline{%
X_{\epsilon }}}\mathbf{W}_{j}=\min_{\ell :\mathbf{x}_{\ell }\in \mathbb{X}%
_{K,\epsilon }}\mathbf{W}_{\ell }\geq 0.
\end{equation*}%
Thus,
\begin{equation*}
\Vert (\mathbf{U}-\mathbf{u})|_{\overline{X_{\epsilon }}}\Vert _{\infty
}\leq \Vert F\mathbf{v}+\Phi \mathbf{I}\Vert _{\infty }\leq 2c_{u}N^{-\frac{%
l-1}{d}}\Vert \mathbf{v}\Vert _{\infty }+\Vert \mathbf{u}_{2}\Vert _{\infty
}=O\big( N^{-\frac{l-1}{d}}\big) +O\big(N^{-\frac{1}{d}}\big),
\end{equation*}%
with probability higher than $1-\frac{5}{N}$, where we have used Lemma~\ref{consistency} in the last inequality for $\mathbf{u}%
_{2}$. Last, we consider the bound for $\mathbf{U}-\mathbf{u}$\ on the set $%
X\backslash \overline{X_{\epsilon }}=X\backslash (X_{\epsilon }\cup \mathbb{X%
}_{K,\epsilon })$, containing the points that are not connected to any
interior points in $X_{\epsilon }$\ through $K$ nearest neighbors. Notice
that $X\backslash \overline{X_{\epsilon }}\subset X\backslash X_{\epsilon
}\subset M\backslash M_{\epsilon }$, so that
\begin{equation*}
\Vert (\mathbf{U}-\mathbf{u})|_{X\backslash \overline{X_{\epsilon }}}\Vert
_{\infty }\leq \Vert (\mathbf{U}-\mathbf{u})|_{X\backslash X_{\epsilon
}}\Vert _{\infty }=\Vert \mathbf{u}_{2}\Vert _{\infty }=O(N^{-\frac{1}{d}}),
\end{equation*}%
where we again use Lemma~\ref{consistency}. Collecting the probability, the
proof is complete.
\end{proof}

As we discussed in Remark~\ref{rem4.3}, the error bound of the PDE solution will be dominated by the error induced by the volume-constraint approach, where the error is proportional to the truncation distance to the boundary, which we identify as $\epsilon>0$. Choosing $\epsilon\sim h_{X,M}$, this error will dominate if {\color{black}the least squares approximation} is employed with polynomials of degree-$l>3$.
The choice of $\epsilon\sim h_{X,M}$ also makes sense as it is an upper bound of the separation distance $q_{X,M}$ since if we choose $\epsilon$ to be less than the separation distance, then we are most likely not truncating any points.

\subsubsection{Numerical verification of the volume-constraint approach}


In Fig.~\ref{figtruncate}(b), we numerically demonstrate the convergence rate in solving the Dirichlet problem in \eqref{diri_pde} on a two-dimensional semi-torus embedded in $\mathbb{R}^3$ (see more details about the manifold and the PDE parameters in the following Example~\ref{sec5.3}). As the theory suggested, we expect an error rate of order-$N^{-\frac{1}{2}}$ in this two-dimensional example. The proposed approach, which employs the volume-constraint method by solving the linear problem \eqref{diri_linproblem} induced by restricting the PDE on $Y$ as defined in \eqref{setY} gives more accurate inverse error relative to applying the volume-constraint by restricting the PDE to $X_\epsilon$, where the distance $\epsilon$ is chosen to be proportional to $\epsilon^*$ defined in \eqref{eps_alg}. Setting the scaling $\alpha$ in  $\epsilon = \alpha\epsilon^*$ to be too large reduces the accuracy. On the other hand, when $\alpha$ is too small, we empirically cannot see the convergence rate possibly due to inaccurate results induced by too small of truncation distance (possibly below the separation distance). As we noted right after \eqref{eps_alg}, restricting on $X_{\epsilon^*}$ is not the same as restricting on the set $Y$ since there are possibly other points in $X\backslash X_{\epsilon^*}$ with $w_1<0$. So, restricting the PDE on $X_\epsilon^*$ yields a truncation of (possibly) more points than only restricting on the set $Y$. Besides the accuracy and consistent convergence rate, an important aspect of the proposed approach that only solving the PDE on the set $Y$ is that it does not require one to know or to estimate the distance of each data point from the boundary, which is significantly advantageous compared to the $\epsilon$-distance based approach that required such information.

\begin{figure*}[htbp]
{\scriptsize \centering
\begin{tabular}{cc}
{\normalsize (a) Sketch } & {\normalsize (b) convergence }  \\
\includegraphics[width=3in, height=2 in]{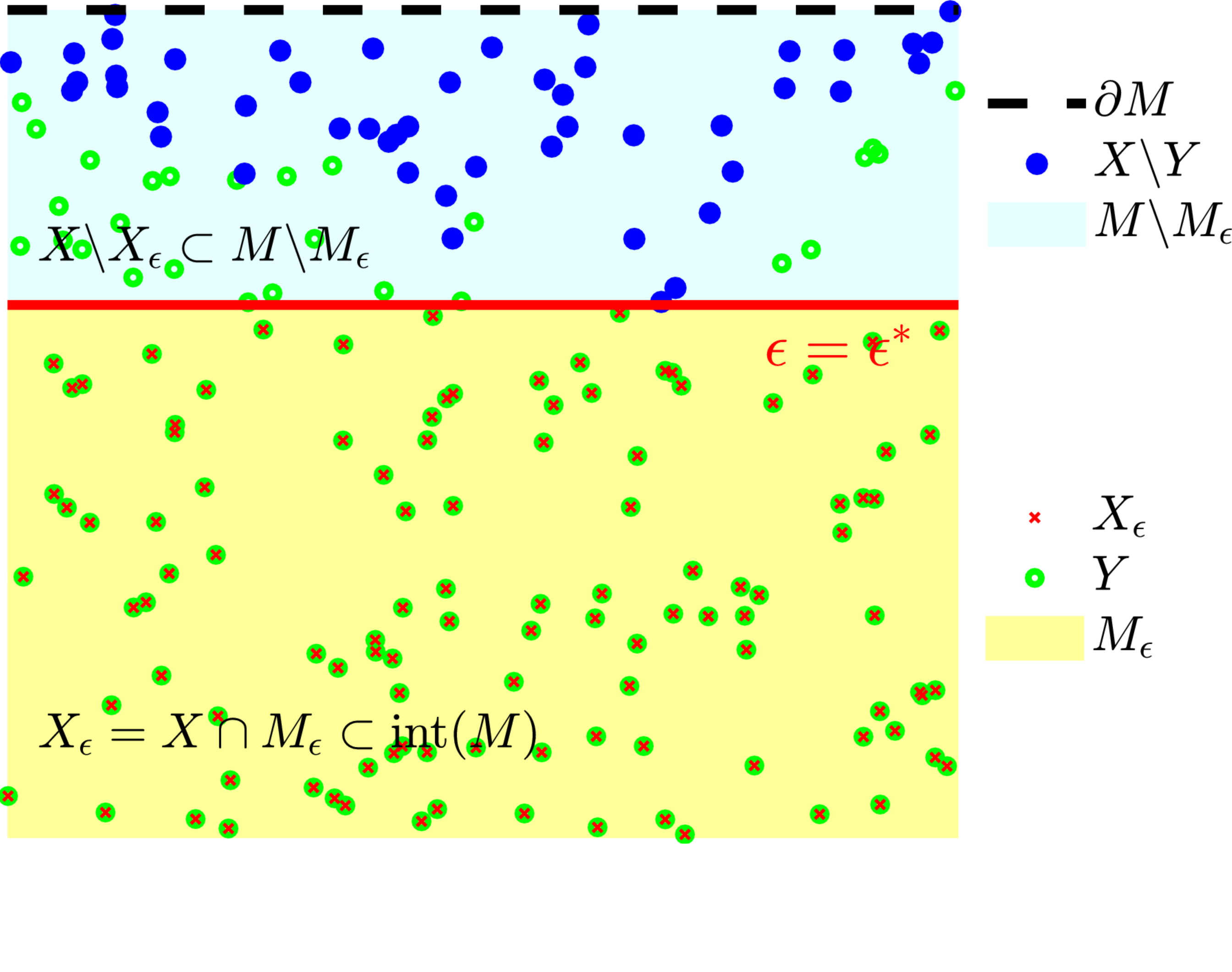} &
\includegraphics[width=3in, height=2.2 in]{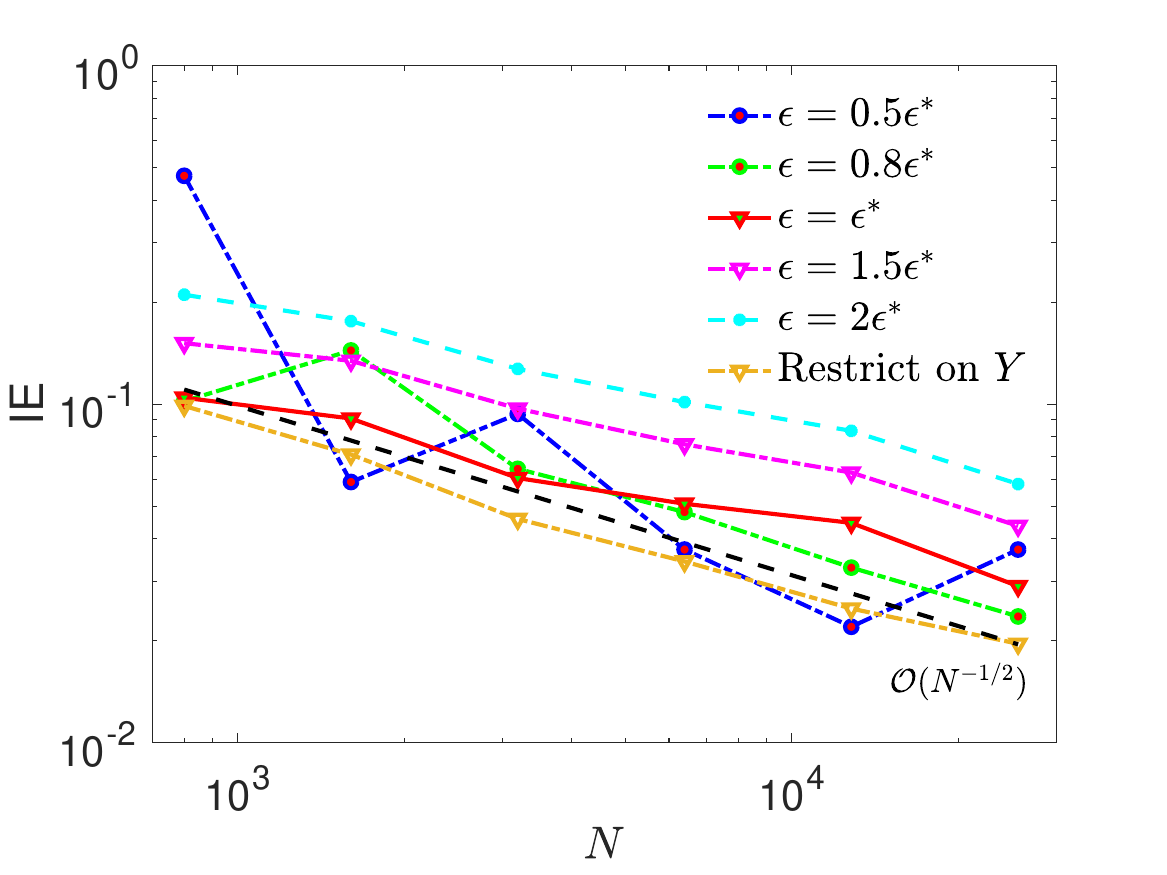} 
\end{tabular}
}
\caption{(a) Sketch of the volume-constraint approach. (b) \textbf{2D semi-torus in $\mathbb{R}^3$}. Inverse errors (IE) based on volume-constraint approach: Comparison of truncating based on the proposed algorithm, restriction on $Y$ as defined in \eqref{setY}, and the $\epsilon$-distance based approach, where we set $\epsilon$ to be proportional to $\epsilon^*:= \max_{\mathbf{x}_i \in X\backslash Y} d_g(\mathbf{x}_i, \partial M)$ as defined in \eqref{eps_alg}.}
\label{figtruncate}
\end{figure*}

\section{Numerical experiments}\label{numerics}


In this section, we will numerically verify the performance of the stable approximation obtained by solving the linear programming in \eqref{eqn:linear_optm}-\eqref{eqn:linear_const} constrained on the intrinsic polynomial least squares approximation proposed in Section \ref{Sec:int_LS}. While the theoretical results in the previous section assume that the projection matrix $\mathbf{P}$ at every point cloud data or the tangent vectors $\{\mathbf{t}_{\mathbf{x}_i,j}\}_{i=1,\ldots, N,j=1,\ldots,d}$ are given, the numerical results in this section assume they are unknown and use in Algorithm \ref{algo:local-LS} the approximate tangent vectors and projection matrices $\mathbf{\hat{P}}$ obtained from the second-order local SVD method \cite{harlim2022rbf} with $K_\mathbf{\hat{P}}$-nearest neighbor points. Such an estimate is needed especially when the manifolds are unknown as we demonstrate below with the "Stanford bunny" and "face" surfaces.

\comment{To quantify the consistence of the approximation, we define the forward error (\textbf{FE}) as
\BEA
\textbf{FE}=\max\limits_{1\leq i\leq N}|\Delta_Mu(\textbf{x}_i)-(\mathbf{L}\textbf{u})_i|,
\EEA
{\color{blue} where $\mathbf{u} = (u(\mathbf{x}_1),\ldots,u(\mathbf{x}_N))$ and $\textbf{L}$ is the discrete Laplacian matrix obtained by the linear programming in \eqref{eqn:linear_optm}-\eqref{eqn:linear_const}.} Denoting $\mathbf{U}$ as the estimated solution over $X$, the inverse error (\textbf{IE}) is defined as
\BEA
\textbf{IE}=\Vert \mathbf{u}-\mathbf{U}\Vert_\infty,
\EEA
a metric to quantify the convergence of the PDE solution. {\color{blue}For simple manifolds (tori in our examples), $u$ will be specified as the analytical solution. For unknown manifold examples, where analytical solution is not available, we will identify $u$ as the approximate solution obtained by surface Finite Element Method (FEM).}}


We will organize the discussion as follows: In Sections~\ref{sec5.1}-\ref{sec5.2}, we report results on closed surfaces, a 2D torus embedded in $\mathbf{R}^9$ and the Stanford Bunny, respectively. For diagnostic, we compare the proposed method to a local approach, the RBF-generated finite difference (RBF-FD) type PDE solver \cite{shankar2015radial}, for manifolds without boundary. In Section~\ref{sec5.3}-\ref{sec5.4}, we report results on a Dirichlet problem on a half torus embedded in $\mathbb{R}^3$ and an unknown "face" manifold, respectively.
For diagnostic, we compare the proposed method to a Graph-Laplacian-based method, Variable-Bandwidth Diffusion Maps (VBDM) \cite{bh:16vb} PDE solver for manifolds with boundary. See Appendix~\ref{app:C} for the details of RBF-FD and VBDM methods.

\comment{
\subsection{Closed manifolds}\label{Sec:closed_m}
In this section, we consider solving the linear elliptic PDE,
\BEA
(a-\Delta _{M})u=f.
\label{eqn:elliptic_PDE}
\EEA
on closed manifolds $M$.
The discrete solution $\textbf{U}$ over $X$ could be calculated as
$$
\textbf{U}=(\textbf{a}-\textbf{L})^{-1}\textbf{f}.
$$
{\color{blue}where $\textbf{a}$ denotes a diagonal matrix with diagonal components $a_{ii} = a(\textbf{x}_i)>0$ and $\textbf{L}$ is the discrete Laplacian matrix obtained by the linear programming in \eqref{eqn:linear_optm}-\eqref{eqn:linear_const}.}

{\color{red}I feel the above is repetitive. Once section 3.2 is cleaned, maybe it will be there at the end.}}

\subsection{Torus in high dimension}\label{sec5.1}
First, we investigate the solver for the elliptic PDE in \eqref{closed_pde} on a general torus. For $q\in\mathbb{N}$, the parameterization of a general torus in $\mathbb{R}^{2q+1}$ is
given by
\begin{equation}
\textbf{x}=\left[
\begin{array}{c}
x^{1} \\
x^{2} \\
\vdots \\
x^{2q-1} \\
x^{2q} \\
x^{2q+1}%
\end{array}%
\right] =\left[
\begin{array}{c}
(c_{0}+\cos \theta )\cos \phi \\
(c_{0}+\cos \theta )\sin \phi \\
\vdots \\
\frac{1}{q}(c_{0}+\cos \theta )\cos q\phi \\
\frac{1}{q}(c_{0}+\cos \theta )\sin q\phi \\
\sqrt{\sum_{i=1}^{q}\frac{1}{i^{2}}}\sin \theta%
\end{array}%
\right] ,  \label{eqn:gentor}
\end{equation}%
where the two intrinsic coordinates $0\leq \theta ,\phi \leq 2\pi $ and the
radius is chosen as $c_{0}=2>1$. The Riemannian metric is
\begin{equation}
g=\left[
\begin{array}{cc}
{\sum_{i=1}^{q}\frac{1}{i^{2}}}& 0 \\
0 & q(c_{0}+\cos \theta )^{2}%
\end{array}%
\right].  \label{eqn:metgentors}
\end{equation}%
The
Laplace-Beltrami operator acting on a function can be computed as,%
\begin{equation}
\Delta _{M}u=\frac{1}{(c_{0}+\cos \theta )}\left[ \frac{\partial }{\partial
\theta }\left( \left( c_{0}+\cos \theta \right) \frac{1}{\sum_{i=1}^{q}\frac{1}{i^{2}}}%
\frac{\partial u}{\partial \theta }\right) +\frac{\partial }{\partial \phi }%
\left( \frac{1}{q(c_{0}+\cos \theta )}\frac{\partial u}{\partial
\phi }\right) \right] .  \label{eqn:psigento}
\end{equation}

Numerically, the grid points $\{\theta_i,\phi_i\}$ are randomly sampled from the uniform distribution on $[0,2\pi)\times[0,2\pi)$. The constant $a$ in (\ref{closed_pde}) is chosen to be $1$. In our numerical test, we consider the torus in $\mathbb{R}^9$, i.e., $q=4$. The true solution is set to be $u(\theta,\phi)=\sin(\theta)\sin(\phi)$ and the manufactured $f$ is calculated in local coordinates by $f:=(a-\Delta_M)u$ using the above formula.

\begin{figure*}[htbp]
{\scriptsize \centering
\begin{tabular}{cc}
{\normalsize (a) \textbf{FE} of $\Delta_M$} & { \normalsize (b) \textbf{IE} of solution }
\\
\includegraphics[width=3in, height=2.2 in]{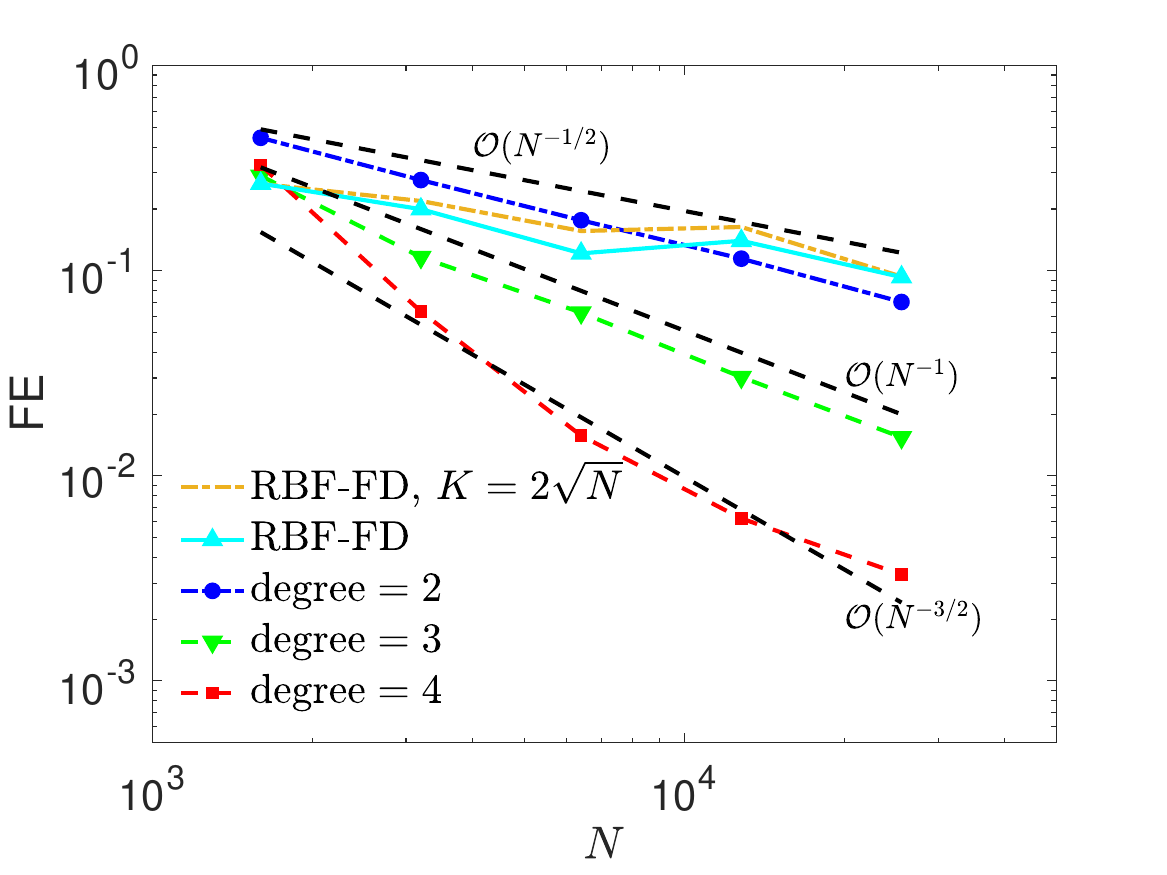} & 
\includegraphics[width=3in, height=2.2 in]{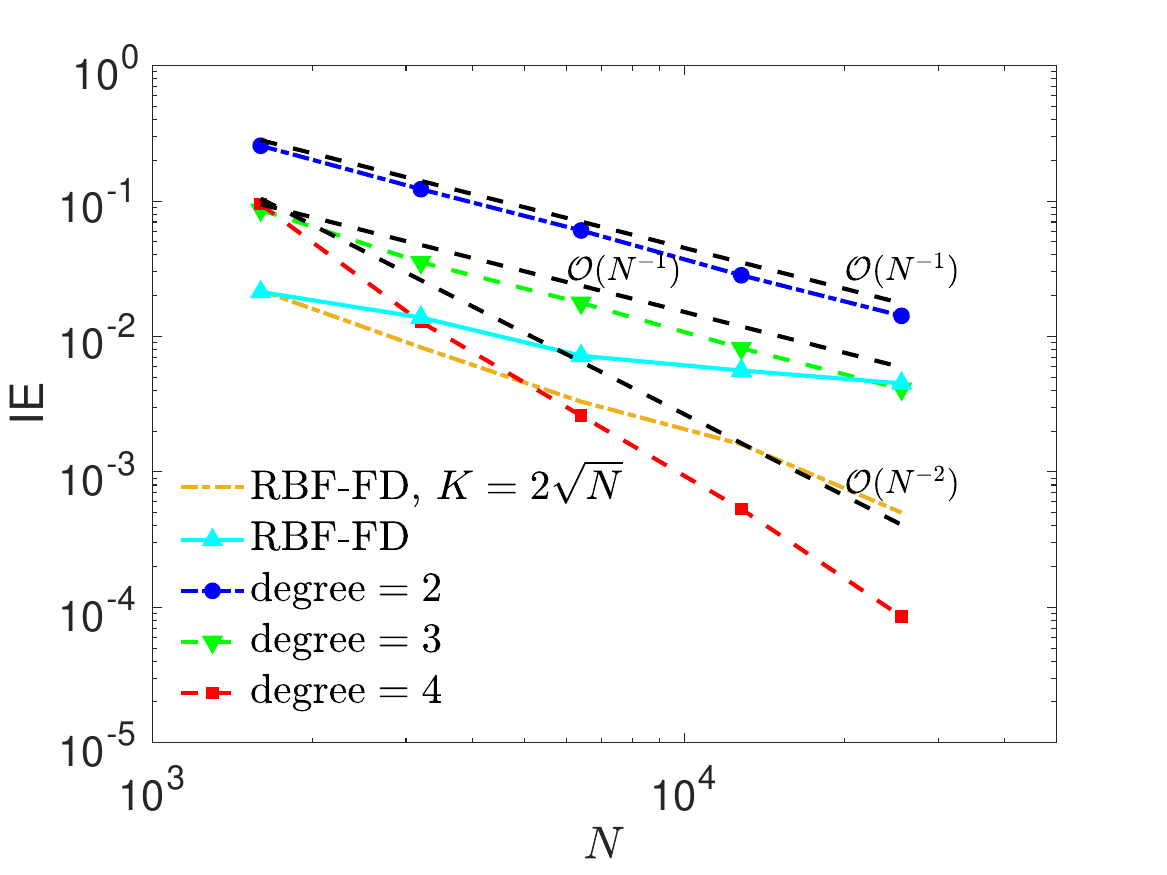}
\end{tabular}
}
\caption{\textbf{2D general torus in $\mathbb{R}^9$}. (a) \textbf{FE} of $\Delta_M$. (b) \textbf{IE} of solutions of the elliptic PDE. }
\label{fig_general_torus}
\end{figure*}

In Fig.~\ref{fig_general_torus}, we show the average forward error (\textbf{FE}) and inverse error (\textbf{IE}) over 6 independent trials {\color{black}(each trial with a randomly sampled $X$)} as functions of $N$. To illustrate the convergence of solutions over the number of points $N=[1600,3200,6400,12800,25600]$, we fix $K=41$ nearest neighbors in the intrinsic-polynomial least squares approximation.
The projection matrix, $\mathbf{P}$, is estimated using the 2nd-order local SVD scheme \cite{harlim2022rbf} with $K_\mathbf{\hat{P}}=2\sqrt{N}$.

One can see that the \textbf{FE} using degree $l=2,3,4$ decrease roughly on the order of $N^{-(l-1)/2}$ which is consistent with the error bound in (\ref{errorrate}). For the inverse error of the solution, the converge order for $l=2$ and $l=3$ is the same, roughly equal to $N^{-1}$, and for $l=4$, is roughly equal to $N^{-2}$. Interestingly, the accuracy and the super-convergence rate observed here are competitive compared to the numerical experiment that uses data points exactly on the boundary (see Appendix~\ref{app:B}). As for the RBF-FD, the \textbf{IE} decays much slower than the proposed methods if we employ it with the same sparsity with $K = 41$. On the other hand, if we choose a denser implementation with $K=2\sqrt{N}$ in RBF-FD, the \textbf{IE} could have much smaller errors and decay slightly faster than $N^{-1}$, but still less accurate compare to using polynomials of degree-4 with $K=41$.

\subsection{On "Stanford bunny"}\label{sec5.2}
In this example, we consider solving the elliptic problems on the Stanford Bunny model which is a two-dimensional surface embedded in $\mathbb{R}^3$. The data of this model is downloaded from the Stanford 3D Scanning Repository \cite{bunny}. The original data set of the Stanford Bunny comprises a triangle mesh with 34,817 vertices. Since this data set has singular regions at the bottom, we generate a new mesh of the surface using the Marching Cubes algorithm~\cite{lorensen1987marching} that is available through the Meshlab~\cite{cignoni2008meshlab}. We should point out that the Marching Cubes algorithm does not smooth the surface. We will verify the solution on a smoothed surface, following the numerical work in \cite{shankar2015radial}. Specifically, we smooth the surface using the Screened Poisson surface reconstruction algorithm to generate a watertight implicit surface that fits the point cloud of the bunny. Subsequently, we apply the Poisson-disk sampling algorithm through the Meshlab to generate a point cloud of $N=32000$ points over the surface. Then a new mesh could be obtained by using the Marching Cubes algorithm again.

We solve the PDE problem in \eqref{closed_pde} with $a=0.2$ and $f=0.6(x_1+x_2+x_3)$. In this example, we have no access to the analytic solution due to the unknown embedding function. For comparison, we take the surface finite-element method solution obtained from the FELICITY FEM Matlab toolbox \cite{walker2018felicity} as the reference (see Fig.~\ref{bunny_soln}(a)). In Fig.~\ref{bunny_soln}(b), we report the result using the intrinsic polynomial approximation with degree $l=2$ and $K=15$ nearest neighbors. To estimate the projection matrix at each point, we fix $K_\mathbf{\hat{P}}=12$ in the 2nd-order local SVD scheme (see \cite{harlim2022rbf}). We found that the absolute difference between the FEM solution and the proposed method ($\textbf{IE} = 0.0190$ in Fig.~\ref{bunny_soln}(b)) is much smaller than that of FEM solution and RBF-FD solution ($\textbf{IE} = 0.1270$) that uses the same $K=15$ (not shown here). While similar maximum absolute error ($\textbf{IE} = 0.0218$) can be obtained for RBF-FD with a larger choice of $K=43$ (Fig.~\ref{bunny_soln}(c)), the solution is spatially not that accurate as one can see from Fig.~\ref{bunny_soln}(c). We found that RBF-FD can be spatially more accurate ($\textbf{IE} = 0.0116$) with a larger $K=61$ (not shown here).


From the numerical results in this (and the previous) subsection, we conclude that the proposed technique with local polynomials can be more accurate than RBF-FD when both schemes are employed with the same memory allocation (the same degree of sparseness).

\begin{figure*}[htbp]
{\scriptsize \centering
\begin{tabular}{ccc}
{\normalsize (a) FEM solution} & { \normalsize (b) $l=2$, $\textbf{IE} = 0.0190$, $K=15$ } & { \normalsize (c) RBF-FD, $\textbf{IE} =0.0218$, $K=43$ }
\\
\includegraphics[width=1.95in, height=1.3 in]{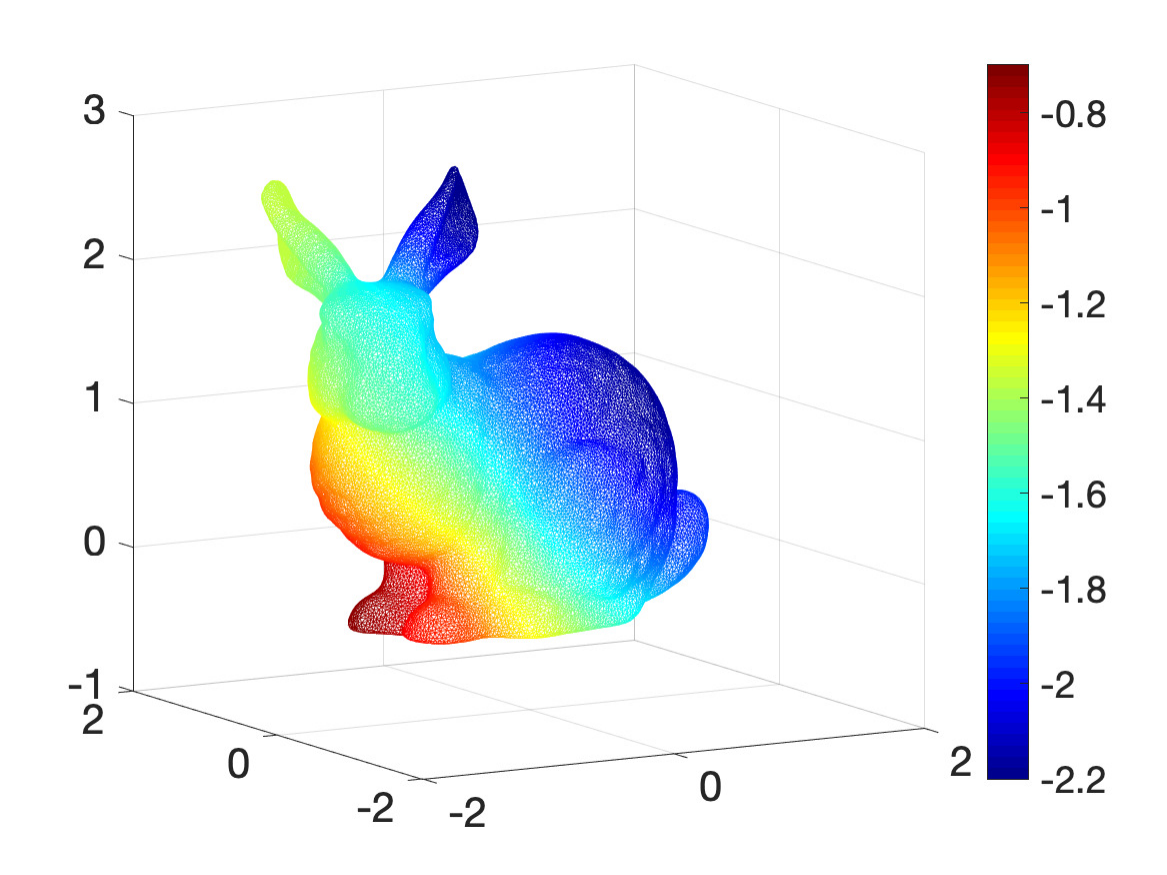} &
\includegraphics[width=1.95in, height=1.3 in]{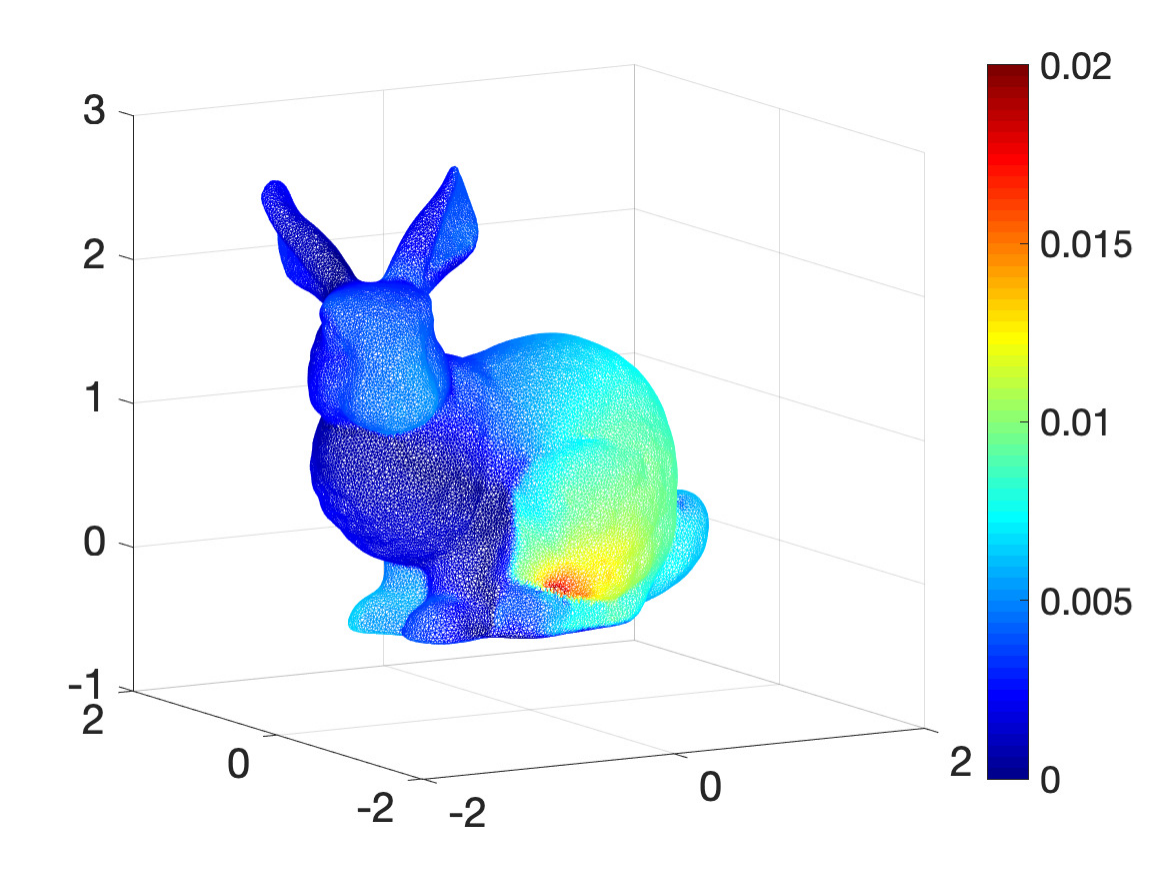} &
\includegraphics[width=1.95in, height=1.3 in]{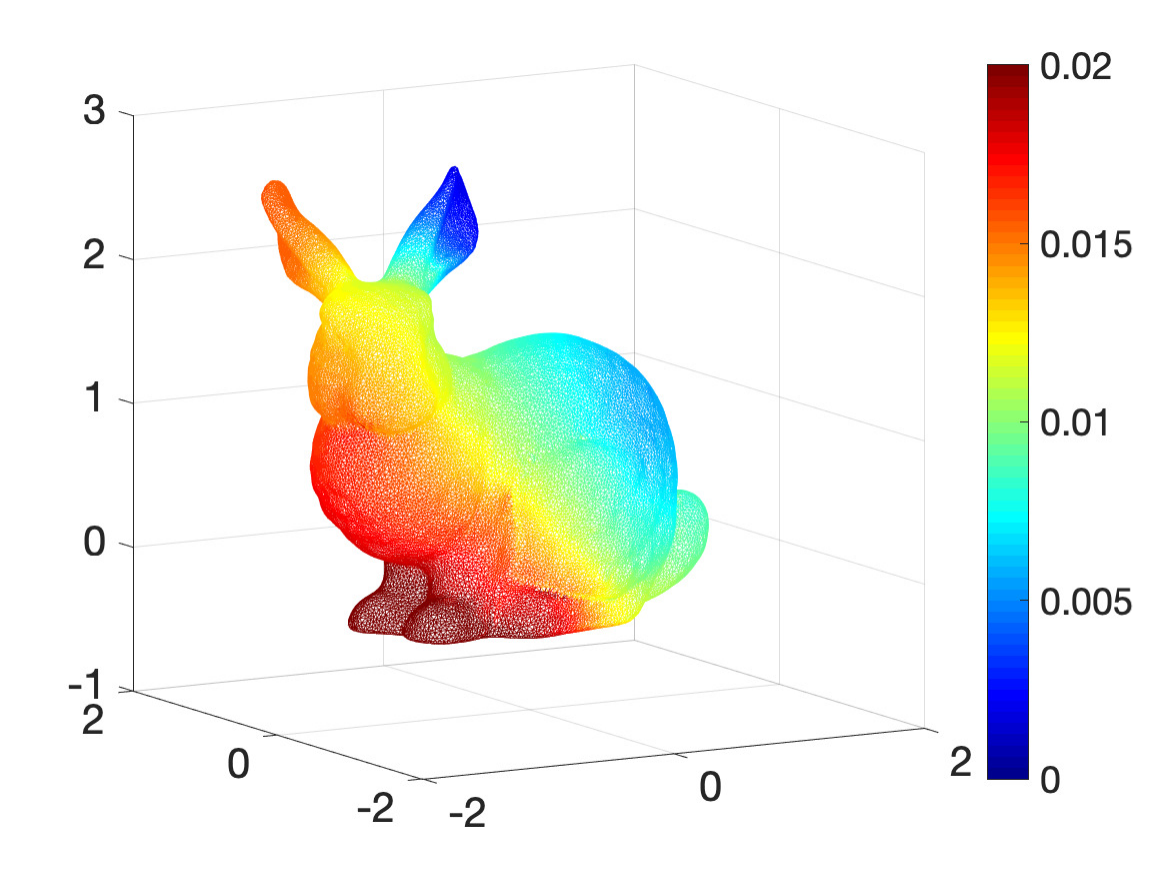}
\end{tabular}
}
\caption{{\bf Bunny example.} $N=32000$. (a) FEM solution. (b) Absolute difference at each point between FEM and intrinsic polynomial with degree $l=2,K=15$. 
(c) Absolute difference at each point between FEM and RBF-FD with $s=0.2, K=43$ (see the shape parameter $s$ in Appendix~\ref{app:C}). }
\label{bunny_soln}
\end{figure*}

\subsection{Semi-torus}\label{sec5.3}

{\color{black}
In this example, we consider solving the Poisson problem in (\ref{diri_pde}) on a semi-torus  which has the embedding function
$$
\iota(\theta, \phi)=\left(\begin{array}{c}{(R+r\cos \theta) \cos \phi} \\ {(R+r\cos \theta) \sin \phi} \\ {r\sin \theta}\end{array}\right), \quad \theta \in[0,2 \pi),\quad \phi\in[0,\pi]
$$
where $R$ is the distance from the centre of the tube to the center of torus and $r$ is the distance from the center of the tube to the surface of the tube with $r<R$. The induced Riemannian metric is
$$
g_{\textbf{x}^{-1}(\theta, \phi)}(v, w)=v^{\top}\left(\begin{array}{cc}r^2 & 0 \\ 0 & (R+r\cos \theta)^{2}\end{array}\right) w, \quad \forall v, w \in T_{\textbf{x}^{-1}(\theta, \phi)} M.
$$
In our numerical experiments, we take $R=2$ and $r=1$. The true solution $u$ of the PDE is set to be $u(\theta,\phi)=\sin\phi\sin\theta$. The forcing term $f$ could be calculated as
$$
f:=\Delta_M u=\frac{1}{\sqrt{|g|}}\left[\frac{\partial}{\partial \theta}\left(\sqrt{|g|} g^{11} \frac{\partial u}{\partial \theta}\right)+\frac{\partial}{\partial \phi}\left(\sqrt{|g|} g^{22} \frac{\partial u}{\partial \phi}\right)\right].
$$
}

Numerically, the grid points $\{\theta_i,\phi_i\}$ are randomly sampled from the uniform distribution on $[0,2\pi)\times[0,\pi]$.  To illustrate the convergence of solutions over number of points $N=[800,1600,3200,6400,12800,25600]$, we fix $K=51$ in the intrinsic polynomial approximation.  As discussed in the previous section~\ref{sec:comunk}, we assume that no boundary points are given and the boundary location is unknown. We identify the boundary points close to boundary as those with $w_1>0$ (as defined in Section~\ref{sec:comunk}) and approximate the solution by solving the linear problem in \eqref{diri_linproblem}.

In Fig.~\ref{fig_semi_torus}, we show the average \textbf{FE} and \textbf{IE} over 6 independent trials as a function of $N$. One can see that for the forward error of $\Delta_M$, the error decreases roughly on the order of $N^{-1/2}$ when the degree of $l=2$ intrinsic polynomial is implemented. On the other hand, the \textbf{FE} decreases on the order of $N^{-1}$ with the intrinsic polynomial of degree $l=3$. For the \textbf{IE} of solutions, the convergence rates for $l=2$ and $l=3$ are the same, roughly equal to $N^{-1/2}$, which agrees with the rate of VBDM and the theoretical error rate in Theorem~\ref{convrate_diri} as well.

\begin{figure*}[htbp]
{\scriptsize \centering
\begin{tabular}{cc}
{\normalsize (a) \textbf{FE} of $\Delta_M$} & { \normalsize (b) \textbf{IE} of solutions }
\\
\includegraphics[width=3
in, height=2.2 in]{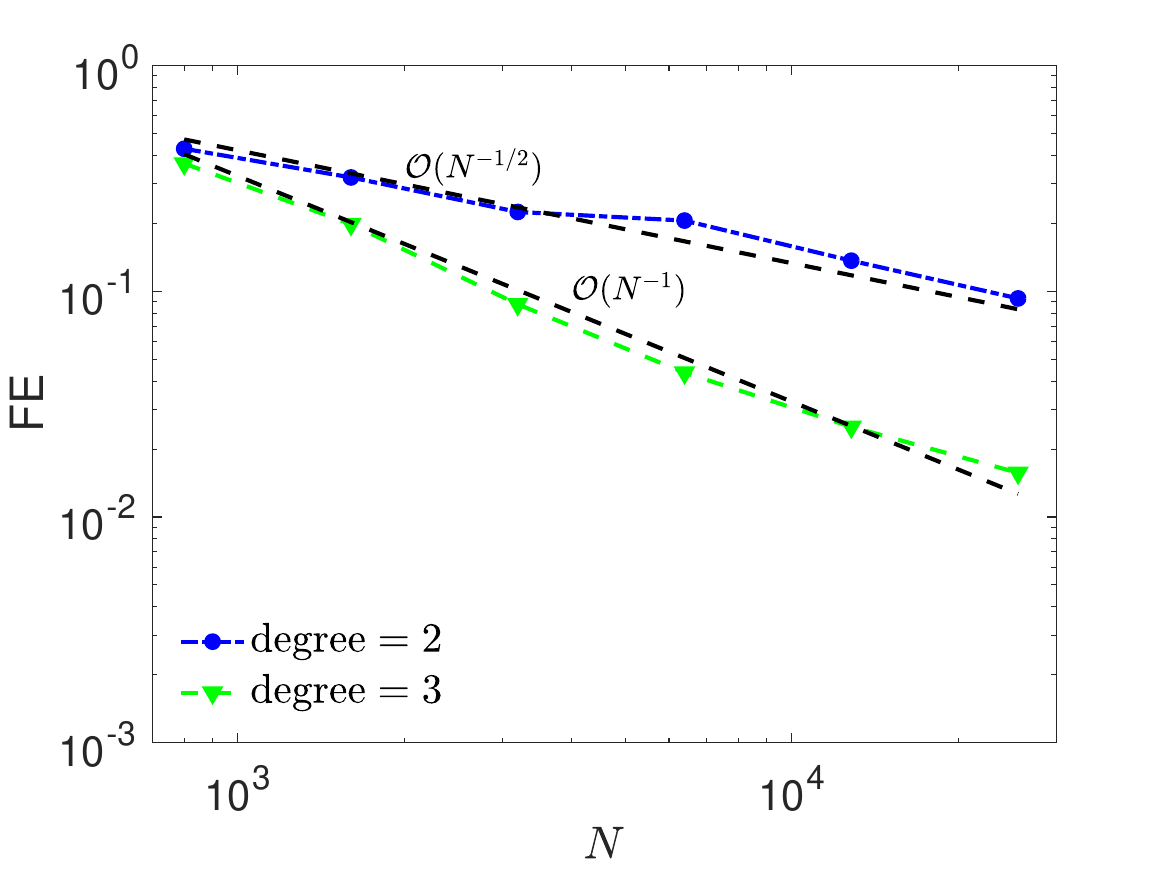} &
\includegraphics[width=3in, height=2.2 in]{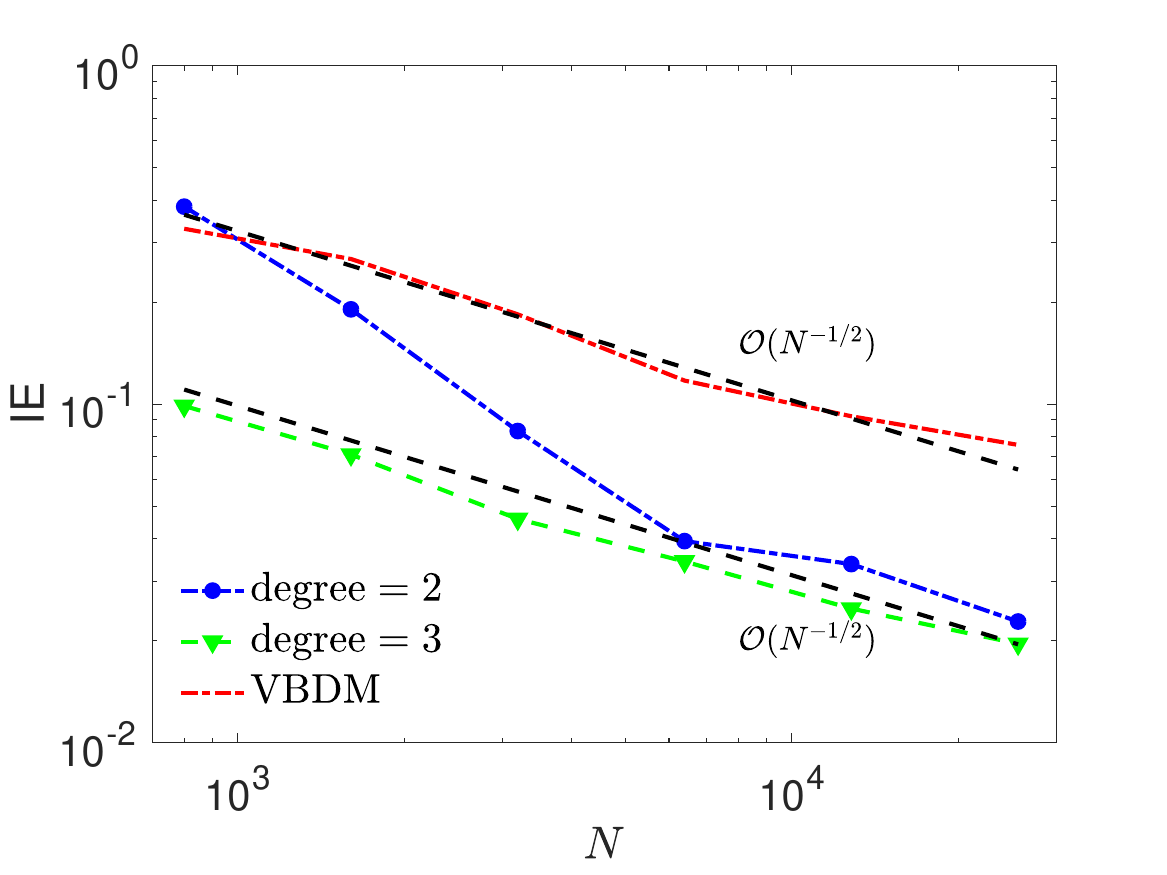}
\end{tabular}
}
\caption{\textbf{2D semi-torus in $\mathbb{R}^3$}. $K=51$ nearest
neighbors are used. (a) \textbf{FE} of $\Delta_M$. (b) \textbf{IE} of solutions of the elliptic PDE with the homogenous Dirichlet boundary condition. }
\label{fig_semi_torus}
\end{figure*}

\subsection{On "face"}\label{sec5.4}
In the last example, we consider solving the equation (\ref{diri_pde}) with $f=8\cos(10(x_1+x_2+x_3))+4$ on an unknown manifold example of a two-dimensional "face" $\mathbf{x}=(x_1,x_2,x_3)\in M\subset\mathbb{R}^3$. The boundary points are given on the exact boundary of the "face". The dataset is from Keenan Crane's 3D repository\cite{face}. For comparison, we again take the surface FEM solution obtained from the FELICITY toolbox as the reference  (see Fig.~\ref{face_soln}(a)).
\comment{And we will compare the maximum absolute difference to the FEM solution which is defined as
$$
E=\max\limits_{1\leq i\leq N}|\hat{u}(\textbf{x}_i)-u^{FEM}(\textbf{x}_i)|
$$
where $u^{\mathrm{FEM}}$ is the FEM solution and $\widehat{u}$ is the estimated solution obtained from the corresponding solver.}

In Fig.~\ref{face_soln}(b), we report the result using the intrinsic polynomial approximation with degree of $l=2$ and $K=41$ nearest neighbors. To estimate the projection matrix at each point, we fix $K_\mathbf{\hat{P}}=23$ in the 2nd-order local SVD scheme.  We found that the maximum absolute difference between the FEM solution and the proposed method ($\textbf{IE}=0.0014$) is comparable to that between FEM and VBDM solutions ($\textbf{IE}=0.0014$).  Here, the scaling of the FEM solution is on the order of $10^{-2}$. We should point out here that the solution obtained from the intrinsic polynomial regression could be sensitive to the estimated projection matrix, as the surface is not smooth, especially near the ``nose'' region.



\begin{figure*}[htbp]
{\scriptsize \centering
\begin{tabular}{ccc}
{\normalsize (a) FEM solution} & { \normalsize (b) $l=2$, $\textbf{IE}=0.0014$ } & {\normalsize (c) VBDM, $\textbf{IE}=0.0014$}
\\
\includegraphics[width=1.95in, height=1.3 in]{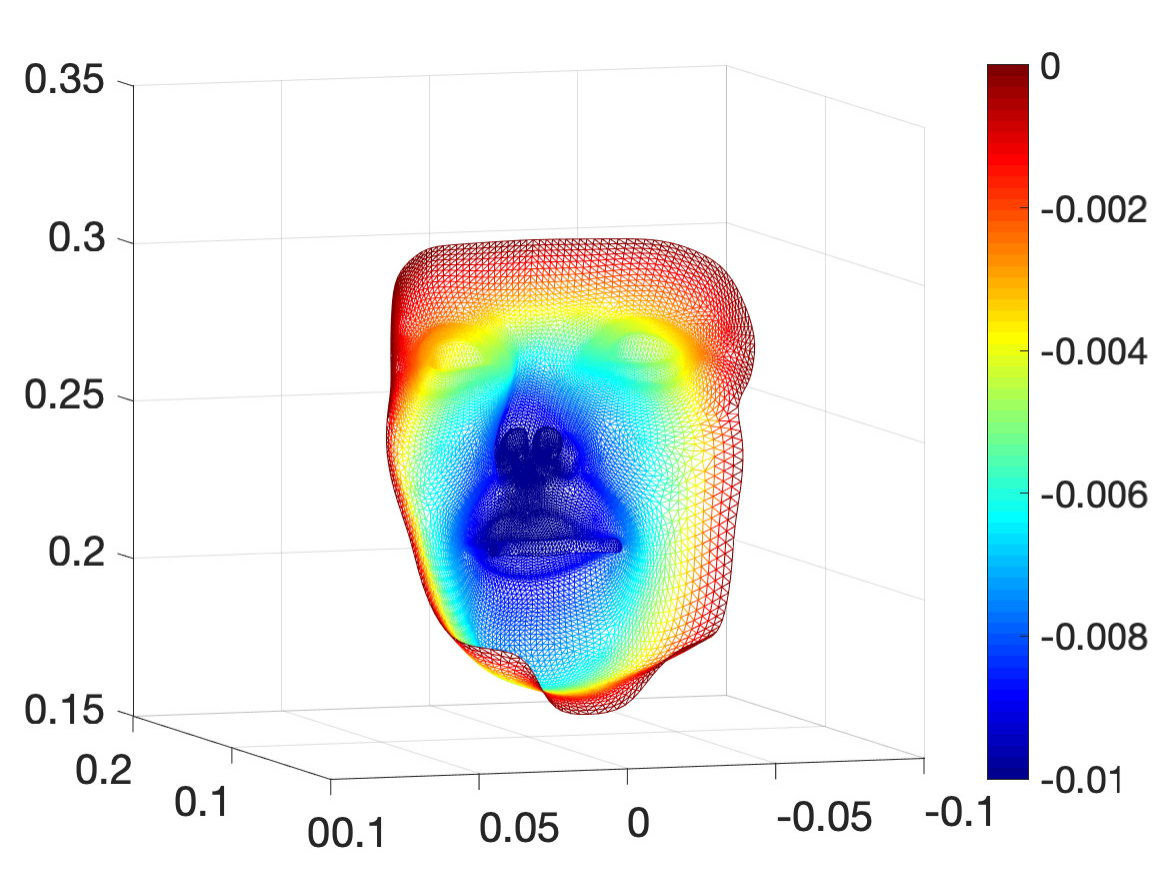} &
\includegraphics[width=1.95in, height=1.3 in]{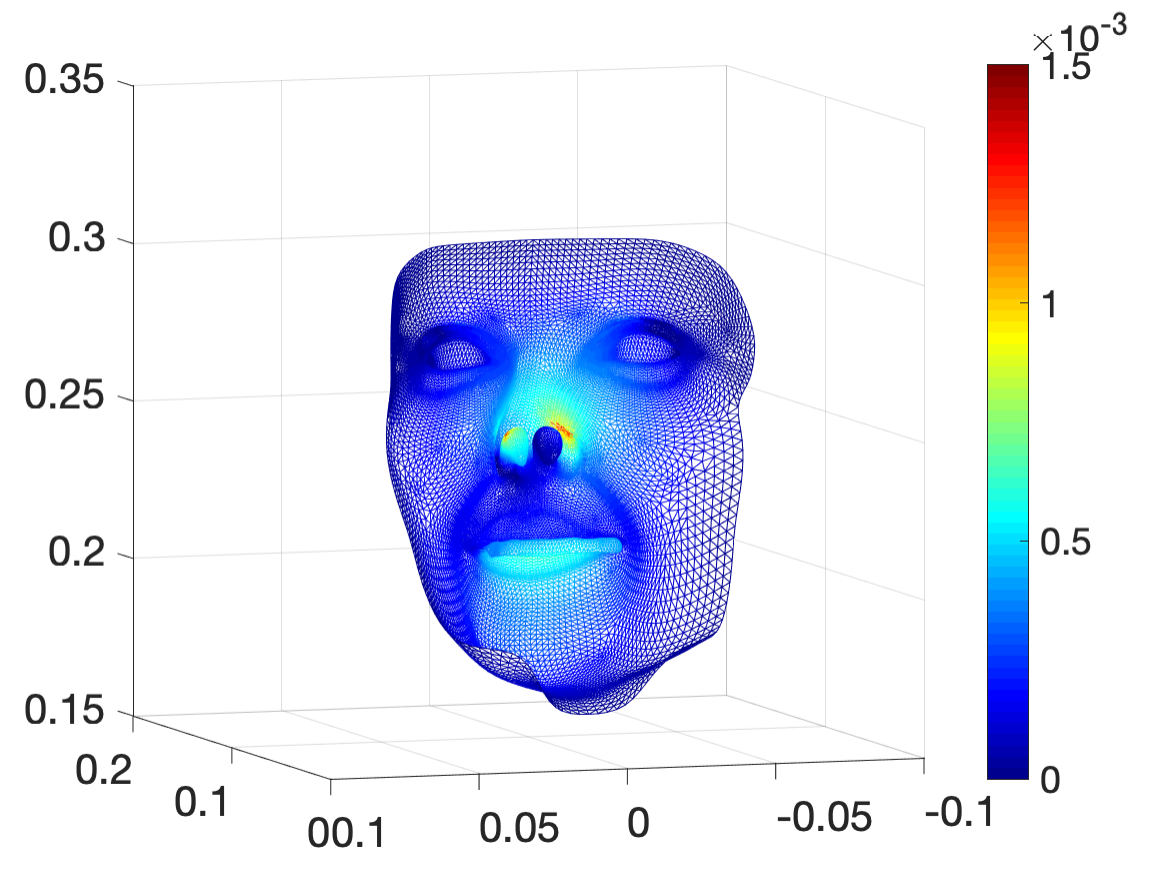} &
\includegraphics[width=1.95in, height=1.3 in]{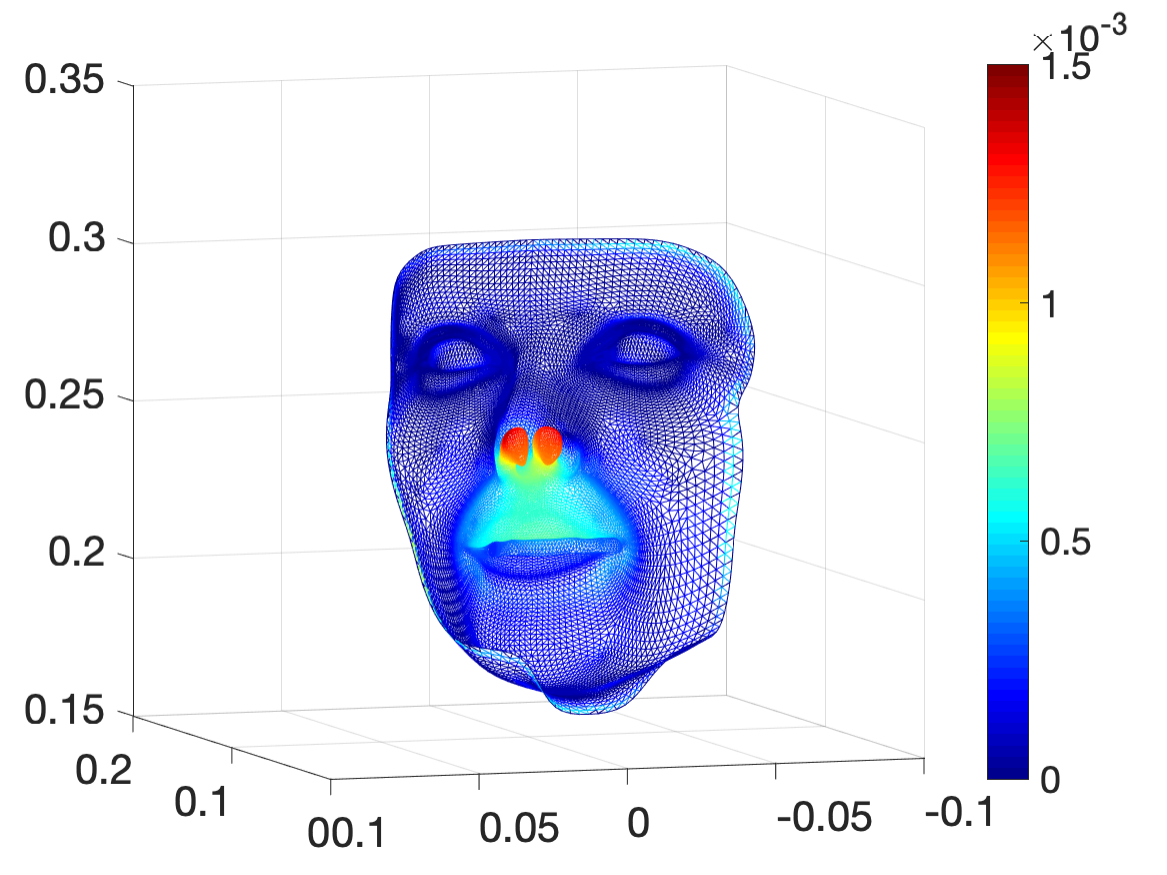}%
\end{tabular}
}
\caption{(Color online) {\bf Face example}. $N=17157$. (a) FEM solution. (b) Absolute difference between FEM and intrinsic polynomial regression with degree of $l=2$. (c) Absolute difference between FEM and VBDM. }
\label{face_soln}
\end{figure*}


\section{Conclusions}\label{sec6}

{\color{black}
In this paper, we considered a GFDM approach for solving Poisson problems on smooth compact manifolds without or with boundaries, identified by randomly sampled data points. The scheme can reach an arbitrary-order algebraic accuracy on smooth manifolds when the tangent spaces of manifolds are given and boundary points are given if Dirichlet problems are considered.
We used Taylor's expansion formulation in \eqref{eqn:int_poly2} and the established error bound of GMLS in \eqref{errorrate} for the consistency analysis which relies on the local expansion of functions with respect to a set of intrinsic polynomial bases defined on the tangent bundles.
However, for randomly sampled data points, if only least squares approximations in Algorithm \ref{algo:local-LS} are employed, there is still no convergence of PDE solutions since the Laplacian matrix is not stable.
To stabilize the Laplacian matrix, we further combined a linear optimization approach with additional constraints in \eqref{eqn:linear_optm}-\eqref{eqn:linear_const}.
We provided the theoretical justification as well as supporting numerical examples to verify the convergence of solutions.
For homogeneous Dirichlet problems with only a point cloud given in the interior of manifolds, we considered the volume-constraint approach which imposes the boundary conditions on the points that are close to the boundary of manifolds.
Numerically, these close-to-boundary points are detected by considering the sign of the least squares weights in our new algorithm even without needing to know the location of the boundary. For this volume-constraint approach, we theoretically proved the convergence of solutions using the discrete maximum principle for randomly sampled data points.

While the result is promising, there are various open questions for future investigation. First, the theoretical justification is still lacking for consistency and stability when $C\neq 0$ in the linear optimization problem \eqref{eqn:linear_optm}-\eqref{eqn:linear_const}.
Second, it is of interest to understand the super-convergence phenomenon which is numerically observed for the solutions using even-degree intrinsic-polynomial approximations for randomly sampled data on manifolds.
Third, it is also of interest to generalize this approach to solve vector-valued PDEs on manifolds with boundaries identified with randomly sampled data points, extending the results in \cite{gross2020meshfree} for closed manifolds identified with quasi-uniform data points.
}

\section*{Acknowledgment}
The research of J.H. was partially supported by the NSF grants DMS-2207328, DMS-2229535, and the ONR grant N00014-22-1-2193.  S. J. was supported by the NSFC Grant No.
12101408 and the HPC Platform of ShanghaiTech University.

\appendix

\section{Proof of Theorem~\ref{thm_closed}}\label{app:A}

The following two technical lemmas will be useful for proving Theorem~\ref{thm_closed}.

\begin{lem}\label{lem:Lap}
Fix some $\mathbf{x}_{i}\in X$. Let $u(\mathbf{x})=\sum_{0\leq |\alpha
|\leq l}b_{\alpha }p_{\mathbf{x}_{i},\alpha }(\mathbf{x})+\sum_{|\alpha
|=l+1}r_{\mathbf{x}_{i,k},\alpha }p_{\mathbf{x}_{i},\alpha }(\mathbf{x})$ be
the Taylor expansion of $u\in C^{l+1}(M^{\ast })$ and $\bar{\bar{u}}(\mathbf{%
x})=\sum_{0\leq |\alpha |\leq 2}b_{\alpha }p_{\mathbf{x}_{i},\alpha }(%
\mathbf{x})$ be the truncation of the Taylor expansion of $u$ upper to
degree $2$. Then,
\begin{equation*}
\Delta _{M}u(\mathbf{x}_{i})=\Delta _{M}\bar{\bar{u}}(\mathbf{x}_{i}).
\end{equation*}
\end{lem}

\begin{proof}
We use the notations as introduced in Section \ref{Sec:Taylor}. Based
on the definition to Laplace-Beltrami operator, we have
\begin{equation*}
\Delta _{M}u(\mathbf{x}_{i})=\sum_{m=1}^{d}\frac{\partial ^{2}\tilde{u}(\vec{%
0})}{\partial s_{m}^{2}}=\sum_{m=1}^{d}\frac{\partial ^{2}u(\mathbf{x}_{i})}{%
\partial s_{m}^{2}},
\end{equation*}%
where $\tilde{u}(\vec{s})=u(\exp _{\mathbf{x}_{i}}(\vec{s}))=u(\mathbf{x})$.
We write down the expression for $\bar{\bar{u}}(\mathbf{x})$,
\begin{equation*}
\bar{\bar{u}}(\mathbf{x})=u(\mathbf{x}_{i})+\sum_{k=1}^{d}z^{k}\frac{%
\partial }{\partial s_{k}}\tilde{u}(\vec{0})+\frac{1}{2}%
\sum_{i,j=1}^{d}z^{i}z^{j}\bigg(\frac{\partial ^{2}}{\partial s_{i}\partial
s_{j}}\tilde{u}(\vec{0})-\sum_{k=1}^{d}\Gamma _{ij}^{k}\frac{\partial }{%
\partial s_{k}}\tilde{u}(\vec{0})\bigg).
\end{equation*}%

For any $m\in \left\{ 1,\ldots ,d\right\} $, we can compute its first and
second derivatives,%
\begin{eqnarray*}
\frac{\partial \bar{\bar{u}}(\mathbf{x})}{\partial s_{m}} &=&\sum_{k=1}^{d}%
\bigg(\mathbf{t}_{k}^{\top }\frac{\partial \mathbf{x}}{\partial s_{m}}\bigg)%
\frac{\partial \tilde{u}(\vec{0})}{\partial s_{k}} \\
&&+\frac{1}{2}\sum_{i,j=1}^{d}\bigg(\mathbf{t}_{i}^{\top }\frac{\partial
\mathbf{x}}{\partial s_{m}}z^{j}+z^{i}\mathbf{t}_{j}^{\top }\frac{\partial
\mathbf{x}}{\partial s_{m}}\bigg)\bigg(\frac{\partial ^{2}}{\partial
s_{i}\partial s_{j}}\tilde{u}(\vec{0})-\sum_{k=1}^{d}\Gamma _{ij}^{k}\frac{%
\partial }{\partial s_{k}}\tilde{u}(\vec{0})\bigg), \\
\frac{\partial ^{2}\bar{\bar{u}}(\mathbf{x})}{\partial s_{m}^{2}}
&=&\sum_{k=1}^{d}\bigg(\mathbf{t}_{k}^{\top }\frac{\partial ^{2}\mathbf{x}}{%
\partial s_{m}^{2}}\bigg)\frac{\partial \tilde{u}(\vec{0})}{\partial s_{k}}
\\
&&+\frac{1}{2}\sum_{i,j=1}^{d}\bigg(\mathbf{t}_{i}^{\top }\frac{\partial ^{2}%
\mathbf{x}}{\partial s_{m}^{2}}z^{j}+z^{i}\mathbf{t}_{j}^{\top }\frac{%
\partial ^{2}\mathbf{x}}{\partial s_{m}^{2}}+2\mathbf{t}_{i}^{\top }\frac{%
\partial \mathbf{x}}{\partial s_{m}}\mathbf{t}_{j}^{\top }\frac{\partial
\mathbf{x}}{\partial s_{m}}\bigg)\bigg(\frac{\partial ^{2}}{\partial
s_{i}\partial s_{j}}\tilde{u}(\vec{0})-\sum_{k=1}^{d}\Gamma _{ij}^{k}\frac{%
\partial }{\partial s_{k}}\tilde{u}(\vec{0})\bigg).
\end{eqnarray*}%
Take $\mathbf{x=x}_{i}$, we obtain that
\begin{eqnarray*}
\frac{\partial ^{2}\bar{\bar{u}}(\mathbf{x}_{i})}{\partial s_{m}^{2}}
&=&\sum_{k=1}^{d}\Gamma _{mm}^{k}\frac{\partial \tilde{u}(\vec{0})}{\partial
s_{k}}+\frac{1}{2}\sum_{i,j=1}^{d}\bigg(2\delta _{im}\delta _{jm}\bigg)\bigg(%
\frac{\partial ^{2}}{\partial s_{i}\partial s_{j}}\tilde{u}(\vec{0}%
)-\sum_{k=1}^{d}\Gamma _{ij}^{k}\frac{\partial }{\partial s_{k}}\tilde{u}(%
\vec{0})\bigg) \\
&=&\sum_{k=1}^{d}\Gamma _{mm}^{k}\frac{\partial \tilde{u}(\vec{0})}{\partial
s_{k}}+\frac{\partial ^{2}}{\partial s_{m}^{2}}\tilde{u}(\vec{0}%
)-\sum_{k=1}^{d}\Gamma _{mm}^{k}\frac{\partial \tilde{u}(\vec{0})}{\partial
s_{k}}=\frac{\partial ^{2}\tilde{u}(\vec{0})}{\partial s_{m}^{2}}.
\end{eqnarray*}%
Take sum over $m\in \left\{ 1,\ldots ,d\right\} $ and the proof is complete.
\end{proof}

\begin{lem} \label{lem:hK}
Fix some integer $K$ and define $h_{K,\max }=\max_{k\in \{1,\ldots
,K\}}\left\Vert \mathbf{x}_{i,k-}\mathbf{x}_{i}\right\Vert $ for a base
point $\mathbf{x}_{i}$, where $\left\Vert \cdot \right\Vert $ denotes the
Euclidean distance. Then,%
\begin{equation*}
\mathbb{P}_{X\sim \mathcal{U}}(h_{K,\max }>\delta )\leq \exp (-C(N-K)\delta
^{d}),
\end{equation*}%
where $C=C(d)/\text{Vol}(M)$ and $\mathcal{U}$ denotes the uniform
distribution on $M$. Moreover, with probability higher than $1-\frac{1}{N}$,
we have
\begin{equation*}
h_{K,\max }=O(N^{-\frac{1}{d}}),
\end{equation*}%
where we have ignored the factor $\log N$ above.
\end{lem}

\begin{proof}
Let $B_{\delta }(\mathbf{x}_{i})$\ dnote a geodesic ball of radius
around the base point $\mathbf{x}_{i}\in X\subset M$ and let $B_{\delta
}^{c}=M\backslash B_{\delta }(\mathbf{x}_{i})$. For a point $\mathbf{x}%
_{m}\in X$\ uniformly sampled from $M$, it holds that
\begin{equation*}
\mathbb{P}_{X\sim \mathcal{U}}\left( \mathbf{x}_{m}\in X\cap B_{\delta
}^{c}\right) =1-\frac{\text{Vol}(B_{\delta }(\mathbf{x}_{i}))}{\text{Vol}(M)}%
.
\end{equation*}%
For small $\delta ,$ we have
\begin{eqnarray*}
&&\mathbb{P}_{X\sim \mathcal{U}}(\max_{k\in \{1,\ldots ,K\}}d_{g}(\mathbf{x}%
_{i,k},\mathbf{x}_{i}) >\delta )=\mathbb{P}\left( \text{there are at least
}N-K\text{ points in }X\text{ outside of }B_{\delta }(\mathbf{x}_{i})\right)
\\
&\leq &\mathbb{P(}\mathbf{x}_{i,K+1}\in B_{\delta }^{c},\ldots ,\mathbf{x}%
_{i,N}\in B_{\delta }^{c})=\left( 1-\frac{\text{Vol}(B_{\delta }(\mathbf{x}%
_{i}))}{\text{Vol}(M)}\right) ^{N-K}.
\end{eqnarray*}%
Using Proposition 14 in \cite{croke1980some}, we obtain that%
\begin{equation*}
\mathbb{P}_{X\sim \mathcal{U}}(\max_{k\in \{1,\ldots ,K\}}d_{g}(\mathbf{x}%
_{i,k},\mathbf{x}_{i})>\delta )\leq \left( 1-\frac{C(d)\delta ^{d}}{\text{Vol%
}(M)}\right) ^{N-K}=(1-C\delta ^{d})^{N-K}\leq \exp (-C(N-K)\delta ^{d}),
\end{equation*}%
where $C(d)$ is a constant depending on the dimension $d$. Using the fact
that $\left\Vert \mathbf{x}_{i,k}-\mathbf{x}_{i}\right\Vert \leq d_{g}(%
\mathbf{x}_{i,k},\mathbf{x}_{i}),$ we arrive at
\begin{equation*}
\mathbb{P}_{X\sim \mathcal{U}}(h_{K,\max }>\delta )\leq \exp (-C(N-K)\delta
^{d}).
\end{equation*}%
Moreover, taking $\exp (-C(N-K)\delta ^{d})=\frac{1}{N},$ we obtain%
\begin{equation*}
\delta =\left( \frac{\log N}{C(N-K)}\right) ^{\frac{1}{d}}\sim O(N^{-\frac{1%
}{d}}),
\end{equation*}%
by ignoring the factor $\log N$. The proof is complete.
\end{proof}

{\bf Proof of Theorem~\ref{thm_closed}:}
Notice that,
\BEA
\left(\hat{L}_{X,A}\mathbf{U} - \hat{L}_{X,A}\mathbf{u}\right)_i &=& f(\mathbf{x}_i) -a(\mathbf{x}_i) u(\mathbf{x}_i)  + \left(\hat{L}_X \mathbf{u}\right)_i  \notag \\
 &=&f(\mathbf{x}_i) -a(\mathbf{x}_i) u(\mathbf{x}_i)  + \left( \left( \hat{L}_X  \mathbf{u}\right)_i - \Delta_M u(\mathbf{x}_i) \right) +  \Delta_M u(\mathbf{x}_i) \notag\\
 &=& \left( \hat{L}_X  \mathbf{u}\right)_i - \Delta_M u(\mathbf{x}_i). \label{proof_eq1}
\EEA
Applying the Taylor expansion for the function $u\in C^{l+1}(M^{\ast })$,
one obtains that%
\begin{equation*}
u(\mathbf{x}_{i,k})=\sum_{0\leq |\alpha |\leq l}b_{\alpha }p_{\mathbf{x}%
_{i},\alpha }(\mathbf{x}_{i,k})+\sum_{|\alpha |=l+1}r_{\mathbf{x}%
_{i,k},\alpha }p_{\mathbf{x}_{i},\alpha }(\mathbf{x}_{i,k}),\text{ \ \ }%
k=1,\ldots ,K,
\end{equation*}%
where the remainder is given in Lagrange's form, $p_{\mathbf{x}_{i},\alpha }(%
\mathbf{x}_{i,k})=\prod_{j=1}^{d}(z_{\mathbf{x}_{i,k}}^{j})^{\alpha
_{j}}=\prod_{j=1}^{d}\left[ \mathbf{t}_{{\mathbf{x}_{i}},j}\cdot (\mathbf{x}%
_{i,k}-{\mathbf{x}_{i}})\right] ^{\alpha _{j}}$, $\{b_{\alpha }\}_{0\leq
|\alpha |\leq l}$ depend only on $\mathbf{x}_{i}$, and $\{r_{\mathbf{x}%
_{i,k},\alpha }\}_{|\alpha |=l+1}$ depend on some point on the geodesic that
connects $\mathbf{x}_{i}$ and $\mathbf{x}_{i,k}$.
Substituting the above Taylor expansion into \eqref{proof_eq1}, we arrive at
\begin{eqnarray*}
&&(\hat{L}_{X}\mathbf{u)}_{i}-\Delta _{M}u(\mathbf{x}_{i})=\sum_{k=1}^{K}%
\hat{w}_{k}u(\mathbf{x}_{i,k})-\Delta _{M}u(\mathbf{x}_{i}) \\
&=&\sum_{k=1}^{K}\hat{w}_{k}\bigg(\sum_{0\leq |\alpha |\leq l}b_{\alpha }p_{%
\mathbf{x}_{i},\alpha }(\mathbf{x}_{i,k})+\sum_{|\alpha |=l+1}r_{\mathbf{x}%
_{i,k},\alpha }p_{\mathbf{x}_{i},\alpha }(\mathbf{x}_{i,k})\bigg)-\Delta
_{M}u(\mathbf{x}_{i}) \\
&=&\bigg(\sum_{k=1}^{K}\sum_{0\leq |\alpha |\leq l}w_{k}b_{\alpha }p_{%
\mathbf{x}_{i},\alpha }(\mathbf{x}_{i,k})-\Delta _{M}u(\mathbf{x}_{i})\bigg)+%
\bigg(\sum_{k=1}^{K}\sum_{|\alpha |=l+1}\hat{w}_{k}r_{\mathbf{x}%
_{i,k},\alpha }p_{\mathbf{x}_{i},\alpha }(\mathbf{x}_{i,k})\bigg) \\
&=&\underbrace{\bigg(\sum_{k=1}^{K}\sum_{1\leq |\alpha |\leq
2}w_{k}b_{\alpha }p_{\mathbf{x}_{i},\alpha }(\mathbf{x}_{i,k})-\Delta _{M}u(%
\mathbf{x}_{i})\bigg)}_{I_{1}}+\underbrace{\bigg(\sum_{k=1}^{K}\sum_{3\leq
|\alpha |\leq l}w_{k}b_{\alpha }p_{\mathbf{x}_{i},\alpha }(\mathbf{x}_{i,k})%
\bigg)}_{I_{2}} \\
&&+\underbrace{\bigg(\sum_{k=2}^{K}\sum_{|\alpha |=l+1}\hat{w}_{k}r_{\mathbf{%
x}_{i,k},\alpha }p_{\mathbf{x}_{i},\alpha }(\mathbf{x}_{i,k})\bigg)}_{I_{3}},
\end{eqnarray*}%
where the third line follows from the consistency constraints in %
\eqref{eqn:linear_const}. In the last equality, the summation term in $I_{1}$
corresponding to $|\alpha |=0$\ vanishes since $\sum_{k=1}^{K}w_{k}=\Delta
_{g}1=0$, and the term in $I_{3}$ corresponding to $k=1$ vanishes since $p_{%
\mathbf{x}_{i},\alpha }(\mathbf{x}_{i,1})=p_{\mathbf{x}_{i},\alpha }(\mathbf{%
x}_{i})=0$.

For the first error term $I_{1}$, using the Lemma~\ref{lem:Lap} in Appendix~\ref{app:A}, we
arrive at
\begin{eqnarray*}
\left\vert I_{1}\right\vert &=&\Big\vert\sum_{k=1}^{K}w_{k}\sum_{1\leq
|\alpha |\leq 2}b_{\alpha }p_{\mathbf{x}_{i},\alpha }(\mathbf{x}%
_{i,k})-\Delta _{M}\sum_{1\leq |\alpha |\leq 2}b_{\alpha }p_{\mathbf{x}%
_{i},\alpha }(\mathbf{x}_{i})\Big\vert \\
&\leq &\sum_{1\leq |\alpha |\leq 2}\Big\vert b_{\alpha }\Big\vert\Big\vert%
\sum_{k=1}^{K}w_{k}p_{\mathbf{x}_{i},\alpha }(\mathbf{x}_{i,k})-\Delta
_{M}p_{\mathbf{x}_{i},\alpha }(\mathbf{x}_{i})\Big\vert=O(N^{-\frac{l-1}{d}%
}),
\end{eqnarray*}%
where the last equality follows from the GMLS estimator in \eqref{errorrate}%
. For the second error term $I_{2}$, using the similar technique, we arrive
at%
\begin{eqnarray*}
\left\vert I_{2}\right\vert &\leq &\sum_{3\leq |\alpha |\leq l}\Big\vert %
b_{\alpha }\Big\vert\Big\vert\sum_{k=1}^{K}w_{k}p_{\mathbf{x}_{i},\alpha }(%
\mathbf{x}_{i,k})\Big\vert \\
&=&\sum_{3\leq |\alpha |\leq l}\Big\vert b_{\alpha }\Big\vert\Big\vert%
\sum_{k=1}^{K}w_{k}p_{\mathbf{x}_{i},\alpha }(\mathbf{x}_{i,k})-\Delta
_{M}p_{\mathbf{x}_{i},\alpha }(\mathbf{x}_{i})\Big\vert=O(N^{-\frac{l-1}{d}%
}),
\end{eqnarray*}%
where the second equality follows from the fact that $\Delta _{M}p_{\mathbf{x%
}_{i},\alpha }(\mathbf{x}_{i})=0$ for all $|\alpha |\geq 3$.

For the third error term $I_{3}$, we define $R_{\max }=\max_{\mathbf{y}\in B(%
\mathbf{x}_{i},C_{2}h_{0}),|\alpha |=l+1}\left\vert r_{\mathbf{y},\alpha
}\right\vert $ and $h_{K,\max }=\max_{k\in \{1,\ldots ,K\}}\left\Vert
\mathbf{x}_{i,k-}\mathbf{x}_{i}\right\Vert $. Notice that $\hat{w}_{k}\geq 0$
for $k=2,\ldots ,K,$ then the error term $I_{3}$ can be bounded by%
\begin{eqnarray}
\left\vert I_{3}\right\vert &\leq &R_{\max }\sum_{k=2}^{K}\sum_{|\alpha
|=l+1}\hat{w}_{k}|p_{\mathbf{x}_{i},\alpha }(\mathbf{x}_{i,k})|\leq R_{\max
}\sum_{k=2}^{K}\sum_{|\alpha |=l+1}\hat{w}_{k}\sum_{j=1}^{d}C_{0}|z_{\mathbf{%
x}_{i,k}}^{j}|^{l+1}  \notag \\
&\leq &C_{0}R_{\max }h_{K,\max }^{l-1}\left(
\begin{array}{c}
l+d \\
d-1%
\end{array}%
\right) \sum_{j=1}^{d}\sum_{k=2}^{K}\hat{w}_{k}(z_{\mathbf{x}%
_{i,k}}^{j})^{2},  \label{eqn:I1_v2}
\end{eqnarray}%
where $\left(
\begin{array}{c}
l+d \\
d-1%
\end{array}%
\right) $\ is the number of indices for $|\alpha |=l+1$. For the second
inequality above, we have used arithmetic geometric average inequality
(Young inequality), that is, there exists a constant $C_{0}$ such that $%
\prod_{j=1}^{d}(z_{\mathbf{x}_{i,k}}^{j})^{\alpha _{j}}\leq
C_{0}\sum_{j=1}^{d}|z_{\mathbf{x}_{i,k}}^{j}|^{l+1}$ for all $\left\vert
\alpha \right\vert =\sum_{j=1}^{d}\left\vert \alpha _{j}\right\vert =l+1$.
For the last inequality, we have used the fact that $|z_{\mathbf{x}_{i,k}}^{j}| = |\mathbf{t}_{{\mathbf{x}_{i}},j}\cdot (\mathbf{x}_{i,k}-{\mathbf{x}_{i}})| \leq \Vert \mathbf{x}_{i,k} - \mathbf{x}_{i} \Vert \leq h_{K,\max}$. Using again the consistency constraints in \eqref{eqn:linear_const}, we
obtain for each $j\in \{1,\ldots ,d\}$ that%
\begin{equation*}
\sum_{k=2}^{K}\hat{w}_{k}(z_{\mathbf{x}_{i,k}}^{j})^{2}=\sum_{k=1}^{K}\hat{w}%
_{k}(z_{\mathbf{x}_{i,k}}^{j})^{2}=\sum_{k=1}^{K}w_{k}(z_{\mathbf{x}%
_{i,k}}^{j})^{2}=[\Delta _{M}(z^{j})^{2}]|_{\mathbf{x}_{i}}+O(h_{X,M}^{l-1}),
\end{equation*}%
where $z^{j}=\mathbf{t}_{{\mathbf{x}_{i}},j}\cdot (\mathbf{x}-{\mathbf{x}_{i}%
})$ and the last equality follows from the GMLS estimator \eqref{errorrate}.
Substituting above into \eqref{eqn:I1_v2}, we arrive at the upper bound for $%
I_{3},$%
\begin{equation*}
\left\vert I_{3}\right\vert \leq C_{0}R_{\max }h_{K,\max }^{l-1}\bigg(%
\begin{array}{c}
l+d \\
d-1%
\end{array}%
\bigg)\bigg(\sum_{j=1}^{d}|\Delta _{M}(z^{j})^{2}(\mathbf{x}%
_{i})|+C_{2}h_{X,M}^{l-1}\bigg)=O(h_{K,\max }^{l-1}).
\end{equation*}%
Using the fact that $h_{K,\max }=O(N^{-\frac{1}{d}})$ with probability
higher than $1-\frac{1}{N}$ as shown in Lemma~\ref{lem:hK} in Appendix~\ref{app:A}, we obtain
that $\left\vert I_{3}\right\vert =O(N^{-\frac{l-1}{d}})$. Using the probability
$h_{X,M}=O(N^{-\frac{1}{d}})$ and $h_{K,\max}=O(N^{-\frac{1}{d}})$, we obtain the consistency
rate of $O(N^{-\frac{l-1}{d}})$ with probability higher than $1-\frac{2}{N}$.

Moreover, since $C=0$, $\hat{L}_{X}$ is diagonally dominant, and thus, $\hat{%
L}_{X,A}$ is strictly diagonally dominant. Using the Ahlberg-Nilson-Varah
bound \cite{ahlberg1963,varah1975}, we obtain $\Vert \hat{L}_{X,A}^{-1}\Vert _{\infty }\leq
1/\min_{i}a(\mathbf{x}_{i})$\ and the proof is complete.



\begin{rem}
Since we empirically found that $C=0$ for $l=2,3$ whereas $C \neq 0$ for $l=4,5$, the diagonal dominance assumption (or equivalently $C=0$) is  valid only for $l=2,3$ in practice whereas becomes invalid for $l=4,5$. Notice that this assumption is important not only for the stability of Laplacian matrix but also for the consistency proof  as seen from \eqref{eqn:I1_v2} in which $\hat{w}_k \geq 0$ is needed for $k=2,\ldots,K$.  If $C \neq 0$ when $l\geq4$, one cannot directly follow our consistency proof since non-diagonal weights ($\hat{w}_k$ for $k\geq 2$)can be either positive or negative. We leave the proof of consistency and stability when $l \geq 4$ for future investigation.
\end{rem}

\section{Semi-torus example for known manifolds and given boundary points } \label{app:B}

In this section, we assume that we know the parameterization (or tangent spaces and projection matrices $\mathbf{P}$) of the semi-torus manifold. We also assume that we are given $\sqrt{N}$ number of random boundary points on each of the exact boundary of the semi-torus (see Fig.~\ref{figB_semit}). Except the above two assumptions, all the other setups are the same with those in Section~\ref{sec5.3}. In Fig.~\ref{figB_semit}, we numerically verify the stability, consistency, and convergence in panels (b), (c), and (d), respectively.
One can see that the consistency rates are $O(N^{-\frac{l-1}{d}})$ as predicted by Lemma~\ref{consistency}. Again, one can observe the  super-convergence for randomly sampled data case on this 2D semi-torus example when boundary points are given.

\begin{figure*}[htbp]
{\scriptsize \centering
\begin{tabular}{cc}
{\normalsize (a) sampled points given boundaries} & {\normalsize (b) stability}
\\
\includegraphics[width=3
in, height=2.2 in]{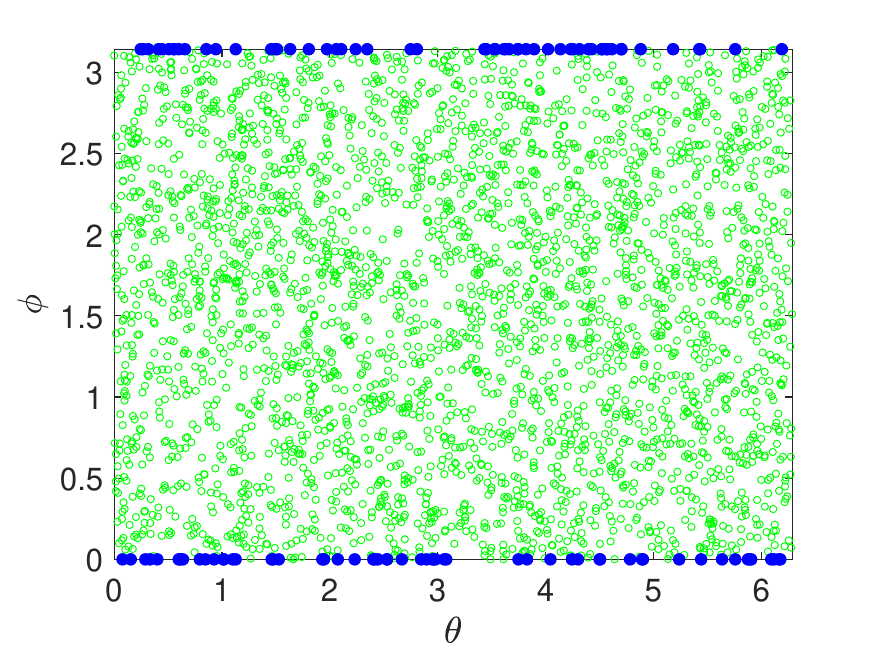} &
\includegraphics[width=3
in, height=2.2 in]{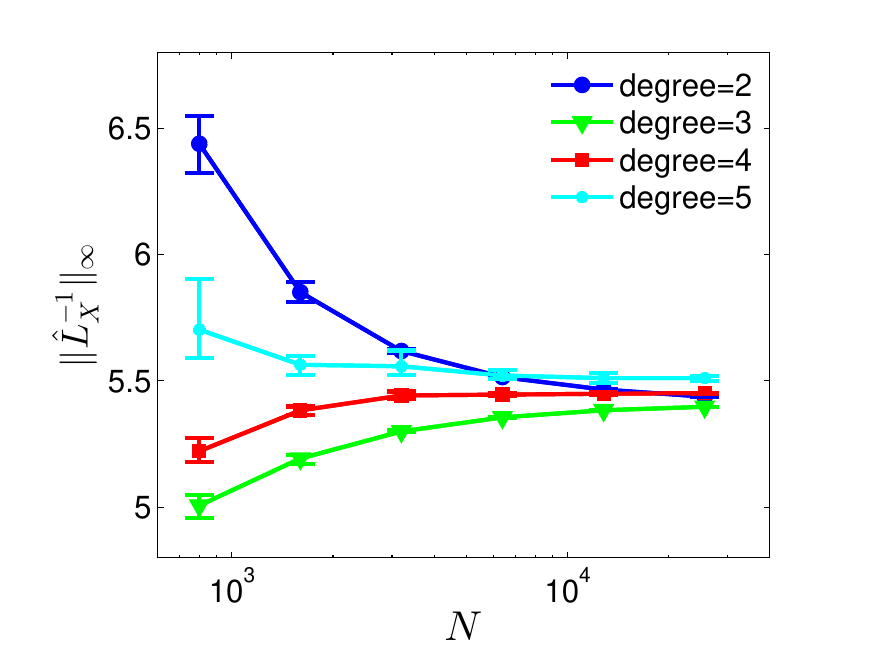}\\
{\normalsize (c) consistency} & {\normalsize (d) convergence}
\\
\includegraphics[width=3
in, height=2.2 in]{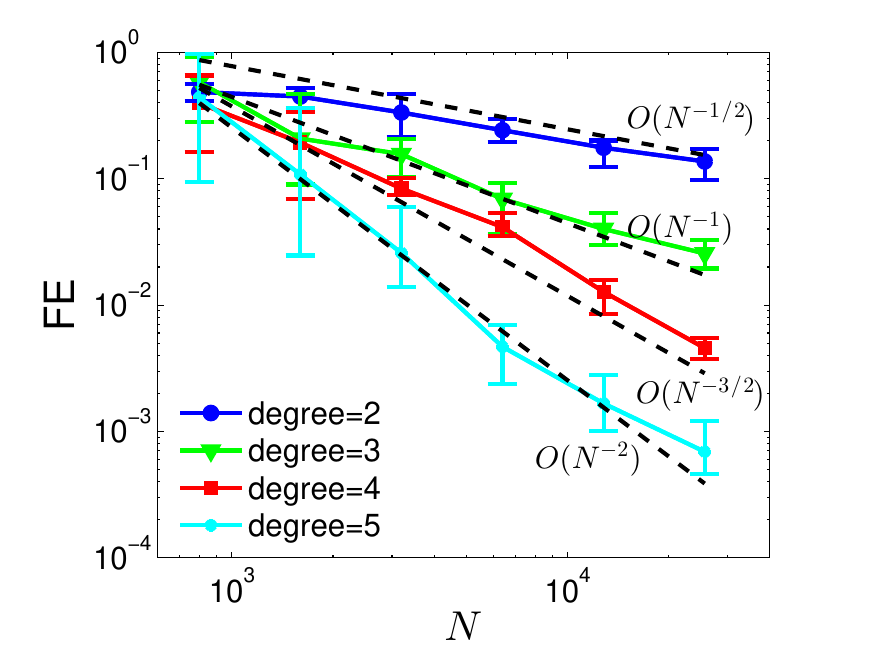} &
\includegraphics[width=3
in, height=2.2 in]{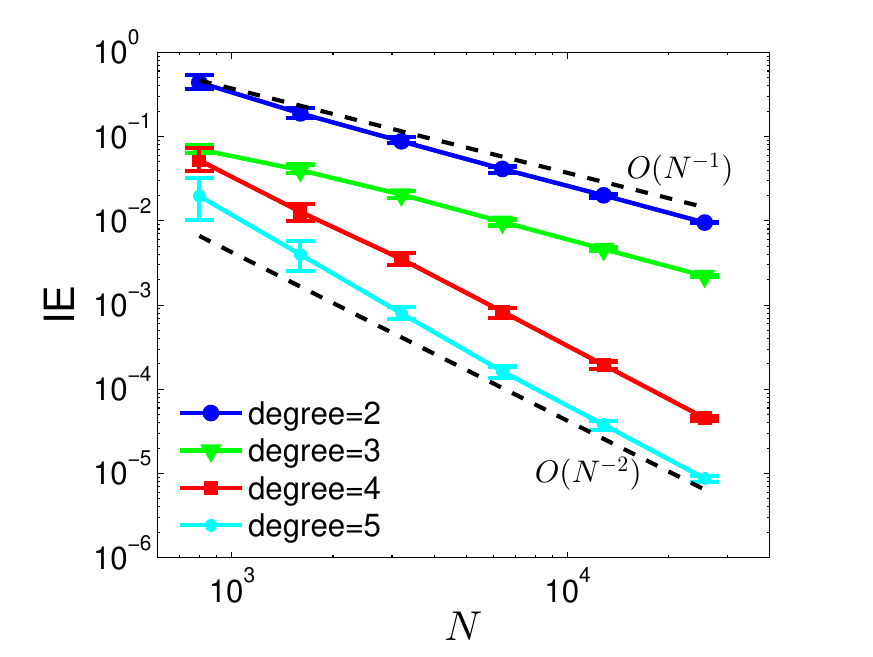}%
\end{tabular}
}
\caption{\textbf{2D semi-torus in $\mathbb{R}^3$}. $K=51$ nearest
neighbors are used. Comparison of
inverse errors (IE) of solutions between (a) with linear programming and (b) without optimization (GMLS). }
\label{figB_semit}
\end{figure*}

\section{RBF-FD and VBDM methods for diagnostic} \label{app:C}

\subsection{RBF-FD approach}

For diagnostic of numerical examples in Sections~\ref{sec5.1}-\ref{sec5.2}, we compare the proposed method to the following local approach, the RBF-generated finite difference type PDE solver:

\textbf{RBF-FD}: In \cite{shankar2015radial}, the authors proposed  the Radial Basis Functions (RBF)-generated Finite Differences (FD) scheme for approximating operators on manifolds. Instead of using the operator $\mathcal{I}_{\mathbb{P}}$ defined in Section \ref{Sec:int_LS},  we consider the  RBF interpolation operator
$$
I_\phi\textbf{f}(\textbf{x})=\sum_{k=1}^Kc_k\phi(\Vert\textbf{x}-\underline{\textbf{x}_k}\Vert),
$$
where $\underline{\textbf{x}_k}$ denotes {\color{black}the $k$th-nearest point to $\textbf{x}$.} Here, the weights $\{w_k\}_{k=1}^K$ for approximating Laplacian operator in (\ref{eqn:FD}) could be found by following the steps from (\ref{eqn:inter}) to (\ref{eqn:approximation}). In our numerical experiments, we implement the RBF-FD with the following Mat\'ern kernel,
$$
\phi_{\frac{n+3}{2}}(r)=(sr)^{3/2}\frac{e^{-sr}}{\sqrt{sr}}\left(1+\frac{1}{sr}\right)=(1+sr)e^{-sr}.
$$
where $s>0$ is the shape parameter that needs to be tuned. We should point out that the comparison will be reported for closed manifolds only. For manifolds with boundary, we do not report RBF-FD since we cannot find any meaningful convergence or accurate approximation, and we suspect that this issue is predominantly due to the randomly distributed data points. In Fig.~\ref{fig_general_torus}, we choose the shape parameter $s=0.5$. In Fig.~\ref{bunny_soln}(c), we show the result of RBF-FD where we choose the shape parameter $s=0.2$.

\subsection{VBDM approach}

For diagnostic of numerical examples in Sections~\ref{sec5.3}-\ref{sec5.4}, we compare the proposed method to the following Graph-Laplacian-based PDE solver:

\par \textbf{Variable Bandwidth Diffusion Map (VBDM):} In \cite{gh2019,harlim2020kernel,jiang2020ghost},  they consider a kernel approximation to the differential operator $\mathcal{L}^\kappa :=-\text{div}_g(\kappa\textup{grad}_g \,)$ with a fixed-bandwidth Gaussian kernel. For more accurate estimation when the data is non-uniformly distributed, we extend this graph-Laplacian approximation using the variable bandwidth kernel \cite{bh:16vb}.

Following the pointwise estimation method in \cite{liang2021solving}, we construct our estimator by choosing the kernel bandwidth to be inversely proportional to the sampling density $q$. Since the sampling density is usually unknown, we
first need to approximate it. While there are many ways to estimate density, in our algorithm we employ kernel density estimation with the following kernel that is closely related to the cKNN \cite{berry2019consistent} and self-tuning kernel \cite{zelnik2004self},
\[
K_{\epsilon ,0 }(\mathbf{x},\mathbf{y})=\exp \left( -\frac{\left\Vert
\mathbf{x}-\mathbf{y}\right\Vert^2 }{2\epsilon \rho_0(\mathbf{x})  \rho_0(\mathbf{y}) }\right),
\]
where $\rho_0(\mathbf{x}) := \Big(\frac{1}{k_2-1}\sum_{j=2}^{k_2} \|\mathbf{x} - \mathbf{x}_j \|^2\Big)^{1/2}$
denotes the average distance of $\mathbf{x}$ to the first $k_2$-nearest neighbors $\{\mathbf{x}_j\}, j=2,\ldots k_2$ excluding itself. Using this kernel, the sampling
density $q(\mathbf{x})$\ is estimated by $Q (\mathbf{x})=\sum_{j=1}^{N}K_{\epsilon ,0
}(\mathbf{x},\mathbf{x}_{j})/\rho_0( \mathbf{x})^{d}$\ at given point
cloud data. For our purpose, let us give a quick overview of the estimation of the Laplace-Beltrami operator. In this case, one chooses the bandwidth function to be $\rho(\mathbf{x})= q(\mathbf{x})^{\beta} \approx Q(\mathbf{x})^{\beta}$, with $\beta =-1/2$. With this bandwidth function, we employ the DM algebraic steps to construct  $K_{\epsilon,\rho}= \exp\left(-\Vert\mathbf{x}-\mathbf{y}\Vert^2/(4\epsilon\rho(\mathbf{x})\rho(\mathbf{y}))\right) $ and define $Q_\rho (\mathbf{x})=\sum_{j=1}^{N}K_{\epsilon ,\rho
}(\mathbf{x},\mathbf{x}_{j})/\rho( \mathbf{x})^{d}$. Then, we remove the sampling bias by applying a right
normalization $K_{\epsilon ,\rho ,\alpha }(\mathbf{x}_{i},\mathbf{x}%
_{j})= \frac{K_{\epsilon ,\rho }(\mathbf{x}_{i},\mathbf{x}_{j})}{Q_{\rho }(\mathbf{x}%
_{i})^{\alpha }Q_{\rho }(\mathbf{x}_{j})^{\alpha }}$,
where $\alpha = -d/4+1/2$
. We refer to \cite{harlim:18,bh:16vb} for more
details about the VBDM estimator of weighted Laplacian with other choices of $\alpha$ and $\beta$. Define diagonal
matrices $\mathbf{Q}$ and $\mathbf{P}$ with entries $\mathbf{Q}_{ii}=Q_{\rho
}(\mathbf{x}_{i})$ and $\mathbf{P}_{ii}=\rho (\mathbf{x}_{i})$,
respectively, and also define the symmetric matrix $\mathbf{K}$ with entries
$\mathbf{K}_{ij}=K_{\epsilon ,\rho ,\alpha }(\mathbf{x}_{i},\mathbf{x}_{j})$%
. Next, one can obtain the variable bandwidth diffusion map (VBDM)
estimator, $\mathbf{L}_{\epsilon ,\rho }:=\mathbf{P}^{-2}(\mathbf{D}^{-1}%
\mathbf{K-I})/\epsilon $, where $\mathbf{I}$\ is an identity matrix, as a discrete estimator to the Laplace-Beltrami operator in high probability. For computational efficiency, we also set $k_1$ as the nearest neighbor parameter for constructing $\mathbf{K}_{ij}$ (or eventually $\mathbf{L}_{\epsilon ,\rho }$).

We now report the specific choices for $k_1, k_2$ in each example. For VBDM in Fig~\ref{fig_semi_torus}, we choose $k_2=[15,20,28,40,55,70]$ for density estimation and $k_1=[30,40,56,80,110,140]$ nearest neighbors to construct the estimator $\mathbf{L}_{\epsilon,\rho}$.
In Fig~\ref{face_soln}, we show the  result of VBDM where we choose $k_2=64$ for density estimation and $k_1=128$ nearest neighbors to construct the estimator $\mathbf{L}_{\epsilon,\rho}$. As for the choice of parameter $\epsilon$, we simply use the automated tuning technique that is found to be robust for variable bandwidth kernels (see Section~5 of \cite{bh:16vb}).


\end{document}